\documentclass[a4paper,twoside,10pt]{article}
\usepackage[left=3.0cm,right=3.0cm,top=3.5cm,bottom=3.5cm]{geometry}
\usepackage{mathrsfs}
\usepackage{mathtools}
\usepackage{graphicx}
\usepackage{epstopdf} 
\usepackage{stmaryrd}
\usepackage{float}
\usepackage{bm}
\usepackage{amsfonts,amssymb,amsthm}
\usepackage{hyperref}
\usepackage{color}
\usepackage{scalefnt}
\usepackage{tikz}
\usetikzlibrary{positioning}
\usetikzlibrary{calc}
\pgfdeclarelayer{background}
\pgfdeclarelayer{foreground}
\pgfsetlayers{background,main,foreground}

\newtheorem{theorem}{Theorem}[section]
\newtheorem{lemma}[theorem]{Lemma}
\newtheorem{proposition}[theorem]{Proposition}

\newtheorem{remark}[theorem]{Remark}

\newtheorem{assumption}[theorem]{Assumption}
\numberwithin{equation}{section}
\counterwithin{figure}{section}
\counterwithin{table}{section}

\usepackage{color}

\newcommand{\Norm}[1]{{\left\|{#1} \right\|}}
\newcommand{\NNorm}[1]{{\left\vert\kern-0.25ex\left\vert\kern-0.25ex\left\vert #1     \right\vert\kern-0.25ex\right\vert\kern-0.25ex\right\vert}}
\newcommand{\SemiNorm}[1]{{\left|{#1} \right|}}
\newcommand{\jump}[1]{\left[\!\left[#1\right]\!\right]}

\newcommand{\bbR}{\mathbb{R}}

\newcommand{\calT}{\mathcal{T}}
\newcommand{\calE}{\mathcal{E}}
\newcommand{\calM}{\mathcal{M}}

\newcommand{\e}{e}
\newcommand{\Rbb}{\mathbb R}
\newcommand{\Nbb}{\mathbb N}
\newcommand{\Pbb}{\mathbb P}
\newcommand{\tauh}{\calT_h}
\newcommand{\taun}{\tauh}
\newcommand{\Ecalh}{\calE_h}
\newcommand{\EcalhI}{\calE^I_h}
\newcommand{\EcalhB}{\calE^B_h}
\newcommand{\EcalEc}{\EcalE_c}
\newcommand{\EcalEs}{\EcalE_s}
\newcommand{\E}{K}
\newcommand{\hE}{h_\E}
\newcommand{\he}{h_\e}
\newcommand{\gammae}{\gamma_\e}
\newcommand{\Ie}{I_\e}
\newcommand{\elle}{\ell_\e}
\newcommand{\boldalpha}{\boldsymbol \alpha}
\newcommand{\nbf}{\mathbf n}
\newcommand{\nbfE}{\nbf_\E}
\newcommand{\nbfEp}{\nbf_{\E^+}}
\newcommand{\nbfEm}{\nbf_{\E^-}}
\newcommand{\nbfe}{\nbf_\e}
\newcommand{\Vh}{V_h}
\newcommand{\VhE}{\Vh(\E)}
\newcommand{\uh}{u_h}
\newcommand{\vh}{v_h}
\newcommand{\wh}{w_h}
\newcommand{\EcalE}{\mathcal E^\E}
\newcommand{\Ecalhc}{\mathcal E_{h,c}}
\newcommand{\Ecalhs}{\mathcal E_{h,s}}
\newcommand{\qtilde}{\widetilde q}
\newcommand{\qtildekmo}{\qtilde_{k-1}}
\newcommand{\Pitilde}{\widetilde \Pi}
\newcommand{\qn}{q_n}
\newcommand{\qk}{q_k}
\newcommand{\Pitildeze}{\Pitilde_{k-1}^{0,\e}}

\newcommand{\PizEkmt}{\Pi^{0,\E}_{k-2}}
\newcommand{\Pitildenabla}{\Pitilde_k^{\nabla,\E}}
\newcommand{\Pinabla}{\Pi_{k}^{\nabla,\E}}

\newcommand{\Pbbtilde}{\widetilde{\Pbb}}
\newcommand{\ah}{a_h}
\newcommand{\aE}{a^\E}
\newcommand{\ahE}{\ah^\E}
\newcommand{\SE}{S^\E}
\newcommand{\abf}{\mathbf a}
\newcommand{\bbf}{\mathbf b}
\newcommand{\upi}{u_\pi}
\newcommand{\vpi}{v_{\pi}}
\newcommand{\vIE}{v_I^\E}
\newcommand{\vpE}{v_p^\E}
\newcommand{\phiI}{\phi_I}
\newcommand{\vI}{v_I}

\newcommand{\nbfhat}{\widehat{\nbf}}
\newcommand{\ehat}{\widehat \e}
\newcommand{\nbfhatehat}{\nbfhat_{\ehat}}
\newcommand{\fh}{f_h}
\newcommand{\Ncalh}{\mathcal N_h}
\newcommand{\uI}{u_I}
\newcommand{\cpartialE}{c^{\partial\E}}
\newcommand{\s}{s}

\newcommand{\ord}{\eta}

\title{\large{The nonconforming virtual element method with curved edges}}
\author{\normalsize{
Lourenco Beir\~ao Da Veiga\thanks{Dipartimento di Matematica e Applicazioni, Università degli Studi di Milano-Bicocca, 20125 Milano, Italy,
\tt{lourenco.beirao@unimib.it, yi.liu@unimib.it, lorenzo.mascotto@unimib.it, alessandro.russo@unimib.it}}
\thanks{IMATI-CNR, 27100, Pavia, Italy},
Yi Liu\footnotemark[1]~\thanks{School of Mathematical Sciences, Nanjing Normal University, Nanjing 210023, China, \tt{200901005@njnu.edu.cn}},
Lorenzo Mascotto\footnotemark[1]~\footnotemark[2]~\thanks{Fakult\"at f\"ur Mathematik, Universit\"at Wien, 1090 Vienna, Austria},
Alessandro Russo\footnotemark[1]~\footnotemark[2]}}

\date{}

\begin{document}
\maketitle

\begin{abstract}
\noindent We introduce a nonconforming virtual element method for the Poisson equation
on domains with curved boundary and internal interfaces.
We prove arbitrary order optimal convergence
in the energy and~$L^2$ norms,
and validate the theoretical results with numerical experiments.
Compared to existing nodal virtual elements on curved domains,
the proposed scheme has the advantage that it can be designed in any dimension.


\end{abstract}

\section{Introduction} \label{section:introduction}

Partial differential equations are often posed on domains with curved boundaries and internal interfaces.
The geometric error between a curved interface/boundary
and a corresponding ``flat'' approximation
affects the accuracy of the standard finite element method,
leading to a loss of convergence for higher-order elements~\cite{Strang-Berger:1973,Thomee:1973}.
This phenomenon has been addressed in many different ways in the literature, 
the most classical being to employ isoparametric finite elements~\cite{Ergatoudis-Irons-Zienkiewicz:1968, Lenoir:1986},
which require a polynomial approximation of the curved boundary and a careful choice of the isoparametric nodes;
another notable approach, which applies to CAD domains, is that of the Isogeometric Analysis~\cite{Cottrell-Hughes-Bazilevs:2009}.

Both issues can be avoided employing curved virtual elements~\cite{BeiraodaVeiga-Russo-Vacca:2019}.
The virtual element method~\cite{BeiradaVeiga-Brezzi-Cangiani-Manzini-Marini-Russo:2013, BeiradaVeiga-Brezzi-Marini-Russo:2014}
was designed a decade ago as a generalization of the finite element method to a Galerkin method based on polytopal meshes.
Basis functions are defined as solutions to local partial differential problems with polynomial data.
An explicit representation of the basis functions is not required;
rather, the scheme is designed only based on a suitable choice of the degrees of freedom.

In~\cite{BeiraodaVeiga-Russo-Vacca:2019},
test and trial virtual element functions are defined (in 2D)
so as their restrictions on curved edges
are mapped polynomials.
Other variants were developed later.
In~\cite{BeiraodaVeiga-Brezzi-Marini-Russo:2020},
again focusing on the 2D case only,
virtual element functions over curved edges
are restrictions of polynomials.
In~\cite{Bertoluzza-Pennacchio-Prada:2019}, a boundary correction technique
tracing back to the pioneering work~\cite{Bramble-Dupont-Thomee:1972}
was generalized to the virtual element setting;
here, normal-directional Taylor expansions are used to correct function values on the boundary.
The gospel of~\cite{BeiraodaVeiga-Russo-Vacca:2019}
has been applied to the approximation of solutions to the wave equation in~\cite{Dassi-Fumagalli-Mazzieri-Scotti-Vacca}.
Mixed virtual elements on curved domains are analyzed in two and three dimensions in~\cite{Dassi-Fumagalli-Losapio-Scialo-Scotti-Vacca, Dassi-Fumagalli-Scotti-Vacca}.

Other polytopal element method have been designed for curved domains.
Amongst them, we recall
the extended hybridizable discontinuous Galerkin method~\cite{Gurkan-SalaLardies-Kronbichler-FernandezMandez:2016};
the unfitted hybrid high-order method~\cite{Burman-Ern:2019,Burman-Cicuttin-Delay-Ern:2021};
the hybrid high-order method for the Poisson~\cite{Botti-DiPietro:2018, Yemm:2022}
and (singularly perturbed) fourth order problems~\cite{Dong-Ern:2022};
the Trefftz-based finite element method~\cite{Anand-Ovall-Reynolds-Weisser:2020}.

\medskip

In this paper, we focus on the Poisson problem,
design a nonconforming virtual element method on curved domains and with internal curved interfaces,
prove arbitrary order optimal convergence estimates in the energy and $L^2$ norms,
and validate the theoretical results with numerical results.

The proposed scheme is a generalization of the standard nonconforming virtual element method to the case of curved boundaries and internal interfaces.
Indeed, when the curved boundaries happen to be straight, 
the proposed virtual space boils down to that in~\cite{AyusodeDios-Lipnikov-Manzini:2016}.
Compared to its conforming nodal version~\cite{BeiraodaVeiga-Russo-Vacca:2019},
in principle the nonconforming scheme can be designed and analyzed in any dimension at once.

The proposed method is based on the computation of a novel Ritz-Galerkin operator that,
due to computability reasons,
is different from the standard $H^1$ projection operator
typically employed in virtual elements.
A noteworthy challenge of the forthcoming analysis
resides in developing optimal approximation estimates for such a Ritz-Galerkin operator;
this result is interesting per se and could be borrowed
also by other methods handling curved boundaries/interfaces.

Even though the standard nonconforming virtual element
can be algebraically equivalent to the hybrid high-order method~\cite{Cockburn-DiPietro-Ern:2016}
for a particular choice of the stabilization,
the method presented in this paper differs from the hybrid-high order methods on curved domains available in the literature.

\paragraph*{Preliminary notation.}
We denote the usual Sobolev space
of order~$m$, $m\geq 0$,
on an open bounded Lipschitz domain~$D$ in~$\Rbb^d$, $d \in \Nbb$, by $H^m(D)$.
We endow it with the inner-product $(\cdot,\cdot)_{m,D}$,
the norm $\|\cdot\|_{m,D}$, and the seminorm $|\cdot|_{m,D}$.
Let $H_0^1(D)$ the subspace of~$H^1(D)$ of functions with zero trace over the boundary~$\partial D$ of~$D$.
When $m=0$, the space $H^0(D)$ is the space $L^2(D)$
of square integrable functions over~$D$.
In this case, $(\cdot,\cdot)_{0,D} = (\cdot,\cdot)_D$
denotes the standard $L^2$ inner-product.
Sobolev spaces of negative order can be defined by duality.
The notation $\langle\cdot,\cdot\rangle$ stands for the duality pairing~$H^{-\frac12} - H^{\frac12}$ on a given domain.

Further, we introduce the Sobolev spaces~$W^{m,\infty}(D)$, $m\in \Nbb$,
of functions having weak derivatives up to order~$m$, which are bounded almost everywhere in~$D$.
The case~$m=0$ coincides with the usual space~$L^\infty(D)$;
the spaces of noninteger order~$m \ge 0$, $m \not\in \Nbb$, are constructed, e.g., by interpolation theory.
The corresponding norm and seminorm are~$\Norm{\cdot}_{W^{m,\infty}(D)}$
and~$\SemiNorm{\cdot}_{W^{m,\infty}(D)}$.

\paragraph*{Model problem.}
Let $\Omega\subset\Rbb^d$, $d=2,3$,
be a Lipschitz domain with (possibly) curved boundary~$\partial\Omega$.
Given $f$ in~$L^2(\Omega)$ and~$g$ in~$H^{\frac12}(\partial 
\Omega)$,
we consider the following Poisson problem:
Find $u$ such that
\begin{equation}\label{model}
\begin{cases}
-\Delta u = f   &\text{in } \Omega,\\
u = g           &\text{on } \partial\Omega.
\end{cases}
\end{equation}
Introduce~$V_g=\{u\in H^1(\Omega):u_{|\partial\Omega} = g\}$,
$V:=H^1_0(\Omega)$, and the bilinear form
\[
a(u,v) :
= \int_\Omega\nabla u\cdot\nabla v\ \mathrm{d}\Omega 
\qquad \forall \ u,v\in H^1(\Omega).
\]
A variational formulation of~\eqref{model} reads as follows:
\begin{equation}\label{weak}
\begin{cases}
\text{Find }u \in V_g \text{ such that}\\
a(u,v)= (f,v)_{0,\Omega} \qquad \forall v\in V.
\end{cases}
\end{equation}
We assume that the boundary~$\partial \Omega$ is the union
a finite number of smooth curved edges/faces
$\{\Gamma_i\}_{i=1,\cdots,N}$, i.e.,
\[
\bigcup_{i=1}^N\Gamma_i = \partial\Omega.
\]
Each $\Gamma_i$
is of class $\mathcal{C}^{\ord}$, for an integer $\ord \geq 1$,
which will fixed in Assumption~\ref{assumption:eta} below:
if $d=2$,
there exists a given regular and invertible $\mathcal{C}^{\ord}$-parametrization
$\gamma_i:I_i\rightarrow\Gamma_i$ for $i = 1, \dots , N$, 
where $I_i:= [a_i, b_i]\subset \bbR$ is a closed interval;
if $d=3$, there exists a given regular and invertible $\mathcal{C}^{\ord}$-parametrization
$\gamma_i:F_i\rightarrow\Gamma_i$ for $i = 1, \dots , N$,
where~$F_i$ is a straight polygon.
The smoothness parameter~$\ord$ depends on the order of the numerical scheme
and will be specified later.

Since all the~$\Gamma_i$ can be treated analogously in the forthcoming analysis,
 we drop the index~$i$
and assume that~$\partial \Omega$ contains only one curved face~$\Gamma$.
To further simplify the presentation,
we focus on the two dimensional case
and postpone the discussion of the three dimensional case to Section~\ref{section:3D} below.
We further assume that $\gamma:[0,1] \to \Gamma$.

\begin{remark} \label{remark:internal-interfaces}
The forthcoming analysis can be extended to
the case of internal interfaces (and jumping coefficients) with minor modifications.
To simplify the presentation, we stick to the case of~$\Gamma$ being a curved boundary face;
however, we shall present numerical experiments for jumping coefficients
across internal curved interfaces.
\end{remark}

\paragraph*{Structure of the paper.}
In Section~\ref{section:meshes},
we introduce regular polygonal meshes,
and broken and nonconforming polynomial and Sobolev spaces.
In Section~\ref{section:VEM}, we design a novel nonconforming virtual element method for curved domain;
show its well posedness;
discuss its lack of polynomial consistency.
In Sections~\ref{section:stability} and~\ref{section:interpolation},
we prove stability and interpolation properties of the new virtual element functions.
Section~\ref{section:convergence} is devoted to the proof of the rate of convergence in the $H^1$ and $L^2$ norms.
Details on the 3D version of the method are discussed in Section~\ref{section:3D}.
In Section~\ref{section:numerical-esperiments}, we present numerical experiments that verify the theory established the previous sections.

\section{Meshes and broken spaces} \label{section:meshes}
In this section, we introduce regular polygonal meshes
with whom we associate broken polynomial and Sobolev spaces,
and nonconforming polynomial spaces.

\subsection{Mesh assumptions} \label{subsection:meshes}
Let~$\{\taun\}$ be a sequence of partitions of~$\Omega$ into polygons,
possibly with curved edges along the curved boundary.
We denote the diameter of each element~$\E$ by~$\hE$
and the mesh size function of~$\taun$ by~$h := \max_{\E \in \taun}\hE$.
Let~$\Ecalh$ be the set of edges in~$\taun$.
We denote the size of any edge~$\e$ by~$\he$,
where the size of a (possibly curved) edge is the distance between its two endpoints;
see Remark \ref{lengthofedges} below.
Let $\EcalhI$ and $\EcalhB$ be the sets of all interior and boundary edges in~$\taun$\;
with~$\Ecalhc$ and~$\Ecalhs$ denoting the sets
of curved and straight edges in~$\tauh$, respectively.

For each element $\E\in\taun$,
we denote the sets of its edges by~$\EcalE$,
which we split into straight~$\EcalEs$
and curved edges~$\EcalEc$, respectively.
We denote the set of elements containing at least one edge in~$\Ecalh^B$ by~$\taun^B$;
an element $\E$ in~$\taun^B$ may contain more than one edge in~$\Ecalh^B$.
With each element~$\E$, we associate the outward unit normal vector~$\nbfE$;
with each edge~$\e$, we associate a unit normal vector~$\nbfe$ out of the available two.

Henceforth, we demand the following regularity assumptions on the sequence~$\{\taun\}$:
there exists a positive constant~$\rho$ such that
\begin{itemize}
\item[\textbf{(G1)}] each element~$\E$ is star-shaped with respect to a ball of radius larger than or equal to~$\rho \hE$;
\item[\textbf{(G2)}] for each element~$\E$ and any of its (possibly curved) edges~$\e$,
$\he$ is larger than or equal to~$\rho \hE$.
\end{itemize}
Assumptions (\textbf{G1})--(\textbf{G2}) imply that
each element has a uniformly bounded number of edges.

We introduce a parametrization of the edges:
\begin{itemize}
\item for any straight edge~$\e$ with endpoints $\bm{x}_{\e}^1$
and $\bm{x}_{\e}^2$,
we introduce the parametrization $\gammae(t) =\dfrac{t}{\he} (\bm{x}_{\e}^2-\bm{x}_{\e}^1) + \bm{x}_{\e}^1$;
\item for any curved edge $\e$,
we introduce the parametrization $\gammae:\ \Ie \subset I :=[0,1] \rightarrow e$
as the restriction of the global parametrization
$\gamma:\ I \rightarrow \Gamma$ to the interval~$\Ie$.
\end{itemize}

\begin{remark}\label{lengthofedges}
Let the lenght of a curved edge $\elle =\int_{\Ie}\|\gamma'(t)\|\mathrm{d}t$.
Since $\gamma$ and~$\gamma^{-1}$ are fixed once and for all,
and are of class~$W^{1,\infty}$, the quantity $\he$ introduced above is comparable with~$\elle$.
Therefore in the following we shall simply refer to both quantities as ``length''.
\end{remark}

We shall write $x\lesssim y$ and $x\gtrsim y$ instead of $x\leq Cy$ and $x\geq Cy$, respectively,
for a positive constant $C$ independent of~$\taun$.
Moreover, $x\approx y$ stands for 
$x \lesssim b$ and~$b \lesssim a$ at once.
The involved constants will be written explicitly only when necessary.

The validity of (\textbf{G1})--(\textbf{G2})
guarantees that the constants in the forthcoming trace and inverse inequalities are uniformly bounded.

\subsection{Broken and nonconforming spaces} \label{subsection:broken-spaces}
Let~$\Pbb_n(\E)$, $n\in \Nbb$,
be the space of polynomials of maximum degree~$n$ over each element~$\E$;
we use the convention $\Pbb_{-1}(\E) = \{0\}$.
Given~$\bm{x}_{\E}$ the centroid of~$\E$,
we introduce a basis for the space~$\Pbb_n(\E)$ given by the set of scaled and shifted monomials
\[
\calM_n(\E) =
\left\{
\left(\dfrac{\bm{x}-\bm{x}_{\E}}{\hE}\right)^{\boldalpha}
\; \forall  \boldalpha \in \Nbb^2, \; \vert \boldalpha \vert \le n,\;\;
\forall \bm x \in \E \right\}.
\]
Here, $\bm{\alpha}$ denotes a multi-index~$\bm{\alpha}=(\alpha_1,\alpha_2)$.

Similarly, let~$\Pbb_n(\Ie)$, $n\in \Nbb$,
be the space of polynomials of maximum degree~$n$ over the interval~$\Ie$.
Given~${x}_{\Ie}$ and~$h_{\Ie}$ the midpoint and length of~$\Ie$,
we introduce a basis for the space~$\Pbb_n(\Ie)$ given by the set of scaled and shifted monomials
\[
\calM_n(\Ie) =
\left\{
\left(\dfrac{{x}-{x}_{\Ie}}{h_{\Ie}}\right)^{\alpha}
\; \forall  \alpha \in \Nbb, \;  \alpha \le n,\;\;
\forall x \in \Ie \right\}.
\]
Recalling that~$\gammae:\Ie\subset I \to \e$ denotes the parametrization of the edge~$\e$,
we consider the following mapped polynomial space and scaled monomial set:
\[
\begin{split}
\Pbbtilde_n(\e)      = \{\widetilde{q} = q\ \circ\ \gammae^{-1}:\ q\in {P}_n(\Ie)\},
\qquad
\widetilde{\calM}_n(\e)  = \{\widetilde{m} = m\ \circ\ \gammae^{-1}:\ m \in \calM_n(\Ie)\}.
\end{split}
\]
If~$\e$ is straight, then $\Pbbtilde_n(\e)$ and~$\widetilde{\calM}_n(\e)$
boil down to a standard polynomial space
and scaled monomial set, respectively.
For any~$s>0$, we introduce the broken Sobolev space over a mesh~$\taun$ as
\[
H^s(\taun) := \{ v\in L^2(\Omega):\, v_{|\E} \in H^s(\E)\quad \forall \E \in \taun \} ,
\]
and equip it with the broken norm and seminorm
\[
\Norm{v}_{s,h}^2 := 
\sum_{\E\in\taun}\Norm{v}_{s,\E}^2 ,
\qquad
\SemiNorm{v}_{s,h}^2 :=
\sum_{\E\in\taun}\SemiNorm{v}_{s,\E}^2 .
\]
We define the jump across the edge~$\e$ of any~$v$ in~$H^1(\taun)$ as
\[
\jump{v}:=
\begin{cases}
v_{|\E^+}\nbfEp +v_{|\E^-} \nbfEm   & \text{if~$\e\in\EcalhI$,  $\e \subset \partial\E^+ \cap \partial \E^-$ for given~$\E^+$, $\E^- \in \taun$} \\
v\nbfe                               & \text{if~$\e\in\EcalhB$, $\e \subset \E$ for a given~$\E \in \taun$.}
\end{cases}
\]
The nonconforming Sobolev space of order~$k$ over~$\taun$ is given as follows:
\[
H^{1,nc}(\taun,k) 
:= \left\{
v\in H^1(\taun) \middle|  
\int_{\e}\llbracket v \rrbracket\cdot\nbfe \ q{\mathrm{d}} s =0 ,\;
\forall q\in \Pbbtilde_{k-1}(\e),\ \forall e\in\Ecalh
\right\}.
\]
In the straight edges case,
this space coincides with the standard nonconforming Sobolev space in~\cite{AyusodeDios-Lipnikov-Manzini:2016}.

\section{The nonconforming virtual element method on curved polygons} \label{section:VEM}

In this section, we introduce the nonconforming virtual element method for the approximation of solutions to~\eqref{weak}.
In Section~\ref{subsection:ve-space},
we introduce the local and global virtual element spaces
and endow them with suitable sets of degrees of freedom (DoFs).
In Section~\ref{subsection:polynomial-projections-bf},
we discretize the bilinear form
by means of computable polynomial projectors and stabilizing bilinear forms.
With this at hand, we introduce the method in Section~\ref{subsection:VEM}.

\subsection{Nonconforming virtual element spaces}
\label{subsection:ve-space}
We define a local virtual element space of order $k$ in~$\Nbb$ on the (possibly curved) element~$\E$:
\begin{equation} \label{local-space}
\VhE
:= \left\{ \vh \in H^1(\E) 
\middle| \ \Delta \vh \in \Pbb_{k-2}(\E),\
\nbfE \cdot \nabla \vh \in \Pbbtilde_{k-1}(\e)\ 
\forall e \in \EcalE
\right\}.
\end{equation}
We have that $\Pbb_n(\E)_{|\e} = \Pbbtilde_n(\e)$ if $\e$ is a straight edge;
instead, if~$\e$ is a curved edge,
$\Pbb_0(\E)_{|\e} \subset\Pbbtilde_n(\e)$
but in general, $\Pbb_n(\E)_{|\e} \not\subset \Pbbtilde_n(\e)$.
This implies that on curved elements the space~$\VhE$ contains constant functions
but not the space~$\Pbb_k(\E)$.
\medskip

We can define the following sets of degrees of freedom
for the space~$\VhE$.

\begin{itemize}
\item on each edge~$\e$ of~$\E$, the moments
\begin{equation}\label{edge-dofs}
\bm{D}_{\e}^{i}(\vh) 
= |\e|^{-1} \int_\e \vh \widetilde{m}_i\ \mathrm{d}s\qquad
\forall \widetilde{m}_i \in \widetilde{\calM}_{k-1}(\e);
\end{equation}
\item the bulk moments
\begin{equation}\label{bulk-dofs}
\bm{D}_{\E}^{j}(\vh)
= |\E|^{-1} \int_\E \vh m_j\ \mathrm{d}K
\qquad  \forall m_j \in \calM_{k-2}(\E).
\end{equation}
\end{itemize}
In Figure~\ref{figure:dof}, we give a graphic representation for such linear functionals in the case $k=2$.

\begin{figure}[H]
\centering
{\scalefont{1}
\begin{tikzpicture}
\draw[very thick] (-1.95,0.9)--(-2.6,-1.0)--(-1.95,-2.6)--(1.3,-2.6)--(2,-0.9)--(1.5,0.6);
\draw[red,very thick] (1.5,0.6) arc (25:145:2);
\filldraw [black](-2.4,-0.43) circle [radius=4pt];
\filldraw [black](-2.2,0.2) circle [radius=4pt];
\filldraw [black](-2.38,-1.6) circle [radius=4pt];
\filldraw [black](-2.1,-2.2) circle [radius=4pt];
\filldraw [black](-0.9,-2.6) circle [radius=4pt];
\filldraw [black](0.4,-2.6) circle [radius=4pt];
\filldraw [black](1.5,-2.2) circle [radius=4pt];
\filldraw [black](1.8,-1.4) circle [radius=4pt];
\filldraw [black](1.85,-0.5) circle [radius=4pt];
\filldraw [black](1.65,0.1) circle [radius=4pt];
\filldraw [red](-0.9,1.65) circle [radius=4pt];
\filldraw [red](0.7,1.5) circle [radius=4pt];
\draw [orange,thick](-0.3,-0.5) circle [radius=4pt];
\end{tikzpicture}}
\caption{\small{Representation of the DoFs for $k=2$. The edge moments~\eqref{edge-dofs} are the black balls (for straight edges) and red balls (for curved edges);
the bulk moments~\eqref{bulk-dofs} are the internal orange circles.}}
\label{figure:dof}
\end{figure}
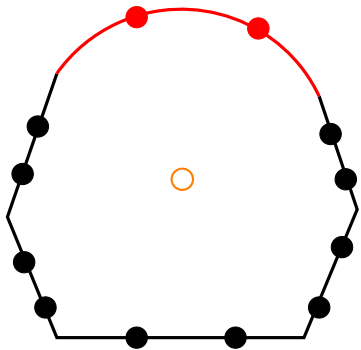

The following result generalizes~\cite[Lemma~$3.1$]{AyusodeDios-Lipnikov-Manzini:2016} to the case of curved elements.
\begin{lemma}
The sets of linear functionals~\eqref{edge-dofs} and~\eqref{bulk-dofs}
are a set of unisolvent DoFs for the space~$\VhE$.
\end{lemma}
\begin{proof}
The number of linear functionals in~\eqref{edge-dofs} and~\eqref{bulk-dofs} equals the dimension of the space~$\VhE$.
So, we only have to prove the unisolvence.
Let~$\vh$ in~$\VhE$ be such that
\begin{equation} \label{muvh}
\bm{D}_{\e}^{i}(\vh) = 0\quad \forall e\in\EcalE
\ \text{ and } \ \bm{D}_{\E}^{j}(\vh) = 0
\quad \forall i= 1,\cdots,k;\ j = 1,\cdots,k(k-1)/2.
\end{equation}
Then, it suffices to prove~$\vh = 0$.
To this aim, we integrate by parts, use~\eqref{muvh}, and obtain
\[
\int_\E|\nabla \vh|^2\ \mathrm{d} K
= \int_\E-\Delta \vh \vh\ \mathrm{d} K
    + \sum_{\e \in \EcalE} 
    \int_{\e} \nbfE \cdot \nabla \vh \vh \ \mathrm{d} s
= 0.
\]
We deduce~$\nabla \vh=0$ in~$\E$,
whence~$\vh$ is constant.
Thanks to~\eqref{muvh}, $\vh$ has zero average over each edge.
The assertion follows.
\end{proof}

The global nonconforming virtual element space  is constructed
by a standard coupling of the interface degrees of freedom~\eqref{edge-dofs}:
\begin{equation}\label{global-space}
\Vh(\taun)
:=
\left\{
\vh \in H^{1,nc}(\taun,k) \middle|
\vh{}_{|\E} \in \VhE \; \forall \E \in \taun \right\}.
\end{equation}
Further, we define nonconforming virtual element spaces with weakly imposed boundary conditions:
if~$g$ is in~$L^1(\partial \Omega)$,
\begin{equation}\label{Vkg}
\Vh^g(\taun)
:= \left\{
\vh \in \Vh(\taun) \middle| 
\int_\e (\vh-g) \qtildekmo = 0 \quad
\forall \qtildekmo \in \Pbbtilde_{k-1}(\e),\ \e\in\Ecalh^B
\right\}.
\end{equation}

\noindent The above spaces and degrees of freedom are a generalization to curved elements
of their straight counterparts in~\cite{AyusodeDios-Lipnikov-Manzini:2016}.
Interpolation estimates are derived in Section~\ref{section:interpolation} below.

\begin{remark}
The enhanced version of the space $V_h(K)$~\cite{Ahmad-Alsaedi-Brezzi-Marini-Russo:2013}
involves a modification of the space inside the element and not on the (curved) edges.
Therefore, designing an enhanced version of the local spaces in~\eqref{local-space}
is straightforward by combining the tools in~\cite{Ahmad-Alsaedi-Brezzi-Marini-Russo:2013}
with the techniques presented here.
\end{remark}

\subsection{Polynomial projectors and discrete bilinear forms}
\label{subsection:polynomial-projections-bf}
Here, we introduce projections onto polynomial spaces,
stabilizing bilinear forms,
and a discrete bilinear form.

\paragraph*{Projections onto polynomial spaces.}
On each element~$\E$,
we introduce projection operators onto polynomial spaces
of maximum degree~$n$ in~$\Nbb$:
\begin{itemize}
\item the (possibly curved) edge $L^2$ projection $\Pitilde_{n}^{0,\e} : L^2(\e)\rightarrow \Pbbtilde_n(\e)$
given by
\begin{equation} \label{L2-projection-edge}
\int_\e\qtilde_n  (v -\Pitilde_n^{0,\e} v )\ \mathrm{d}s =0
\qquad \forall \qtilde_n\in  \Pbbtilde_n(\e);
\end{equation}
\item the Ritz-Galerkin projection
$\Pitilde_n^{\nabla,\E} : H^1(\E)\rightarrow \Pbb_n(\E)$
satisfying, for all~$\qn$ in~$\Pbb_n(\E)$,
\begin{equation} \label{RG-projection-bulk}
\begin{split}
& \int_\E \nabla \qn \cdot \nabla\Pitilde_n^{\nabla,\E} v \ \mathrm{d}K
= -\int_\E\Delta \qn  v \ \mathrm{d}K
+\sum_{\e\in\calE_{h}^K} \int_\e\Pitilde_{n-1}^{0,\e}
(\nbfE \cdot \nabla \qn) v \ \mathrm{d}s\\ 
& = -\int_\E\Delta \qn  v \ \mathrm{d}K
+\sum_{\e\in\calE_{h,s}^K}\int_\e 
(\nbfE \cdot \nabla \qn)  v \ \mathrm{d}s
+\sum_{\e\in\calE_{h,c}^K}\int_\e\Pitilde_{n-1}^{0,\e}
(\nbfE \cdot \nabla \qn) v \ \mathrm{d}s,
\end{split}
\end{equation}
together with
\begin{equation} \label{RG-projection-zero-average}
\begin{cases}
\int_{\partial \E} (v -\Pitilde_n^{\nabla,\E} v )\ \mathrm{d}s =0   & \text{if } n=1,\\
\int_{\E} (v -\Pitilde_n^{\nabla,\E} v )\ \mathrm{d}\E =0           & \text{if } n\geq 2;
\end{cases}
\end{equation}
\item the $L^2$ projection
$\Pi_n^{0,\E} : L^2(\E)\rightarrow \Pbb_n(\E)$ given by
\begin{equation} \label{L2-projection-bulk}
\int_\E \qn  (v -\Pi_n^{0,\E} v )\ \mathrm{d}K =0
\qquad \forall \qn \in \Pbb_n(\E). 
\end{equation}
\end{itemize}

\begin{remark} \label{remark:RG-projection-inconsistency}
If all the edges of an element~$\E$ are straight,
then $\Pitildenabla \qk= \qk$ for all~$\qk$ in~$\Pbb_k(\E)$ 
and~$\Pitildenabla$ boils down to the standard VEM operator $\Pi^{\nabla,\E}_k$~\cite{BeiradaVeiga-Brezzi-Cangiani-Manzini-Marini-Russo:2013}.
Instead, if at least one edge of~$\E$ is curved,
then $\Pitildenabla \qk \neq \qk$ unless~$k=0$.
This fact is the reason of the lack of polynomial consistency of the method on curved elements.
\end{remark}

Next, we show that the three projectors above are computable by means of the DoFs.
\begin{proposition}
Given the DoFs~\eqref{edge-dofs}--\eqref{bulk-dofs} of a given~$\vh$ in~$\VhE$,
we can compute~$\Pitildeze \vh{}_{|\e}$,
$\Pi_{k-2}^{0,\E} \vh$,
and~$\Pitildenabla \vh$,
where the three operators are defined in~\eqref{L2-projection-edge},
\eqref{RG-projection-bulk}--\eqref{RG-projection-zero-average},
and~\eqref{L2-projection-bulk}, respectively.
\end{proposition}
\begin{proof}
The computability of~$\Pitildeze \vh{}_{|\e}$
and~$\Pi_{k-2}^{0,\E}$ follows immediately from the edge and bulk DoFs, respectively.
As for~$\Pitildenabla \vh$,
since the average condition~\eqref{RG-projection-zero-average} is obviously computable,
it suffices to show the computability of the right hand side of \eqref{RG-projection-bulk}.
The first term on the right-hand side is computable from~\eqref{bulk-dofs};
the second and third terms are computable from~\eqref{edge-dofs}.
\end{proof}

\begin{remark} \label{remark:use-of-Pitildenabla}
Differently from the standard nonconforming virtual element framework~\cite{AyusodeDios-Lipnikov-Manzini:2016},
we do not use an $H^1$ projection,
but rather the Ritz-Galerkin projection~$\Pitildenabla$ in \eqref{RG-projection-bulk}--\eqref{RG-projection-zero-average}.
The reason is that the definition of the local space and the choice of the DoFs
would not allow us to compute the standard $H^1$ projection from the degrees of freedom,
due to the noncomputability of $\int_\e (\nbfE \cdot \nabla \qn)  v_h$
on curved edges.
\end{remark}

\paragraph*{A discrete bilinear form.}
We define a discrete local bilinear form
$\aE_h (\cdot,\cdot): \VhE \times \VhE \rightarrow \bbR$
on each element~$\E$ as follows:
\[
\ahE(\uh,\vh) = \aE(\Pitildenabla \uh ,\Pitildenabla \vh)
+\SE((I-\Pitildenabla )\uh,(I-\Pitildenabla )\vh)
\quad \forall \uh,\ \vh \in \VhE.
\]
Above, $\SE(\cdot,\cdot) : \VhE\times \VhE\rightarrow \bbR$ 
is a bilinear form that satisfies two properties:
it is computable via the local set of DoFs over~$\E$;
it satisfies the stability bounds
\begin{equation} \label{stability-bounds}
\SemiNorm{\vh}_{1,\E}^2
\lesssim \SE(\vh,\vh)
\lesssim \SemiNorm{\vh}_{1,\E}^2
\qquad \forall \vh \in \VhE \cap \ker(\Pitildenabla).
\end{equation}
Denoting by $\{ \bm{D}^l \}_{l=1}^{N_{\mathrm{dof}}(\E)}$ the set of all degrees of freedom of $\VhE$,
a possible stabilization satisfying~\eqref{stability-bounds} is
\[
\SE(\uh,\vh) = \sum_{l=1}^{N_{\mathrm{dof}}(\E)}
\bm{D}^l(\uh)\bm{D}^l(\vh),
\]
which can be rewritten in terms of boundary and bulk contributions as
\begin{equation}\label{dofi-dofi}
\SE(\uh,\vh)
= \sum_{\e\in\EcalE} \sum_{i=1}^{k}\bm{D}_{\e}^i(\uh)\bm{D}_{\e}^i(\vh)
+\sum_{j=1}^{k(k-1)/2}\bm{D}_{\E}^j(\uh)\bm{D}_{\E}^j(\vh).
\end{equation}
We postpone to Section~\ref{section:stability} below the analysis of such a stabilizing term.
Of course, other stabilizations satisfying the computability and stability properties can be defined;
we stick to the choice in~\eqref{dofi-dofi} since it is the most popular in the virtual element community.

\begin{remark} \label{remark:inconsistency}
The bilinear form~$\ahE(\cdot,\cdot)$ does not satisfy the usual consistency property of the virtual element method~\cite{AyusodeDios-Lipnikov-Manzini:2016, BeiradaVeiga-Brezzi-Cangiani-Manzini-Marini-Russo:2013}
when~$\E$ is a curved polygon.
The reason is the use of a polynomial projection that is not the usual $H^1$ projection;
see Remark~\ref{remark:RG-projection-inconsistency} for further details.
\end{remark}

Finally, we introduce a global discrete bilinear form
$a_h (\cdot,\cdot): \Vh \times \Vh\rightarrow \bbR$ defined as
\begin{equation}\label{ak}
\ah(\uh,\vh) = \sum_{\E\in\taun}\ahE(\uh,\vh)
\qquad \forall \uh,\ \vh \in \Vh.
\end{equation}

\paragraph*{The discrete right-hand side.}
Here, we construct a computable discretization of the right-hand side~$(f, \vh)_{0,\Omega}$ in~\eqref{weak}.
For $k\geq2$, we introduce
\begin{equation} \label{fk1}
(\fh, \vh)_{0,\Omega}
= \sum_{\E\in\taun}\int_\E  f\ \PizEkmt\vh \mathrm{d}K 
\qquad \qquad \forall \vh \in \Vh;
\end{equation}
Instead, for~$k = 1$, we approximate~$f$ by its piecewise projection onto constants,
average the test function~$\vh$ over the edges of~$\E$,
and write
\begin{equation}
\label{fk2}
(\fh, \vh)_{0,\Omega}
= \sum_{\E\in\taun} 
\left\{ |\E| (\Pi_{0}^{0,\E} f) 
\left( \dfrac{1}{N_{\e}} \sum_{\e\in\EcalE}\bm{D}_{\e}^0(\vh) \right) \right\}.
\end{equation}

\subsection{The method} \label{subsection:VEM}
We propose the following nonconforming virtual element method:
\begin{equation}\label{VEM}
\begin{cases}
& \text{find $\uh\in \Vh^g(\taun)$ such that} \\
& \ah (\uh,\vh)= (\fh, \vh)_{0,\Omega}
        \qquad\qquad \forall \vh\in \Vh^0(\taun).
\end{cases}
\end{equation}
The well posedness of method~\eqref{VEM}
requires further technical tools,
which we derive in Section~\ref{section:stability} below.
For this reason, we postpone its proof to Theorem~\ref{theorem:global-stability&well-posedness}.

\begin{remark}
Method~\eqref{VEM} is designed for meshes with rather general ${\mathcal C}^1$ curved edges.
The assumptions on the meshes in Section~\ref{subsection:meshes},
and notably the fact that curved edges are used only on the (fixed) curved boundary,
will be used to derive convergence estimates in Section~\ref{section:convergence} below.
\end{remark}

\section{Stability analysis}\label{section:stability}
In this section, we prove the stability bounds~\eqref{stability-bounds} for the stabilization~\eqref{ak}.
We proceed in some steps.
First, we recall the inverse inequalities in \cite[Lemma~$6.3$]{BeiraodaVeiga-Lovadina-Russo:2017}
for curved elements.
The proof is independent on whether the element is curved or not.
\begin{lemma}\label{lemma:inverse-L2Laplacian-H1}
Let $\E\in\taun$, for any $v\in H^1(\E)$ such that $\Delta v\in \Pbb_n(\E)$,
we have
\[
\|\Delta v\|_{0,\E}\lesssim \hE^{-1}\SemiNorm{v}_{1,\E}.
\]
\end{lemma} 

We introduce the scaled norm
\[
\NNorm{w}_{\frac12,\partial \E}
:= \hE^{-\frac12}\Norm{w}_{0,\partial \E}
    +\SemiNorm{w}_{\frac12,\partial \E},
\]
and the corresponding negative norm
\[
\NNorm{w}_{-\frac12,\partial \E}
:= \sup_{q\in H^{\frac12}(\partial \E)}
    \frac{\int_{ \partial \E}qw\ \mathrm{d}s}
{\NNorm{q}_{\frac12,\partial \E}}.
\]

We recall the Neumann trace inequality for Lipschitz domains;
see, e.g., \cite[Theorem A.33]{Schwab:1998}.
\begin{lemma} \label{lemma:Neumann-trace-inequality}
Given an element~$K$ and~$v$ in~$H^1(\E)$ such that
$\Delta v$ belongs to~$L^2(\E)$, we have
\[
\NNorm{\nbfE \cdot \nabla v}_{-\frac12,\partial \E}
\lesssim |v|_{1,\E}+\hE\|\Delta v\|_{0,\E}.
\]
\end{lemma}

Next, we state a polynomial inverse inequality on the boundary of an element~$\E$.

\begin{lemma}\label{lemma:inv012}
Given an element~$\E$ and~$\vh$ in~$\VhE$,
we have
\[
\Norm{\nbfE \cdot \nabla \vh}_{0,\partial \E}
\lesssim \hE^{-\frac12} \NNorm{\nbfE \cdot \nabla \vh}_{-\frac12,\partial \E}.
\]
\end{lemma}
\begin{proof}
Using that~$\vh$ belongs to~$\VhE$,
we have that~$\nbfE \cdot \nabla \vh$
belongs to~$\Pbbtilde_{k-1}(\e)$.
So, proving the assertion boils down to proving
a mapped polynomial inverse estimate;
see, e.g., \cite[Lemma~$2.3$]{BeiraodaVeiga-Mascotto:2022}.
The uniformity of the constant follows from the regularity
of the parametrization of the curved edge.
\end{proof}

We are ready to show the continuity of the stabilization~$\SE(\cdot,\cdot)$ in~\eqref{dofi-dofi}.

\begin{proposition}\label{proposition:stability-SE}
Given an element~$\E$, $\vh$ and~$\wh$ in~$\VhE$,
and~$\SE(\cdot,\cdot)$ as in~\eqref{dofi-dofi},
we have the continuity property
\[
\SE(\vh,\wh)
\lesssim
(\hE^{-2} \Norm{\vh}^2_{0,\E} + \SemiNorm{\vh}^2_{1,\E})^{\frac12} 
(\hE^{-2} \Norm{\wh}^2_{0,\E} + \SemiNorm{\wh}^2_{1,\E})^{\frac12}.
\]
If~$\vh$ and~$\wh$ have zero average
on ~$\E$ or~$\partial \E$, we further deduce
\[
\SE(\vh,\wh)
\lesssim
\SemiNorm{\vh}_{1,\E} \SemiNorm{\wh}_{1,\E} .
\]
\end{proposition}
\begin{proof}
By the standard Cauchy-Schwarz inequality, it is sufficient to prove that

\begin{equation}\label{svv}
\begin{aligned}
\SE(\vh,\vh)
&= \sum_{\e\in\EcalE} \sum_{i=1}^{k} \bm{D}_{\e}^i(\vh)^2
    +\sum_{i=1}^{k(k-1)/2} \bm{D}_{\e}^j(\vh)^2 \\
&\lesssim 
\hE^{-2}\Norm{\vh}^2_{0,\E}+\SemiNorm{\vh}^2_{1,\E}
\qquad \forall \vh,\wh \in \VhE.
\end{aligned}
\end{equation}
We bound the edge and boundary contributions on the left-hand side separately.
We start with the first one.
The standard trace inequality asserts that
\[
\hE^{-1}\Norm{\vh}_{\partial \E}^2
\lesssim \hE^{-2}\Norm{\vh}_{0,\E}^2+\SemiNorm{\vh}_{1,\E}^2.
\]
Recall that $m_i$ in~$\calM_{k-1}(\Ie)$ are shifted and scaled monomials of maximum degree~$k-1$ over~$\Ie$.
Then, using the regularity of the element, we have
$\|m_i\|_{\Ie}^2\lesssim h_{\Ie} \approx \he \approx \hE$.
Recalling that each element has a uniformly bounded number of edges, it follows that
\[
\begin{split}
& \sum_{\e\in\EcalE} \sum_{i=1}^{k}\bm{D}_{\e}^i(\vh)^2
= \sum_{\e\in\EcalE} |\e|^{-2}
\sum_{i=1}^{k} \left(\int_\e \vh \widetilde{m}_i\ \mathrm{d}s  \right)^2 \\
&\leq \sum_{\e\in\EcalE} |\e|^{-2}\sum_{i=1}^{k}\Norm{\vh}_{0,\e}^2 \|\widetilde{m}_i\|_{0,\e}^2
\lesssim \sum_{\e\in\EcalE} |\e|^{-2}\Norm{\vh}_{0,\e}^2\sum_{i=1}^{k}\|m_i\|_{0,\Ie}^2\\
&\lesssim \sum_{\e\in\EcalE}|\e|^{-1} \Norm{\vh}_{0,\e}^2
\lesssim \hE^{-1}\Norm{\vh}_{0,\partial \E}^2 
\lesssim \hE^{-2}\Norm{\vh}^2_{0,\E}+\SemiNorm{\vh}^2_{1,\E},
\end{split}
\]
which is the bound on the edge contributions of the stabilization in~\eqref{dofi-dofi}.

Next, we show the upper bound for the second term on the left-hand side of~\eqref{svv}.
Since we have that~$\Norm{m_j}_{L^\infty(\E)} \leq 1$
for any shifted and scaled monomial $m_j$ in~$\calM_{k-2}(\E)$,
it is possible to infer
\[
\begin{aligned}
& \sum_{j=1}^{k(k-1)/2} \bm{D}_{\e}^j(\vh)^2
= |\e|^{-2} \sum_{j=1}^{k(k-1)/2}
\left(\int_\E \vh m_j\ \mathrm{d}K\right)^2 \\
&\leq |\e|^{-2} \!\!\! \sum_{j=1}^{k(k-1)/2} \!\!\!
\left(\int_\E \vh \ \mathrm{d}K\right)^2
\leq |\e|^{-1} \!\!\!\sum_{j=1}^{k(k-1)/2}\!\!\!
\Norm{\vh}_{0,\E}^2
\lesssim \hE^{-2} \Norm{\vh}_{0,\E}^2,
\end{aligned}
\]
which concludes the main part of the proof. 
The final assertion follows immediately from known Poincar\'e-type inequalities on Lipschitz domains.
\end{proof}

In order to show the coercivity of the stabilization~$\SE(\cdot,\cdot)$ in~\eqref{dofi-dofi}, we need a further technical result;
see, e.g., \cite[Lemma~$4.1$]{Chen-Huang:2018} for the straight edge case
(the extension to curved edges follows with a mapping argument).
 
\begin{lemma}\label{lemma:norm-equivalence}

Let $e$ an edge of $K$ and $a := \sum_{i} a_{\e,i} \widetilde{m}_{\e,i}$ belong to~$\Pbbtilde_n(\e)$
where~$\widetilde{m}_{\e,i}$ are the shifted and scaled monomials in~$\widetilde{\calM}_n(\e)$,
and collect the coefficients~$a_{\e,i}$ in the vector~$\abf$.
Then, we have the following norm equivalence:
\[
\he\|\abf\|_{\ell^2}^2\lesssim \|a\|_{0,\e}^2\lesssim \he\|\abf\|_{\ell^2}^2.
\]
Let~$b := \sum_{j} b_j m_j$ belong to~$\Pbb_n(\E)$
where~$m_j$ are the shifted and scaled monomials in~$\calM_n(\E)$,
and collect the coefficients~$b_j$ in the vector~$\bbf$.
Then, we have the following norm equivalence
\[
\hE^2 \|\bbf\|_{\ell^2}^2
\lesssim \|b\|_{0,\E}^2
\lesssim \hE^2\|\bbf\|_{\ell^2}^2.
\] 
\end{lemma}

We are now ready to show the coercivity of the stabilization~$\SE(\cdot,\cdot)$ in~\eqref{dofi-dofi}.
\begin{proposition}\label{proposition:stability-ahE}
Given an element~$\E$, $\vh$ in~$\VhE$,
and~$\SE(\cdot,\cdot)$ as in~\eqref{dofi-dofi},
we have the coercivity property
\begin{equation}\label{Sstability2}
\SE(\vh,\vh)\gtrsim \SemiNorm{\vh}^2_{1,\E}.
\end{equation}
\end{proposition}
\begin{proof}
Let $\vh$ in~$\VhE$. An integration by parts gives
\begin{equation}\label{nablav2}
\begin{split}
|\vh|_{1,\E} ^2 
& = -\int_\E \Delta \vh \vh\ \mathrm{d}K
 +\int_{\partial \E} (\nbfE \cdot \partial \vh) \vh\ \mathrm{d}s\\
& = -\int_\E \Delta \vh \vh\ \mathrm{d}K
  +\sum_{\e\in\EcalE}\int_{\e} (\nbfE \cdot \partial \vh) \vh\ \mathrm{d}s.
\end{split}
\end{equation}
Expanding~$\Delta \vh$ into a shifted and scaled monomial basis,
and using Lemmas~\ref{lemma:inverse-L2Laplacian-H1} and~\ref{lemma:norm-equivalence},
the first integral on the right-hand side of~\eqref{nablav2} can be dealt with as follows:
\begin{equation}\label{deltavv}
\begin{aligned}
-\int_\E \Delta \vh \vh\ \mathrm{d}K
&= -\sum_{j=1}^{k(k-1)/2}  b_j\int_\E m_j  \vh\ \mathrm{d}K
= -\sum_{j=1}^{k(k-1)/2}  b_j |\E| \bm{D}^j_{\e}(\vh)\\
&\leq  |\E| \|\bbf\|_{\ell^2}
\left( \sum_{j=1}^{k(k-1)/2}\bm{D}^j_{\e}(\vh)^2 \right)^{1/2}
\leq  |\E| \|\bbf\|_{\ell^2}\SE(\vh,\vh)^{1/2}\\
&\lesssim \hE\|\Delta \vh\|_{0,\E}\SE(\vh,\vh)^{1/2}
\lesssim |\vh|_{1,\E}\SE(\vh,\vh)^{1/2}.
\end{aligned}
\end{equation}
As for the boundary integral in~\eqref{nablav2}, first expanding $(\nbfE \cdot \nabla \vh)$
into the $\{\widetilde{m}_{i,\e}\}_{i=1}^k$ basis, then
employing Lemmas~\ref{lemma:inverse-L2Laplacian-H1},
\ref{lemma:inv012}, and~\ref{lemma:norm-equivalence},
we infer
\begin{equation}\label{nablavnv}
\begin{aligned}
& \sum_{\e\in\EcalE}\int_{\e}
(\nbfE \cdot \nabla \vh) \vh \ \mathrm{d}s 
= \sum_{\e\in\EcalE}\sum_{i=1}^{k} a_{i,\e}
\int_{ e}\widetilde{m}_{i,\e}\vh\ \mathrm{d}s
= \sum_{\e\in\EcalE} \sum_{i=1}^{k} a_{i,\e} |\e|\bm{D}^i_{\e}(\vh)\\
&\leq  \sum_{\e\in\EcalE}|\e| \|\abf\|_{\ell^2}
\left(\sum_{i=1}^{k}\bm{D}^i_{\e}(\vh)^2 \right)^{1/2}
\leq \sum_{\e\in\EcalE}|\e| \|\abf\|_{\ell^2}\SE(\vh,\vh)^{1/2}\\
&\lesssim \sum_{\e\in\EcalE} \he^{\frac12} 
\Norm{\nbfE \cdot \nabla \vh}_{0,\e} \SE(\vh,\vh)^{1/2}
\lesssim \hE^{\frac12}
\Norm{\nbfE \cdot \nabla \vh}_{0,\partial \E}
 \SE(\vh,\vh)^{1/2}\\
&\lesssim \NNorm{\nbfE \cdot \nabla \vh}_{-\frac12, \partial \E} \SE(\vh,\vh)^{1/2}
\lesssim (|\vh|_{1,\E} + \hE\|\Delta \vh\|_{0,\E})\SE(\vh,\vh)^{1/2}\\
&\lesssim |\vh|_{1,\E}\SE(\vh,\vh)^{1/2} .
\end{aligned}
\end{equation}
Collecting \eqref{deltavv} and \eqref{nablavnv} in \eqref{nablav2},
we obtain the bound in \eqref{Sstability2}.
\end{proof}

Consider the operator~$\Pi_0^0 : H^1(\E) \to \Rbb$ given by
\[
\Pi_0^0 v := \frac{1}{|\partial \E|}\int_{\partial \E} v\ \mathrm{d}s.
\]
The following scaled Poincar\'e inequality is valid:
\begin{equation}\label{poincare}
\|v-\Pi_0^0 v\|_{0,\E}\lesssim \hE\SemiNorm{v}_{1,\E}.
\end{equation}
To analyze the stability properties of the local discrete bilinear form, we need the following result,
which states the stability in~$H^1$ of the operator~$\Pitildenabla$ defined in~\eqref{RG-projection-bulk}--\eqref{RG-projection-zero-average}.

\begin{lemma}\label{PTcont}
Given an element~$\E$ and $v_h$ in~$\VhE$, we have
\begin{equation}\label{tpikbd}
|\Pitildenabla v_h|_{1,\E}
\lesssim \SemiNorm{v_h}_{1,\E}.
\end{equation}
\end{lemma}
\begin{proof}
Introduce $\bar v_h = v_h-\Pi_0^0 \vh$.
By definition of~$\Pitildenabla$,
we have $|\Pitildenabla \bar v_h|^2_{1,\E}  = |\Pitildenabla \vh|^2_{1,\E} $.  
Therefore, substituting $q_k = \Pitildenabla \bar v_h$ in \eqref{RG-projection-bulk} we obtain
$$
\begin{aligned}
& |\Pitildenabla \vh|_{1,\E}^2 = |\Pitildenabla \bar v_h |_{1,\E}^2 = \\
& = -\int_\E \bar v_h \Delta \Pitildenabla \bar \vh \ \mathrm{d}K
+\sum_{\e\in\calE_{s}^K}\int_\e\bar v_h
(\nbfE \cdot \nabla \Pitildenabla \bar v_h)  \ \mathrm{d}s
+ \sum_{\e\in\calE_{c}^K}\int_\e\bar v_h 
\Pitildeze (\nbfE \cdot \nabla \Pitildenabla \bar v_h)  \ \mathrm{d}s\\
&= \int_\E \nabla\bar v_h\cdot\nabla \Pitildenabla \bar v_h \ \mathrm{d}K
-\sum_{\e\in\calE_{c}^K}\int_\e
\bar v_h  (I- \Pitildeze)
(\nbfE \cdot \nabla \Pitildenabla \bar v_h)\ \mathrm{d}s\\
&\leq |\bar v_h|_{1,\E}|\Pitildenabla \bar v_h|_{1,\E}
+\|\nbfE \cdot \nabla \Pitildenabla \bar v_h \|_{0,\partial \E}
\sum_{\e\in\calE_{c}^K} \|\bar v_h\|_{0,\e}.
\end{aligned}
$$
Applying  standard trace inequality and the scaled Poincar\'e inequality, we infer
\[
\hE^{-\frac12}\|\bar v_h\|_{0,\e}
\lesssim \hE^{-1}\|\bar v_h\|_{0,\E}+|\bar v_h|_{1,\E}
\lesssim|\bar v_h|_{1,\E}
\]
and
\[
\hE^{\frac12}
\| \nbfE \cdot \nabla \Pitildenabla \bar v_h \|_{0,\partial \E}
\lesssim | \Pitildenabla \bar v_h |_{1,\E}
+\hE|\nabla \Pitildenabla \bar v_h |_{1,\E}
\lesssim |\Pitildenabla \bar v_h|_{1,\E}.
\]
Combining the three above estimate, we get the assertion:
\[
|\Pitildenabla \vh|_{1,\E}^2 
\lesssim |\bar v_h|_{1,\E}|\Pitildenabla \bar v_h|_{1,\E}
\lesssim |{\vh}|_{1,\E}|\Pitildenabla{\vh}|_{1,\E}.
\]
\end{proof}

We conclude this section, by proving stability estimates on the local discrete bilinear form~$\ahE(\cdot,\cdot)$;
stability estimates for the global discrete bilinear form~$\ah(\cdot,\cdot)$ are an immediate consequence.
\begin{proposition}\label{prop:Xnew}
Given an element~$\E$ and the discrete bilinear form~$\ahE(\cdot,\cdot)$ based on the stabilization in~\eqref{dofi-dofi},
we have the stability estimates
\[
\SemiNorm{\vh}^2_{1,\E}
\lesssim \ahE(\vh,\vh)
\lesssim \SemiNorm{\vh}^2_{1,\E} 
\qquad\qquad \forall \vh \in \VhE .
\]
\end{proposition}
\begin{proof}
We first prove the lower bound.
Let $\bar v_h = \vh-\Pi_0^0 \vh$.
Using Proposition~\ref{proposition:stability-ahE}, we have
\[
\begin{aligned}
\aE(\vh,\vh) 
\!=\! \aE(\bar v_h,\bar v_h) 
& \!\lesssim\! \SE(\bar v_h,\bar v_h)
\!\lesssim\! \SE((I-\Pitildenabla)\bar v_h,(I-\Pitildenabla)\bar v_h)
\!+\! \SE(\Pitildenabla \bar v_h,\Pitildenabla \bar v_h).
\end{aligned}
\]
Since~$\Pitildenabla$ preserves constants, we can write
$(I-\Pitildenabla)\bar v_h = (I-\Pitildenabla)\vh$.
From Proposition~\ref{proposition:stability-SE}
applied to the second term on the right-hand side
of the inequality above, we infer
\[
\SE(\Pitildenabla \bar v_h,\Pitildenabla \bar v_h) 
\lesssim \hE^{-2}\|\Pitildenabla \bar v_h\|^2_{0,\E} +|\Pitildenabla \bar v_h|^2_{1,\E}.
\]
By the definition of~$\Pitildenabla$, it follows that
\[
\Pi_0^0\Pitildenabla \bar v_h 
= \Pi_0^0\bar v_h =\Pi_0^0(\vh-\Pi_0^0\vh)=0.
\]
Using~\eqref{poincare} and $|\Pitildenabla \bar v_h|^2_{1,\E}  = |\Pitildenabla \vh|^2_{1,\E} $, we deduce
\[
\SE(\Pitildenabla \bar v_h,\Pitildenabla \bar v_h) 
 \lesssim |\Pitildenabla \vh|^2_{1,\E} .
\]
Combining the above inequalities yields
\[
\aE(v,v) 
\lesssim  \SE((I-\Pitildenabla)\vh,(I-\Pitildenabla)\vh)  
    +|\Pitildenabla \vh|^2_{1,\E} = \ahE(\vh,\vh).
\]
Next, we focus on the upper bound.
Recalling that $\Pi_0^0\Pitildenabla \vh =\Pi_0^0 \vh $,
the definition of~$\ahE(\cdot,\cdot)$ in~\eqref{ak},
and inequality~\eqref{tpikbd},
we conclude the proof:
\[
\begin{aligned}
\ahE(\vh,\vh)
 &= |\Pitildenabla \vh|^2_{1,\E}+\SE((I-\Pitildenabla)\vh,
 (I-\Pitildenabla)\vh) \\
&\lesssim \SemiNorm{\vh}^2_{1,\E}+h^{-2}_{\E}\|(I-\Pitildenabla)\vh\|_{0,\E}^2+|(I-\Pitildenabla)\vh|_{1,\E}^2\\
&\lesssim \SemiNorm{\vh}^2_{1,\E}+|(I-\Pitildenabla)\vh|_{1,\E}^2
\lesssim \SemiNorm{\vh}^2_{1,\E} .
\end{aligned}
\]
\end{proof}

Upon using Proposition~\ref{prop:Xnew}
and the Brenner-Poincar\'e inequality~\cite{Brenner:2003},
we derive the following global stability estimates
and the well posedness of method~\eqref{VEM}.
\begin{theorem} \label{theorem:global-stability&well-posedness}
Given the global discrete bilinear form~$\ah(\cdot,\cdot)$ based on the local stabilizations in~\eqref{dofi-dofi},
we have the global stability estimates
\[
\SemiNorm{\vh}^2_{1,h}
\lesssim \ah(\vh,\vh)
\lesssim \SemiNorm{\vh}^2_{1,h}
\qquad \qquad \vh \in \Vh.
\]
As a consequence, method~\eqref{VEM} is well posed.
\end{theorem}

\section{Polynomial and virtual element approximation estimates} \label{section:interpolation}
In this section, we introduce error estimates by means of polynomial and virtual element functions.
Notably, we investigate the approximation properties
of the three following approximants for sufficiently regular functions~$v$:
\begin{itemize}
\item the~$L^2$ projection~$\vpi$ of $v$ into the polynomial space~$\Pbb_k(\E)$;
\item
the virtual element function~$\vpE$ defined as the solution to
\begin{equation}\label{vpk}
\begin{cases}
-\Delta \vpE = -\PizEkmt\Delta v
& \text{in } \E \\
\nbfE\cdot \nabla \vpE =\Pitildeze (\nbfE \cdot \nabla v) - \cpartialE
& \text{on } \e\in\EcalE \\
\int_{\partial\E}(\vpE -v)\ \mathrm{d}s = 0,
\end{cases}
\end{equation}
where
\[
\cpartialE:=
\begin{cases}
    -\vert \partial\E\vert^{-1} \int_\E \Delta v    & \text{if } k=1 \\
    0                                               & \text{if } k\ge2; \\
\end{cases}
\]
\item the DoFs interpolant~$v_I^\E$ in~$\VhE$ of~$v$
defined as
\begin{equation}\label{vik}
\begin{cases}
\int_\E (\vIE-v) m\ \mathrm{d}K = 0
& \forall m \in \calM_{k-2}(\E) \\
\int_\e (\vIE-v) \widetilde{m}\ \mathrm{d}s  =0 
&  \forall \widetilde{m} \in \widetilde{\calM}_{k-1}(\e), \quad \forall e\in \EcalE.
\end{cases}
\end{equation}
\end{itemize}
The functions~$\vpi$ and~$\vIE$ are well defined by construction.
Also~$\vpE$ is well defined; in fact, problem~\eqref{vpk} is well posed
as compatibility conditions are valid.
For~$k\ge2$,
\[
\begin{split}
& \sum_{\e\in\EcalE} \int_\e \nbfE\cdot \nabla \vpE \mathrm{d}s 
  = \sum_{\e\in\EcalE} \int_\e
   \Pitilde_{k}^{0,\e} (\nbfE\cdot \nabla v) \mathrm{d}s 
  = \sum_{\e\in\EcalE} \int_\e
\nbfE\cdot \nabla v \ \mathrm{d}s \\
& = -\int_\E\Delta v\mathrm{d}K 
  =  -\int_\E\PizEkmt\Delta v\mathrm{d}K
  = -\int_\E \Delta \vpE \mathrm{d}K ;
\end{split}
\]
for~$k=1$,
\[
\begin{split}
& \sum_{\e\in\EcalE} \int_\e \nbfE\cdot \nabla \vpE \mathrm{d}s 
  = \sum_{\e\in\EcalE} \int_\e
   \Pitilde_{k}^{0,\e} (\nbfE\cdot \nabla v) \mathrm{d}s 
   - \cpartialE \vert\partial\E\vert 
  = \sum_{\e\in\EcalE} \int_\e
     \nbfE\cdot \nabla v \ \mathrm{d}s 
     - \cpartialE \vert\partial\E\vert \\
&  = -\int_\E\Delta v\mathrm{d}K  + \int_\E\Delta v\mathrm{d}K  
  =  0
  = -\int_\E \Delta \vpE \mathrm{d}K.
\end{split}
\]
While~$\vpE$ can be constructed elementwise providing a piecewise discontinuous
virtual element approximant,
the DoFs interpolant~$\vIE$ yields a global nonconforming virtual element interpolant~$\vI$
by coupling the face DoFs. The discontinuous approximant $\vpE$ will be instrumental in proving the approximation properties of $\vIE$.

\begin{assumption} \label{assumption:eta}
Henceforth, the regularity parameter~$\ord$ of the curved boundary
introduced in Section~\ref{section:introduction} satisfies $\ord \ge k$.
\end{assumption}

Polynomial error estimates for~$\vpi$ are well know; see e.g.\cite{Brenner-Scott:2008}:
\[
\|v-\vpi\|_{0,\E}+\hE|v-\vpi|_{1,\E}
\lesssim \hE^{k+1} |v|_{k+1,\E}
\qquad \forall v\in H^{k+1}(\E).
\]
To prove interpolation error estimates,
we present an auxiliary error estimate,
which can be proven proceeding along the same lines as in~\cite[Lemma~$3.2$]{BeiraodaVeiga-Russo-Vacca:2019}
for the operator $\Pitilde_{n}^{0,\e}$.

\begin{lemma}\label{lemma:proes}
Let~$n \in \Nbb$ and the regularity parameter~$\eta$ of the curved boundary satisfy~$\eta\ge n+1$.
Then, given an element~$\E$ and any of its edges~$\e$,
for all $0\leq m \leq s \leq n + 1$, we have
\[
|v - \Pitilde_{n}^{0,\e}v|_{m,\e}
\lesssim \he^{s-m}\Norm{v}_{s,\e}
\qquad \forall v \in H^s(\e) .
\]
\end{lemma}

We also recall the properties of the Stein's extension operator~$E$ in~\cite[Chapter~VI, Theorem~$5$]{Stein:1970}.
\begin{lemma} \label{lemma:Stein}
Given a Lipschitz domain~$\Omega$ in~$\Rbb^2$ and $s \in \Rbb$, $s \ge 0$,
there exists an extension operator
$E: H^s(\Omega)\rightarrow H^s(\bbR^2)$
such that
\begin{equation}\label{Extension}
Ev |_\Omega = v
\quad\text{and}\quad
\|Ev\|_{s,\bbR^d}\lesssim \Norm{v}_{s,\Omega}
\qquad\qquad \forall v\in H^s(\Omega).
\end{equation}
The hidden constant depends on~$s$
but not on the diameter of~$\Omega$.
\end{lemma}
We are in a position to show the first local interpolation result.
\begin{lemma}\label{v-vpk}
Given~$\vpE$ as in~\eqref{vpk},
then, for all $v$ in~$H^{s+1}(\E)$, $1 \le s \le k$,
we have
\[
\SemiNorm{v-\vpE}_{1,h}
\lesssim h^s\Norm{v}_{s+1,\Omega}.
\]
\end{lemma}
\begin{proof}
An integration by part yields

\[
\begin{aligned}
& |v-\vpE|_{1,\E}^2
= -(\Delta (v-\vpE), v-\vpE)_{0,\E}
+\sum_{\e\in\EcalE}
( \nbfE \cdot \nabla (v-\vpE), v-\vpE )_{0,\e}\\
& = -((I-\PizEkmt)\Delta v, v-\vpE)_{0,\E}
   + \sum_{\e\in\EcalE}  ( (I -\Pitildeze) 
      (\nbfE \cdot \nabla v) , v - \vpE)_{0,\e}
   + (\cpartialE,v-\vpE)_{0,\partial\E}  \\
& = -((I-\PizEkmt)\Delta v, v-\vpE)_{0,\E}
   + \sum_{\e\in\EcalE}  ( (I -\Pitildeze) 
      (\nbfE \cdot \nabla v) , v - \vpE)_{0,\e}.
\end{aligned}
\]
We deduce
\[
|v-\vpE|_{1,\E}^2\leq \|(I-\PizEkmt)\Delta v\|_{0,\E}\|v-\vpE\|_{0,\E}
+ \|v-\vpE\|_{0,\partial \E}
\sum_{\e\in\EcalE} \|(I -\Pitildeze) 
(\nbfE \cdot \nabla v)\|_{0,\e}.
\]
Using polynomial approximation properties
and Lemma~\ref{lemma:proes} gives
\[
\|(I-\PizEkmt)\Delta v\|_{0,\E}
\lesssim \hE^{\s-1}\|\Delta v\|_{\s-1,\E}\lesssim \hE^{\s-1}\SemiNorm{v}_{\s+1,\E} .
\]
We split the edge contributions into curved and straight edges terms.
As for the curved edges terms,
we use Lemma~\eqref{lemma:proes} with~$n=k-1$
(recalling Assumption~\ref{assumption:eta})
and write
\[
\sum_{\e\in\EcalEc} \Norm{(I - \Pitildeze) (\nbfE \cdot \nabla v)}_{0,\e}
\lesssim \sum_{\e\in\EcalEc} \hE^{\s-\frac12} \| \nbfE \cdot \nabla v \|_{\s-\frac12,\e}
\lesssim \hE^{\s-\frac12} \sum_{\e\in\EcalEc}\| \nabla v \|_{\s-\frac12,\e}.
\]
As for the straight edges terms,
since~$\Pitildeze$ is the standard $L^2(\e)$ projection onto polynomials of maximum degree~$k-1$ over~$\e$,
we use the trace inequality and write
\[
\begin{split}
& \sum_{\e\in\EcalEs} \Norm{(I - \Pitildeze) (\nbfE \cdot \nabla v)}_{0,\e}
\le \sum_{\e\in\EcalEs} \Norm{\nbfE \cdot \nabla (v-\qk)}_{0,\e}\\
& \le \sum_{\e\in\EcalEs} \Norm{\nabla (v-\qk)}_{0,\e}
  \lesssim \hE^{-\frac12} \SemiNorm{v-\qk}_{1,\E} 
            + \hE^{\frac12} \SemiNorm{v-\qk}_{2,\E}
        \qquad  \forall \qk \in \Pbb_k(\E).
\end{split}
\]
Standard polynomial approximation properties imply
\[
\sum_{\e\in\EcalEs} \Norm{(I - \Pitildeze) (\nbfE \cdot \nabla v)}_{0,\e}
\lesssim \hE^{\s-\frac12} \SemiNorm{v}_{\s+1,\E}.
\]
By the Poincar\'e inequality and trace inequality,
we infer
\[
\|v-\vpE\|_{0,\E}
\lesssim \hE|v-\vpE|_{1,\E}
\]
and
\[
\|v-\vpE\|_{0,\partial \E}
\lesssim \hE^{-\frac12} \|v-\vpE\|_{0,\E} + \hE^{\frac12} |v-\vpE|_{1,\E}
\lesssim \hE^{\frac12} |v-\vpE|_{1,\E}.
\]
Collecting the above estimates leads us to
\[
|v-\vpE|_{1,\E}\lesssim \hE^{\s}(\SemiNorm{v}_{\s+1,\E}
                     + \sum_{\e \in \EcalEc} \Norm{\nabla v}_{\s-\frac12,\e}).
\]
Summing over all the elements, we arrive at
\[
|v-\vpE|_{1,h}
\lesssim h^\s(\SemiNorm{v}_{\s+1,\E}
              +\sum_{i=1}^N\Norm{\nabla v}_{\s-\frac12,\Gamma_i}).
\]
To end up with error estimates involving terms only in the domain~$\Omega$ and not on its boundary,
we use the Stein's extension operator
of Lemma~\ref{lemma:Stein}.
For any curve $\Gamma_i$ on the boundary of~$\partial\Omega$,
let $\mathcal C_i$ be a domain in $\bbR^2$
with part of its boundary given by $\partial \mathcal{C}_i\in C^{k,1}$.
Then, applying the standard trace theorem on smooth domains
and the stability of the (vector valued version) Stein's extension operator in~\eqref{Extension}, we get
\[
\Norm{\nabla v}_{\s-\frac12,\Gamma_i}
=\|E \nabla v\|_{\s-\frac12,\Gamma_i}
\leq \|E \nabla v\|_{\s-\frac12,\partial\mathcal{C}_i}
\lesssim \|E \nabla v\|_{\s-\frac12,\mathcal{C}_i}
\leq \|E \nabla v\|_{k,\bbR^d}\lesssim \Norm{v}_{\s+1,\Omega}.
\]
\end{proof}
Thanks to the approximation properties
of the piecewise discontinuous interpolant~$\vpE$
we deduce the approximation properties of the global nonconforming interpolant~$\vIE$.
\begin{lemma} \label{lemma:DoFs-interpolant-estimates}
For all~$v$ in~$H^{s+1}(\Omega)$,
$1 \le s \le k$, we have
\[
\SemiNorm{v-\vIE}_{1,h}
\lesssim h^s \Norm{v}_{s+1,\Omega}.
\]
\end{lemma}
\begin{proof} 
We follow the guidelines of~\cite[Proposition~$3.1$]{Mascotto-Perugia-Pichler:2018}.
The idea is to use the definition of the DoFs interpolant
so as to bound the corresponding energy error
by the energy error of any piecewise discontinuous virtual element function.

Recalling the definition of~$\vIE$,
for any $\vh$ in~$\VhE$, we get
\[
\begin{aligned}
& \int_\E \nabla (v-\vIE) \cdot \nabla (v-\vIE)
= - \int_\E \Delta (v-\vIE) (v-\vIE)
+ \sum_{\e\in\EcalE} \int_\e \nbfE \cdot \nabla (v-\vIE) (v-\vIE) \\
& = -\int_\E \Delta (v-\vh) (v-\vIE)
+\sum_{\e\in\EcalE} \int_\e \nbfE \cdot \nabla (v-\vh) (v-\vIE)
= \int_\E \nabla(v-\vh) \cdot \nabla(v-\vIE) .
\end{aligned}
\]
The Cauchy–Schwarz inequality entails
\[
|v-\vIE|_{1,\E}
\leq \inf_{\vh \in \VhE} |v-\vh|_{1,\E}
\leq |v-\vpE|_{1,\E}.
\]
The assertion follows summing over the elements,
collecting the above inequalities,
and using Lemma~\ref{v-vpk}.
\end{proof}

\section{Convergence analysis} \label{section:convergence}
In this section, we present the error analysis in the energy and $L^2$-norms
for the nonconforming virtual element method~\eqref{VEM};
see Sections~\ref{subsection:H1-error} and~\ref{subsection:L2-error}, respectively.
In particular, we first derive Strang-like error bounds
and derive optimal error estimates
based on the tools in Section~\ref{section:interpolation}.

\subsection{Convergence analysis in the energy norm} \label{subsection:H1-error}

We prove the error estimates in the energy norm
in some steps.
First, we discuss details on geometrical errors due to the presence of curved elements.

Given a curved edge~$\e$,
we denote the straight segment with endpoints given by the vertices of~$\e$ by~$\ehat$; see Figure~\ref{gammaandzeta}.

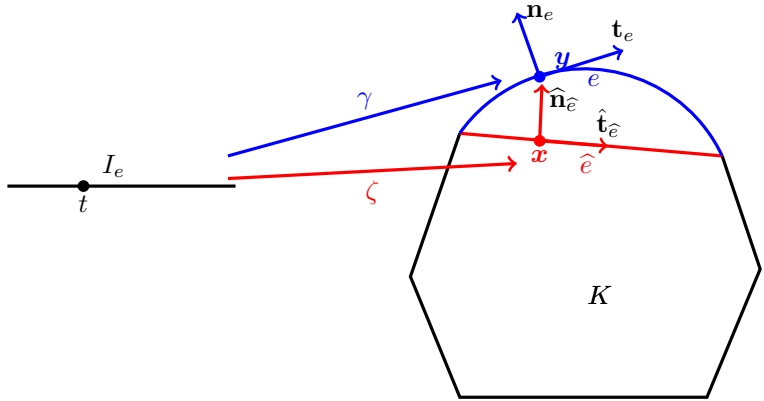
\begin{figure}[H]
\centering
{\scalefont{1}
\begin{tikzpicture}
\draw[very thick] (-7.9,0.2)--(-4.9,0.2);
\draw (-6.5,0.2) node [above] {$\Ie$};
\draw[very thick] (-1.95,0.9)--(-2.6,-1.0)--(-1.95,-2.6)--(1.3,-2.6)--(2,-0.9)--(1.5,0.6);
\draw[red,very thick] (1.5,0.6)--(-1.95,0.9);
\draw[blue,very thick] (1.5,0.6) arc (25:145:2);
\draw[red,very thick,->] (-0.9,0.8)--(-0.87,1.55);
\draw[red,very thick,->] (-0.9,0.8)--(0,0.73);
\draw[blue,very thick,->](-0.9,1.65)--(-1.2,2.5) ;
\draw[blue,very thick,->](-0.9,1.65)--(0.2,2) ;
\draw[red,very thick,->] (-5,0.3)--(-1.2,0.5);
\draw[blue,very thick,->](-5,0.6)--(-1.4,1.6) ;
\draw (-3.1,0.4) node [red,below] {$\zeta$};
\draw (-3.2,1.1) node[blue,above] {$\gamma$};
\draw (-0.9,0.8) node [red,below] {${\bm x}$};
\draw (-0.19,1.6) node[blue] {$\e$};
\draw (-0.29,0.5) node[red] {$\ehat$};
\draw (-0.6,1.62) node[blue,above] {$\bm y$};
\draw (-1.2,2.5) node[right] {$\nbfe$};
\draw (0.2,2) node[above] {${\bf t}_{\e}$};
\draw (0,0.73)node[above] {$\hat{\bf t}_{\ehat}$};
\draw (-0.9,1.35) node[right] {$\nbfhatehat$};
\draw (-6.9,0.2) node[below] {$t$};
\draw (-0.11,-1) node[below] {$ K$};
\draw (-0.11,-1) node[below] {$ K$};
\filldraw [red](-0.9,0.8) circle [radius=2pt];
\filldraw [blue](-0.9,1.65) circle [radius=2pt];
\filldraw [black](-6.9,0.2) circle [radius=2pt];
\end{tikzpicture}
}
\caption{\small{Parametrizations of the curved edge~$\e$ and straight segment~$\ehat$
by means of the parametrizations~$\gamma$ and~$\zeta$.}}
\label{gammaandzeta}
\end{figure}

Recall that~$\nbfe$ is an assigned unit normal vector to the curved edge~$\e$;
fix a unit normal vector~$\nbfhatehat$ to the segment~$\ehat$.
Given~$\gamma$ from the interval $\Ie:=[t_\e^1,t_\e^2]$ into $\Rbb^2$ 
the usual parametrization of the edge~$\e$, we have the following standard estimate; see, e.g., \cite[eq. (2.18)]{Burman-Hansbo-Larson:2018}:
\begin{equation}\label{O(h)}
\|\nbfe-\nbfhatehat\|_{L^{\infty}(\e)}\lesssim \hE \|\gamma\|_{W^{1,\infty}(I_{\e})}.
\end{equation}
With the notation as in Figure~\ref{gammaandzeta},
for any curved edge~$\e$ with endpoints~$\bm{x}_{\e}^1$ and~$\bm{x}_{\e}^2$, 
we introduce the linear map
$\zeta: \Ie \rightarrow \ehat$
given by
\[
\zeta(t) = \dfrac{t-t_{\e}^1}{t_{\e}^2 - t_{\e}^1}(\bm{x}_{\e}^2-\bm{x}_{\e}^1)+\bm{x}_{\e}^1 .
\]
We further define the constant tangent vector to~$\ehat$
\[
\hat{\mathbf{t}}_{\ehat} =\frac{ \bm{x}_{\e}^2 - \bm{x}_{\e}^1}{t_{\e}^2 - t_{\e}^1}.
\]
We have that $\hat{\mathbf{t}}_{\ehat}$ is equivalently defined as~$\zeta'(t)$.
Analogously, we can define ${\mathbf{t}}_{\e} :=\gamma'(t)$.
Since we are approximating a curved edge by a straight segment,
it is known that,
see, e.g., \cite[eq. (2.18)]{Burman-Hansbo-Larson:2018},
\[
\| \hat{\mathbf{t}}_{\ehat} - \mathbf{t}_{\e} \|_{L^\infty(\Ie)}
\lesssim \hE \|\gamma\|_{W^{1,\infty}(I_\e)}.
\]
We rewrite the two parametrizations~$\zeta$ and~$\gamma$ as (for any $t \in \Ie$)
\[
\bm{x}(t) = \zeta(t) = \int_{t_{\e}^1}^{t} \hat{\mathbf{t}}_{\ehat} ds + \bm{x}_{\e}^1,
\qquad\qquad
\bm{y}(t) = \gamma(t) = \int_{t_{\e}^1}^{t} \mathbf{t}_{\e}(s) ds+ \bm{x}_{\e}^1.
\]
Using that the length of the reference interval is approximately~$\hE$,
the difference of the two can be estimated as follows:
 \begin{equation}\label{O(h2)}
\|\bm{x}(t) - \bm{y}(t) \|
\le \int_{t_{\e}^1}^{t} \| \mathbf{t}_{\e}(s) - \hat{\mathbf{t}}_{\ehat} \| ds 
\lesssim \hE^2 .
\end{equation}
Finally, we clearly have $\psi\circ \zeta \in \Pbb_n(\Ie)$
for any $\psi$ in~$\Pbb_{n}(\ehat)$.
Thus, we write
 \begin{equation}\label{tildepsi}
 \tilde{\psi} = \psi\circ \zeta\circ\gamma^{-1}\in \Pbbtilde_{n}(\e) .
 \end{equation}
We further introduce the $H^1$ projection
$\Pinabla: H^1(\E) \to \Pbb_k(\E)$
for all elements~$\E$ as follows:
given a polynomial degree~$k$,
\begin{equation}\label{piknabla}
\begin{cases}
\int_\E \nabla \qk \cdot \nabla (v - \Pinabla v) = 0  \\
\int_{\partial \E} (v - {\Pi}_k^{\nabla,\E} v )
\end{cases}
\qquad   \forall \qk \in \Pbb_k(\E), 
\quad \forall v \in H^1(\E).
\end{equation}
With this notation at hand, we derive error estimates for the Ritz-Galerkin projector~$\Pitildenabla$
defined in~\eqref{RG-projection-bulk}--\eqref{RG-projection-zero-average}.

\begin{lemma}\label{v-pikv2}
For all $v$ in~$H^{s+1}(\Omega)$,
$1 \le s \le k$, we have
\begin{equation}\label{v-pikv}
|v-\Pitildenabla v|_{1,h}
\lesssim h^{s} \Norm{v}_{s+1,\Omega}.
\end{equation}
\end{lemma}
\begin{proof}
According to the triangle inequality and the approximation properties of~$ \Pinabla$, we have
\[
\begin{aligned}
|v-\Pitildenabla v|_{1,\E}
&\leq |v - \Pinabla v|_{1,\E}+|\Pinabla v-\Pitildenabla v|_{1,\E}
\lesssim \hE^\s \SemiNorm{v}_{\s+1,\E}
+ |\Pinabla v - \Pitildenabla v|_{1,\E}.
\end{aligned}
\]
By definition of~$\Pitildenabla$ and~$\Pinabla$,
for all~$\qk$ in~$\Pbb_{k}(\E)$,
we can write
\[
\begin{aligned}
\int_\E \nabla \qk \cdot & \nabla(\Pinabla v -\Pitildenabla v)~ \mathrm{d}\E
= \int_\E \nabla \qk \cdot \nabla v~ \mathrm{d}\E
+\int_\E\Delta \qk v \ \mathrm{d}K
-\sum_{\e\in\EcalE}\int_\e\Pitildeze
(\nbfE \cdot \nabla \qk) v \ \mathrm{d}s\\
& = \sum_{\e\in\EcalE}\int_\e (I- \Pitildeze) (\nbfE \cdot \nabla \qk) v~\mathrm{d}s
= \sum_{\e\in\EcalEc}\int_\e 
(I- \Pitildeze) (\nbfE \cdot \nabla \qk) v ~\mathrm{d}s.
\end{aligned}
\]
If~$k=1$, we have by definition of~$\VhE$ that
$\Pinabla v = \Pitildenabla v$
and the right-hand side vanishes.
So, we focus on the case~$k\geq 2$. We set
\[\bm{\psi} 
:= \nabla(\Pinabla v - \Pitildenabla v)
\in [\Pbb_{k-1}(\E)]^2.
\]
We rewrite the above identities for the choice
$q_k = \Pinabla v-\Pitildenabla v$ and obtain
\begin{equation} \label{difference-Pinabla-Pitildenabla}
\begin{aligned}
& |\Pinabla v-\Pitildenabla v|_{1,\E}^2 
= \sum_{\e\in\EcalE_c}\int_\e (I- \Pitildeze)(\bm{\psi}\cdot\nbfE) v~\mathrm{d}s\\
& =\sum_{\e\in\EcalE_c}\int_\e(I-{\Pitildeze})(\bm{\psi}\cdot\nbfE) (I- \Pitildeze)v~\mathrm{d}s \\
& \leq\sum_{\e\in\EcalE_c}\|(I-{\Pitildeze})(\bm{\psi}\cdot\nbfE) \|_{0,\e}\|(I- \Pitildeze) v\|_{0,\e}.
\end{aligned}
\end{equation}
Let $\tilde{\bm{\psi}} : e \rightarrow \Rbb^2$ be the mapped polynomial associated to $\bm{\psi}_{|_\e}$,
i.e., $\tilde{\bm{\psi}} := \bm{\psi}\circ \zeta\circ\gamma^{-1}\in \Pbbtilde_{k-1}(\e)$.
Since, by the same observation in \eqref{tildepsi}, $\tilde{\bm{\psi}}\cdot \hat{\nbf} \in \Pbbtilde_{k-1}(\e)$, we deduce
\begin{equation}\label{psin}
\begin{aligned}
& \|(I-{\Pitildeze})(\bm{\psi}\cdot\nbfE)\|_{0,\e}
\leq \|\bm{\psi}\cdot \nbfE-\tilde{\bm{\psi}}\cdot \hat{\nbf}\|_{0,\e}+\| {\Pitildeze}(\bm{\psi}\cdot\nbfE) -\tilde{\bm{\psi}}\cdot \hat{\nbf} \|_{0,\e}\\
&\leq \|\bm{\psi}\cdot \nbfE-\tilde{\bm{\psi}}\cdot \hat{\nbf}\|_{0,\e}+\| {\Pitildeze}(\bm{\psi}\cdot \nbfE 
-\tilde{\bm{\psi}}\cdot \hat{\nbf}) \|_{0,\e}
\leq 2\|\bm{\psi}\cdot \nbfE -\tilde{\bm{\psi}}\cdot \hat{\nbf}\|_{0,\e}\\
&\lesssim\he^{\frac12}
 \max_{\bm{y}\in e}|\bm{\psi}(\bm{y})\cdot\nbf(\bm{y}) 
-\tilde{\bm{\psi}}(\bm{y})\cdot \hat{\nbf}|\\
&\lesssim \he^{\frac12} (\max_{\bm{y}\in e} |\tilde{\bm{\psi}}(\bm{y}) \cdot (\nbf(\bm{y}) -\hat{\nbf} )|
+ \max_{\bm{y}\in e}|(\bm{\psi}(\bm{y})
-\tilde{\bm{\psi}} (\bm{y}))\cdot\nbf(\bm{y}) |).
\end{aligned}
\end{equation}
By~\eqref{O(h)}, we have the following estimate for the first term on the right-hand side of~\eqref{psin}:
\begin{equation}\label{O(h)2}
\begin{aligned}
\max_{\bm{y}\in e}|\tilde{\bm{\psi}}(\bm{y})\cdot(\nbf(\bm{y}) -\hat{\nbf} )|
\lesssim \hE\max_{t\in \Ie}|\tilde{\bm{\psi}}(\bm{y}(t))|
\lesssim \hE\max_{t\in \Ie}|\bm{\psi}(\bm{x}(t))|
\lesssim \hE\|{\bm{\psi}}\|_{\infty,\E}.
\end{aligned}
\end{equation}
As for the second term on the right-hand side of~\eqref{psin}, we use~\eqref{O(h2)}:
\begin{equation}\label{O(h2)2}
\begin{aligned}
 \max_{\bm{y}\in e}|(\bm{\psi}(\bm{y})-\tilde{\bm{\psi}}(\bm{y}))\cdot\nbf(\bm{y})|
&= \max_{t\in \Ie}|\bm{\psi}(\bm{y}(t))-\tilde{\bm{\psi}}(\bm{y}(t))|\\
&=\max_{t\in \Ie}|\bm{\psi}(\bm{y}(t))-\bm{\psi}(\bm{x}(t))|
\lesssim h_\E^{\frac32}\|\bm{\psi}\|_{W^{1,\infty}(K)}.
\end{aligned}
\end{equation}
Collecting~\eqref{O(h)2} and~\eqref{O(h2)2} in~\eqref{psin}
and using a polynomial inverse estimate,
we obtain
\begin{equation}\label{psin2}
\|(I- {\Pitildeze})(\bm{\psi}\cdot\nbfE)\|_{0,\e} 
\lesssim \hE^{\frac{3}{2}} \|{\bm{\psi}}\|_{\infty,\E}
        +h_\E^{\frac52}\|\bm{\psi}\|_{W^{1,\infty}(K)} 
\lesssim \hE^{\frac12} \|{\bm{\psi}}\|_{0,\E}.
\end{equation}
On the other hand, the trace inequality for polynomial
and Lemma~\ref{lemma:proes} with~$n=k-1$ imply
\begin{equation} \label{second-bound-Pitildenabla-estimates}
\| (I- {\Pitildeze})v\|_{0,\e}
\lesssim \hE^{\s-\frac12}\Norm{v}_{\s-\frac12,\e}.
\end{equation}
Collecting~\eqref{psin2} and~\eqref{second-bound-Pitildenabla-estimates}
in~\eqref{difference-Pinabla-Pitildenabla},
we deduce
\[
|v-\Pitildenabla v|_{1,h}
\lesssim h^\s \left((\Norm{v}_{s+1,\Omega}
                +\sum_{i=1}^N\Norm{v}_{\s-\frac12,\Gamma_i}\right).
\]
Recalling the Stein's extension operator~$E$ in Lemma~\ref{lemma:Stein}
with stability properties as in~\eqref{Extension}
and proceeding along the same lines of Lemma~\ref{v-vpk},
we obtain
\[
\Norm{v}_{\s-\frac12,\Gamma_i}
\leq\|Ev\|_{\s-\frac12,\partial\mathcal{C}_i}
\lesssim \|Ev\|_{\s,\mathcal{C}_i}
 \leq \|Ev\|_{\s,\bbR^d}
 \lesssim \Norm{v}_{\s,\Omega}.
\]
The assertion follows combining the two estimates above.
\end{proof}

\begin{remark}
Lemma \ref{v-pikv2} plays a critical role in the optimality with respect to~$k$ of the convergence estimates
detailed in Theorems~\ref{theorem:H1} and~\ref{theorem:L2} below.
A more direct approach
would not exploit that~$e$ converges to~$\widehat{e}$ as $h \rightarrow 0$;
the ensuing approximation estimate would either be sub-optimal
(order $h^{k-1/2}$ in the $H^1$ norm)
or require a higher order polynomial degree on curved edges
in the definition of the local spaces
thus leading to a more computationally expensive scheme.
\end{remark}

Next, we provide results that are instrumental to derive optimal convergence of the method;
see Theorem~\ref{theorem:H1}.
We begin by recalling the discretization error on the right-hand side;
see, e.g., \cite[Section~$4.7$]{BeiradaVeiga-Brezzi-Cangiani-Manzini-Marini-Russo:2013}.

\begin{lemma}\label{esf}
Let the right-hand side~$f$ of~\eqref{model}
belong to~$H^{s-1}(\Omega)$, $\ell \le s \le k+1$,
where~$\ell=2$ if~$k=1$, $\ell=1$ if~$k \ge 2$,
and the discrete right-hand side~$\fh$ be as in~\eqref{fk1}--\eqref{fk2}.
Then, for any~$\vh$ in~$\Vh$, we have
\[
| (f- \fh, \vh )_{0,\E}|\lesssim 
\begin{cases}
h \SemiNorm{f}_{1,\Omega} \SemiNorm{\vh}_{1,h}
&  \text{for } k=1,\\
h^{s} \SemiNorm{f}_{s-1,\Omega}\SemiNorm{\vh}_{1,h}
& \text{for } k>1.
\end{cases}
\]
\end{lemma}

Next, we introduce a bilinear form that will allow us to take into account the nonconformity of the method in the error analysis,
namely $\Ncalh: H^{\frac32+\varepsilon}(\Omega) \times H^{1,nc}(\taun,k) \to \Rbb$
defined by
\[
\Ncalh(u,\vh)
:= \sum_{\E \in \taun}
    \int_{\partial \E} (\nbfE \cdot \nabla u) \vh~\mathrm{d}s
=  \sum_{\e \in \Ecalh} \int_\e \nabla u \cdot \jump{\vh}~\mathrm{d}s.
\]
If~$u$, the solution to~\eqref{model}, belongs to~$H^{\frac32+\varepsilon}(\Omega)$
and is such that~$\Delta u$ belongs to~$L^2(\E)$ for all elements~$\E$,
an integration by parts implies that
\begin{equation}\label{eq:007}
a(u,v) 
= (f,\vh)+\Ncalh(u,\vh)
\qquad \forall \vh \in\Vh \subset H^{1,nc}(\taun,k) .
\end{equation}
In the following result,
we cope with the estimate of the term related to the nonconformity of the scheme.
\begin{lemma} \label{lemma:nc-term}
Let $u$, the solution to~\eqref{model} belong to~$H^{s+1}(\Omega)$,
$1/2 < s \le k$.
Then, for all~$\vh$ in~$\Vh$, we have
\[
|\Ncalh(u,\vh)|
\lesssim h^{s} \Norm{u}_{s+1,\Omega} \SemiNorm{v_h}_{1,h}.
\]
\end{lemma}
\begin{proof}
The proof follows along the same line as those in~\cite[Lemma 4.1]{AyusodeDios-Lipnikov-Manzini:2016}. 
We briefly report it here for the sake of completeness. 

From the definitions of the space $H^{1,nc}(\taun,k)$ and the operator $\Pitilde_{k-1}^{0,\e}$, 
recalling that $\Pbb_0(e) \subset \widetilde{\Pbb}_{k-1}(e)$,
finally using the Cauchy–Schwarz inequality, we find
\[
\begin{aligned}
\Ncalh(u,\vh)
&=  \sum_{\e \in \Ecalh} \int_\e
(\nabla u - \Pitildeze \nabla u) \cdot \jump{\vh}~\mathrm{d}s
=  \sum_{\e \in \Ecalh} \int_\e 
(\nabla u -\Pitildeze \nabla u) \cdot (\jump{\vh}- [\![ \Pi_0^{0} \vh ]\!])~\mathrm{d}s\\
&=  \sum_{\e \in \Ecalh} 
\|\nabla u - \Pitildeze \nabla u\|_{0,\e} \|\jump{\vh} - [\![ \Pi_0^{0} \vh ]\!] \|_{0,\e} ,
\end{aligned}
\]
where $\Pi^0_0$ denotes as usual the $L^2$ projection on $\taun$-piecewise constant functions.
Using Lemma~\ref{lemma:proes} and the Poincar\'e inequality,
for each internal edge $e=\partial \E^+\cup\partial\E^-$, we get
\[
\|\nabla u -\Pitildeze\nabla u\|_{0,\e}
\lesssim h^{\s-\frac12}\| u\|_{\s+\frac12,\e}
    \lesssim h^{\s-\frac12}\| u\|_{\s+1,\E^+\cup\E^-}
\]
and
\[
\begin{aligned}
\|\jump{\vh}-  [\![ \Pi_0^{0} \vh ]\!] \|_{0,\e}
& \lesssim h^{-\frac12}\| \vh-  \Pi_0^{0} \vh \|_{0,\E^+\cup\E^-}
+ h^{\frac12}| \vh-\Pi_0^{0} \vh |_{1,\E^+\cup\E^-}
\lesssim h^{\frac12}|\vh|_{1,\E^+\cup\E^-}.
\end{aligned}
\]
An analogous estimate is valid for boundary edges.
The assertion follows summing over all elements.
\end{proof}

Next, we estimate from above a term measuring
the lack of polynomial consistency of the proposed method, i.e.,
the error between the bilinear functions~$a(\cdot,\cdot)$ and $\ah(\cdot,\cdot)$. 
\begin{lemma} \label{esa}
Let $u$ the solution to~\eqref{model} belong to~$H^{s+1}(\Omega)$,
$1\le s\le k$.
Then, we have
\[
|\ah(u,\vh)-a(u,\vh)|
\lesssim h^s \Norm{u}_{s+1,\Omega} \SemiNorm{\vh}_{1,h}
\qquad \forall \vh \in \Vh .
\]
\end{lemma}
\begin{proof}
For any element~$\E$, we have
\[
\begin{aligned}
\ahE(u,\vh)-\aE(u,\vh) 
& = \aE(\Pitildenabla u,\Pitildenabla\vh)+\SE((I-\Pitildenabla )u,(I-\Pitildenabla )\vh)-\aE(u,\vh)\\
& = \aE(\Pitildenabla u,\Pitildenabla\vh-\vh)-\aE(u-\Pitildenabla u,\vh)\\
&\quad+\SE((I-\Pitildenabla )u,(I-\Pitildenabla )\vh)
=: I_1+I_2+I_3.
 \end{aligned}  
\]
As for the term~$I_1$,
we denote $\bar{v}_h = \vh-\Pi_0^0\vh$ and use the definition of the operator $\Pitildenabla$, then 
\[
\begin{aligned}
I_1
& = \aE(\Pitildenabla u,\bar{v}_h-\Pitildenabla\bar{v}_h) 
= \sum_{\e\in\EcalE_c} \int_\e 
(I- \Pitildeze) (\nbfE\cdot \nabla \Pitildenabla u) \bar{v}_h~\mathrm{d}s\\
& \leq \|\bar{v}_h\|_{0,\partial\E} \sum_{\e\in\EcalE_c}
\Norm{(I- \Pitildeze) (\nbfE\cdot \nabla \Pitildenabla u)}_{0,\e}.
\end{aligned}
\]
Using the approximation of the operator $\Pitildeze$,
the Poincar\'e inequality,
and the trace inequality,
we infer, for~$\upi$ as in~\eqref{L2-projection-bulk},
the two inequalities
 \begin{equation}\label{I1}
\begin{aligned}
& \hE^{\frac12}
\Norm{(I- \Pitildeze) (\nbfE\cdot \nabla \Pitildenabla u)
 }_{0,\e}\\
& \lesssim \hE^{\frac12}
\Norm{(I- \Pitildeze) (\nbfE\cdot \nabla u_\pi)}_{0,\e}
+\hE^{\frac12}
\Norm{(I- \Pitildeze) (\nbfE\cdot \nabla (\Pitildenabla u-u_\pi))}_{0,\e}\\
& \lesssim \hE^{s} 
\Norm{\nbfE\cdot \nabla u_\pi}_{s-\frac12,\e}
+ \hE^{\frac12} \Norm{\nbfE\cdot \nabla (\Pitildenabla u - u_\pi)}_{0,\e}\\
&\lesssim \hE^\s\|\nabla u_\pi\|_{\s-\frac12,\e} 
    +\hE^{\frac12}\|\nabla( u_\pi-\Pitildenabla u)\|_{0,\e}
\lesssim \hE^\s\|u_\pi\|_{\s+1,\E}
    +|u_\pi-\Pitildenabla u|_{1,\E}\\
&\lesssim \hE^\s\Norm{u}_{\s+1,\E}
        +\hE^\s \|u-u_\pi\|_{\s+1,\E}
        +|u_\pi-u|_{1,\E}
        +\SemiNorm{u-\Pitildenabla u}_{1,\E}\\
&\lesssim \hE^\s \Norm{u}_{\s+1,\E}
        +\SemiNorm{u-\Pitildenabla u}_{1,\E}
\end{aligned}
\end{equation}
and
\[
\hE^{-\frac12}\|\bar{v}_h\|_{0,\partial \E}
\lesssim \hE^{-1}\|\bar{v}_h\|_{0,\E}+|\bar{v}_h|_{1,\E}
\lesssim\SemiNorm{\vh}_{1,\E}.
\]
As for the term~$I_2$,
by the continuity~\eqref{tpikbd} and approximation properties~\eqref{v-pikv} of~$\Pitildenabla$, we deduce
\[
I_2
\leq \SemiNorm{u-\Pitildenabla u}_{1,\E} \SemiNorm{\vh}_{1,\E}.
\]
We handle the term~$I_3$
using Proposition~\ref{proposition:stability-SE},
the Poincar\'e inequality,
the approximation properties~\eqref{v-pikv} of~$\Pitildenabla$,
and Lemma~\ref{PTcont}:
$$
\begin{aligned}
I_3
&\lesssim\big{(} \hE^{-2}\|(I-\Pitildenabla)u\|_{0,\E}^2
+\SemiNorm{(I-\Pitildenabla)u}_{1,\E}^2\big{)}^{\frac12}\\
& \quad\ {(} \hE^{-2}\|(I-\Pitildenabla)\vh\|_{0,\E}^2
+|(I-\Pitildenabla)\vh|_{1,\E}^2\big{)}^{\frac12} \\
& \lesssim \SemiNorm{(I-\Pitildenabla)u}_{1,\E} |(I-\Pitildenabla)\vh|_{1,\E}
\lesssim |u-\Pitildenabla u|_{1,\E} \SemiNorm{\vh}_{1,\E}.
\end{aligned}$$
The assertion follows collecting the above estimates,
summing over the elements,
and using the approximation properties of~$\Pitildenabla$ in Lemma~\ref{v-pikv2}.
\end{proof}

We are now in a position to state the abstract error analysis in the energy norm for method~\eqref{VEM}
and deduce optimal error estimates.

 \begin{theorem}\label{theorem:H1} 
Let $u$ and~$\uh$ be the solutions to~\eqref{model} and~\eqref{VEM}, respectively.
Then, for every $\uI$ in~$\Vh$, we have
\begin{equation}\label{Strang}
 \SemiNorm{u-\uh}_{1,h}
 \lesssim |u-\uI|_{1,h}
          +\sup_{\vh\in \Vh}
          \left(\frac{| (f-\fh,\vh)_{0,\Omega}|}{\SemiNorm{\vh}_{1,h}}
          +\frac{\Ncalh(u,\vh)}{\SemiNorm{\vh}_{1,h}}
          +\frac{|\ah(u,\vh)-a(u,\vh)|}{\SemiNorm{\vh}_{1,h}} \right).
\end{equation}
Furthermore, if~$u \in H^{s+1}(\Omega)$ and~$f \in H^{s-1}(\Omega)$, with $1 \le s \le k$,
we also have
\begin{equation} \label{error-estimates-H1}
\SemiNorm{u-\uh}_{1,h}
\lesssim h^{s}(\Norm{u}_{s+1,\Omega}
+\SemiNorm{f}_{s-1,\Omega}).
\end{equation}
\end{theorem}
\begin{proof}
We first prove the Strang-type estimate~\eqref{Strang}.
We split~$u-\uh$ as~$(u-\uI )+(\uI-\uh)$,
use the triangle inequality, and get
\[
\SemiNorm{u-\uh}_{1,h}
\leq |u-\uI|_{1,h}+|\uI-\uh|_{1,h}.
\]
Set $e_h := \uh-\uI$.
Following the lines of \cite[Theorem~$4.3$]{AyusodeDios-Lipnikov-Manzini:2016} and recalling \eqref{eq:007},
the coercivity of~$\ah(\cdot,\cdot)$ allows us to write
\[
\begin{aligned}
|e_h|_{1,h}^2
&\lesssim \ah(e_h,e_h)= \ah(\uh,e_h)-\ah(\uI,e_h)\\
& = (\fh- f,e_h)_{0,\Omega} - \Ncalh(u,e_h)
-\sum_{\E\in\taun}\ahE(\uI-u,e_h)
-\sum_{\E\in\taun}\ahE(u,e_h)
+\sum_{\E\in\taun}\aE(u,e_h) .
\end{aligned}
\]
The proof of~\eqref{Strang} follows by standard manipulations of the right-hand side above.
The error estimates~\eqref{error-estimates-H1} finally follow bounding the terms
on the right-hand side of~\eqref{Strang} by means of
Lemmas~\ref{lemma:DoFs-interpolant-estimates}, \ref{esf}, \ref{lemma:nc-term},
and~\ref{esa}.
\end{proof}

\subsection{Convergence analysis in a weaker norm} \label{subsection:L2-error}
In this section, we prove $L^2$ error estimates for~\eqref{VEM} based on extra assumptions on~$\Omega$.
Also in this case, we follow the guidelines of the nonconforming element method error analysis;
see \cite[Section~$4.1$]{AyusodeDios-Lipnikov-Manzini:2016}.
We focus on the case~$k\ge3$,
and discuss the cases~$k=1$ and~$2$ in Remark~\ref{remark:k12} below.

\begin{theorem}\label{theorem:L2}
Let~$k\ge3$.
Assume that~$u$ and~$f$,
the solution and right-hand side of~\eqref{model},
belong to~$H^{s+1}(\Omega)$,
and~$H^{s-1}(\Omega)$, $1 \le s \le k$.
Let~$\uh$ be the solution to~\eqref{VEM}.
If~$\Omega$ is convex, then we have
\[
\begin{split}
\|u-\uh\|_{0,\Omega}
& \lesssim h^{s+1}(\Norm{u}_{s+1,\Omega}
+\SemiNorm{f}_{s-1,\Omega}).
\end{split}
\]
\end{theorem}
\begin{proof}
Consider the dual problem
\[
\begin{cases}
-\Delta \phi = u-\uh & \text{in } \Omega,\\
\phi         = 0     & \text{on } \partial\Omega.
\end{cases}
\]
The convexity of~$\Omega$ entails that $\phi\in H^2(\Omega)\cap H^1_0(\Omega)$ is the unique solution to the above problem
and satisfies the elliptic regularity estimates
\[
\|\phi\|_{2,\Omega}\lesssim \|u-\uh\|_{0,\Omega}.
\]
Proceeding as in \cite[Theorem~$4.5$]{AyusodeDios-Lipnikov-Manzini:2016}
using that~$k\ge3$,
and taking~$\phiI$ in~$\Vh$ as the DoFs interpolant of~$\phi$ in~\eqref{vik},
we obtain
\[
\begin{aligned}
\|u-\uh\|_{0,\Omega}^2
&= \sum_{\E\in\taun}\aE(\phi - \phiI,u-\uh)
    +\Ncalh(\phi,u-\uh)+\Ncalh(u,\phiI)\\
&\quad +  (f-\fh, \phiI)_{0,\Omega}
    + \sum_{\E\in\taun}(\ahE(\uh,\phiI)-\aE(\uh,\phiI))\\
&\lesssim (h\SemiNorm{u-\uh}_{1,h} +h^{s+1})\Norm{u}_{s+1,\Omega} 
                \|u-\uh\|_{0,\Omega}\\
&\quad + h^2 \|f-\fh\|_{0,\Omega} \|u-\uh\|_{0,\Omega}
 + \sum_{\E\in\taun}(\ahE(\uh,\phiI)-\aE(\uh,\phiI)).
\end{aligned}
\]
For any element~$\E$, we have
\[
\begin{aligned}
& \ahE(\uh,\phiI)-\aE(\uh,\phiI) \\
& = \aE(\Pitildenabla \uh,\Pitildenabla\phiI)-\aE(\uh,\phiI)
 +\SE((I-\Pitildenabla )\uh,(I-\Pitildenabla )\phiI)\\
&= \aE(\Pitildenabla\uh,\Pitildenabla\phiI-\phiI)
-\aE(\uh-\Pitildenabla\uh,\Pitildenabla\phiI)\\
&\quad-\aE(\uh-\Pitildenabla\uh,\phiI-\Pitildenabla\phiI)
+\SE((I-\Pitildenabla )\uh,(I-\Pitildenabla )\phiI)
=: I_1 + I_2 + I_3 + I_4.
 \end{aligned}
 \]
As for the term~$I_1$,
we use the definition of~$\Pitildenabla$,
see~\eqref{RG-projection-bulk}--\eqref{RG-projection-zero-average},
and obtain
\[
\begin{aligned}
I_1
& = \sum_{\E\in\taun} \sum_{\e\in\EcalEc} 
\int_\e (I- \Pitildeze) (\nbfE\cdot \nabla \Pitildenabla \uh) \phiI~ \mathrm{d}s \\
& = \sum_{\E\in\taun} \sum_{\e\in\EcalEc} \int_\e
(I- \Pitildeze) (\nbfE\cdot \nabla \Pitildenabla \uh) \ (I- \Pitildeze) \phiI~ \mathrm{d}s\\
& \leq \sum_{\E\in\taun} \sum_{\e\in\EcalEc}
\Norm{(I- \Pitildeze) (\nbfE\cdot \nabla \Pitildenabla \uh) }_{0,\e}
\Norm{(I- \Pitildeze)\phiI}_{0,\e}.
\end{aligned}
\]
Similarly with the inequality \eqref{I1}, we then infer
\[
\begin{aligned}
& \hE^{\frac12} \Norm{(I- \Pitildeze)
(\nbfE\cdot \nabla (\Pitildenabla \uh)) }_{0,\e}\\
& \leq \hE^{\frac12}
\Norm{(I- \Pitildeze) (\nbfE\cdot \nabla \Pitildenabla u)
}_{0,\e}
+ \hE^{\frac12} 
\Norm{(I- \Pitildeze) ( \nbfE\cdot \nabla (\Pitildenabla u - \uh) )}_{0,\e}\\
&\lesssim \hE^{s} \Norm{u}_{\s+1,\E}
        + \hE^{\frac12}\|\nabla \Pitildenabla(u-\uh)\|_{0,\e}
\lesssim \hE^{s} \Norm{u}_{\s+1,\E}
        + | \Pitildenabla(u-\uh) |_{1,\E}\\
&\lesssim \hE^{s} \Norm{u}_{\s+1,\E}
        + |u- \Pitildenabla u|_{1,\E}
        + |\Pitildenabla(u-\uh)|_{1,\E}\\
& \lesssim \hE^{s} \Norm{u}_{\s+1,\E}
        + |u- \Pitildenabla u|_{1,\E}
        + \SemiNorm{u-\uh}_{1,\E},
\end{aligned}
\]
and
\[
\begin{aligned}
& \hE^{-\frac12}\|(I- \Pitildeze)\phiI\|_{0,\e}
\leq \hE^{-\frac12}\|(I- \Pitildeze)(\phiI-\phi)\|_{0,\e}+\hE^{-\frac12}\|(I- \Pitildeze)\phi\|_{0,\e}\\
& \lesssim \hE^{-1} \|\phiI-\phi\|_{0,\E}+|\phiI-\phi|_{1,\E}+\hE\|\phi\|_{\frac{3}{2},\e}
\lesssim |\phiI-\phi|_{1,\E}+\hE\|\phi\|_{\frac{3}{2},\e}.
\end{aligned}
\]
As for the term~$I_2$,
we set $\bm{\psi} := \nabla \Pitildenabla\phiI - \nabla\phiI$
and $\bar{u}_h := \uh-\Pi_0^0\uh$,
and arrive at
\[
\begin{aligned}
& I_2
= - \aE(\bar{u}_h-\Pitildenabla\bar{u}_h,\Pitildenabla\phiI)
= - \sum_{\e\in\EcalEc}\int_{\e}
(I-\Pitildeze) (\nbfE\cdot \nabla \Pitildenabla \phiI) \bar{u}_h ~\mathrm{d}s \\
& = - \sum_{\e\in\EcalEc} \int_{\e}
(I-\Pitildeze) (\nbfE\cdot \nabla (\Pitildenabla \phiI-\phiI)) \bar{u}_h~\mathrm{d}s\\
& = - \sum_{\e\in\EcalEc}\int_{\e}(I-\Pitildeze) (\bm{\psi} \cdot \nbfE) (I-\Pitildeze)\bar{u}_h~\mathrm{d}s
\leq \sum_{\e\in\EcalEc}\|(I-\Pitildeze) (\bm{\psi}\cdot\nbfE)\|_{0,\e}\|(I-\Pitildeze)\bar{u}_h\|_{0,\e}.
\end{aligned}
\]
Inequality~\eqref{psin2} entails
\[
\begin{aligned}
\|(I-\Pitildeze) (\bm{\psi}\cdot\nbfE)\|_{0,\e}
& \lesssim \hE^{\frac12}\|\bm{\psi}\|_{0,\E}\\
& \lesssim \hE^{\frac12}| \Pitildenabla\phiI - \phiI|_{1,\E}
 \lesssim \hE^{\frac12}( |\phi-\phiI|_{1,\E}+|(I-\Pitildenabla)\phi|_{1,\E})
 \end{aligned}
\]
and
\[
\begin{aligned}
& \|(I-\Pitildeze)\bar{u}_h\|_{0,\e}\
\leq\|(I-\Pitildeze)(\bar{u}-\bar{u}_h)\|_{0,\e}
+ \|(I-\Pitildeze)\bar{u}\|_{0,\e}
\lesssim \|\bar{u}-\bar{u}_h\|_{0,\e}
        +\hE^{s} \|\bar{u}\|_{\s,\e}\\
&\lesssim \hE^{-\frac12}\|\bar{u}-\bar{u}_h\|_{0,\E}+\hE^{\frac12}|\bar{u}-\bar{u}_h|_{1,\E}
         +\hE^{s}\Norm{u}_{\s,\e}
\lesssim \hE^{\frac12}\SemiNorm{u-\uh}_{1,\E}
        +\hE^{s}\Norm{u}_{\s,\e}.
 \end{aligned}
\]
Furthermore, we have
$$
| I_3 |
\leq |\uh-\Pitildenabla\uh|_{1,h}|\phiI-\Pitildenabla\phiI|_{1,h}.
$$
Finally, we bound $I_4$:
$$
\begin{aligned}
|I_4|
&\lesssim\big{(} \hE^{-2}\|(I-\Pitildenabla)\uh\|_{0,\E}^2
+|(I-\Pitildenabla)\uh|_{1,\E}^2\big{)}^{\frac12}
{(} \hE^{-2}\|(I-\Pitildenabla)\phiI\|_{0,\E}^2
+|(I-\Pitildenabla)\phiI|_{1,\E}^2\big{)}^{\frac12} \\
&\lesssim
|\uh-\Pitildenabla\uh|_{1,\E}|\phiI-\Pitildenabla\phiI|_{1,\E}.
\end{aligned}
$$
Each of the above terms can be readily estimated
by adding and subtracting~$u$, $\Pitildenabla u$, $\phi$, and~$\Pitildenabla\phi$:
\[
\begin{aligned}
|\uh-\Pitildenabla\uh|_{1,\E}
&\lesssim \SemiNorm{u-\uh}_{1,\E}
    +\SemiNorm{(I-\Pitildenabla)u}_{1,\E}
    +|\Pitildenabla(u-\uh)|_{1,\E}\\
&\lesssim \SemiNorm{u-\uh}_{1,\E}
        +\SemiNorm{(I-\Pitildenabla)u}_{1,\E}
\end{aligned}
\]
and
\[
\begin{aligned}
|\phiI-\Pitildenabla\phiI|_{1,\E}
&\lesssim\ |\phi-\phiI|_{1,\E}
    +|(I-\Pitildenabla)\phi|_{1,\E}
    +|\Pitildenabla(\phi-\phiI)|_{1,\E}\\
&\lesssim \ |\phi-\phiI|_{1,\E}
            +|(I-\Pitildenabla)\phi|_{1,\E}.
\end{aligned}
\]
Combining the above estimates
and using Lemmas~\ref{lemma:DoFs-interpolant-estimates}, \ref{esf}, \ref{lemma:nc-term},
and~\ref{esa}, the assertion follows.
\end{proof}

\begin{remark} \label{remark:k12}
The proof of Theorem~\ref{theorem:L2}
does not cover the cases~$k=1$ and~$2$,
the reason being the presence of the term
\[
(f-\fh, \phiI)_{0,\Omega}.
\]
Following~\cite[Section~2.7]{BeiraodaVeiga-Brezzi-Marini:2013},
one can derive optimal convergence in the $L^2$ norm for~$k=1$
and one order suboptimal convergence for~$k=2$.
In this case, optimality can be recovered
by either a suitable modification of the right-hand side
(using better discretization of the test function in the discrete loading term)
or by adapting the enhanced version of the virtual element method.
\end{remark}

\section{The 3D version of the method}\label{section:3D}

In this section, we discuss the nonconforming virtual element method for 3D curved domain
and why the theoretical results are an extension of their two-dimensional counterparts.
Polyhedral meshes are now employed
and curved faces are the parametrizations of flat polygons.
The geometric assumptions \textbf{(G1)}--\textbf{(G2)}
in Section~\ref{section:meshes} are still required,
but are valid for~$\E$ being a polyhedron.
We further require
\begin{itemize}
\item[\textbf{(G3)}] every face~$F$ of~$\E$ (or, if the face is curved, the associated parametric polygon $\widehat{F}$) is star-shaped with respect to all the points of a disk radius larger than or equal to $\rho\hE$.
\end{itemize}
The local nonconforming virtual element space
on the (possibly curved) polyhedron $\E$ reads
\[
\VhE
:= \left\{ \vh \in H^1(\E) 
\middle| \ \Delta \vh \in \Pbb_{k-2}(\E),\
\nbfE \cdot \nabla \vh \in \Pbbtilde_{k-1}(F)\ 
\forall F \subset \partial K
\right\} ,
\]
where $\Pbbtilde_{k-1}(F)$ is the push forward on $F$ of $\Pbb_{k-1}(\widehat{F})$. 
The definition of~$\Vh$ is the natural extension of its 2D counterpart.
The degrees of freedom are given by scaled (possibly curved) face moments
with respect to (possibly mapped) polynomials up to order~$k-1$
and scaled bulk moments up to order~$k-2$.
The global nonconforming virtual element space is obtained
by a nonconforming coupling of the face degrees of freedom as in the 2D case:
\[
\Vh(\taun)
:=
\left\{ \vh \in H^{1,nc}(\taun,k) \middle|
\vh{}_{|\E} \in \VhE \; \forall \E\in\taun \right\},
\]
where $\llbracket \cdot \rrbracket_F$ is the jump across the (possibly curved) face and
\[
H^{1,nc}(\taun,k):=
\left\{
v\in H^1(\taun) \middle|  
\int_{F}\llbracket v \rrbracket_F\cdot\nbf_F q \ {\mathrm{d}} F,\;
\forall q\in \Pbbtilde_{k-1}(F),\ \forall F\subset\partial \E,~\E\in\taun
\right\}.
\]
No essential modifications of the 2D structure take place in 3D.
This is the major advantage of using the nonconforming version of the method.
The 3D version of the conforming VEM on curved domains in~\cite{BeiraodaVeiga-Russo-Vacca:2019}
should involve curved virtual element spaces on faces
that are currently unknown.

The construction of the local and global discrete bilinear forms in Section~\ref{subsection:polynomial-projections-bf}
directly extends to the 3D case.
The only difference resides in the design of the stabilizing term
that requires a different scaling.
It can be proved that the stabilization
\[
S^\E(u_h,\vh) 
:= \sum_{l=1}^{N_{\mathrm{dof}(K)}} \hE
\bm{D}^l(\uh)\bm{D}^l(\vh)\quad\quad\forall u_h,~\vh\in \VhE
\]
satisfies the 3D version of the estimates in~\eqref{stability-bounds}.
In fact, the key technical tools in stability analysis
are inverse inequalities, see Lemma \ref{lemma:inverse-L2Laplacian-H1};
the Neumann trace inequality, see Lemma~\ref{lemma:Neumann-trace-inequality}; 
polynomial inverse inequalities on the element boundaries, see Lemma \ref{lemma:inv012};
``direct estimates'' such as Poincar\'e-type and trace-type inequalities.
All such bounds are valid also in 3D.

The abstract error analysis is dealt with similarly to the 2D case; see Theorems~\ref{theorem:H1} and~\ref{theorem:L2}.
The only minor modification is in the definition of the nonconformity term
which in 3D is defined as
\[
\Ncalh(u,v) 
:= \sum_{F\in\Ecalh^3} \int_F \llbracket v\rrbracket_F\cdot\nabla u \ {\mathrm{d}} F,
\]
where~$\Ecalh^3$ denotes the set of (curved) faces in the polyhedral decomposition.
Thus, the proof of error estimates for the nonconforming term follows along the same lines as in the 2D case,
since \cite[Lemma 4.1]{AyusodeDios-Lipnikov-Manzini:2016}
is valid in arbitrary space dimension.

\section{Numerical experiments}\label{section:numerical-esperiments}

In this section, we present some numerical experiments,
so as to validate of Theorems~\ref{theorem:H1}
and~\ref{theorem:L2}.

Since the energy and $L^2$ errors are not computable,
we rather consider the computable error quantities:
\begin{equation} \label{computable-quantities}
\begin{aligned}
E_{H^1}
:= \frac{(\sum_{\E\in\taun}
| u- \Pitildenabla \uh|_{1,\E}^2)^{\frac12}}{|u|_{1,\Omega}},
\qquad 
E_{L^2}
:= \frac{(\sum_{\E\in\taun}
\|u- \Pitildenabla \uh\|_{0,\E}^2)^{\frac12}}{\Norm{u}_{0,\Omega}}.\\
\end{aligned}
\end{equation}
The two quantities above convergence with the same rate as the exact errors $\SemiNorm{u-\uh}_{1,h}$ and $\Norm{u-\uh}_{0,\Omega}$.

In Section~\ref{subsection:curved-boundary},
we consider two test cases on domains with curved boundary;
in Section~\ref{subsection:curved-interface},
we consider a test case with an internal curved interface and curved boundary.

\subsection{Curved boundary} \label{subsection:curved-boundary}

As a first test case,
we consider the domain $\Omega=\{(x,y)|x^2+y^2 < 1\}$, see Figure~\ref{figure:test-case-1} (left-panel)
and the exact (analytic) solution
\[
u_1(x,y) =\sin(\pi x)\cos(\pi y).
\]
The function~$u_1$ has inhomogeneous Dirichlet boundary conditions over~$\partial \Omega$.

\begin{figure}[H]
\begin{center}
\includegraphics[width=2.5in]{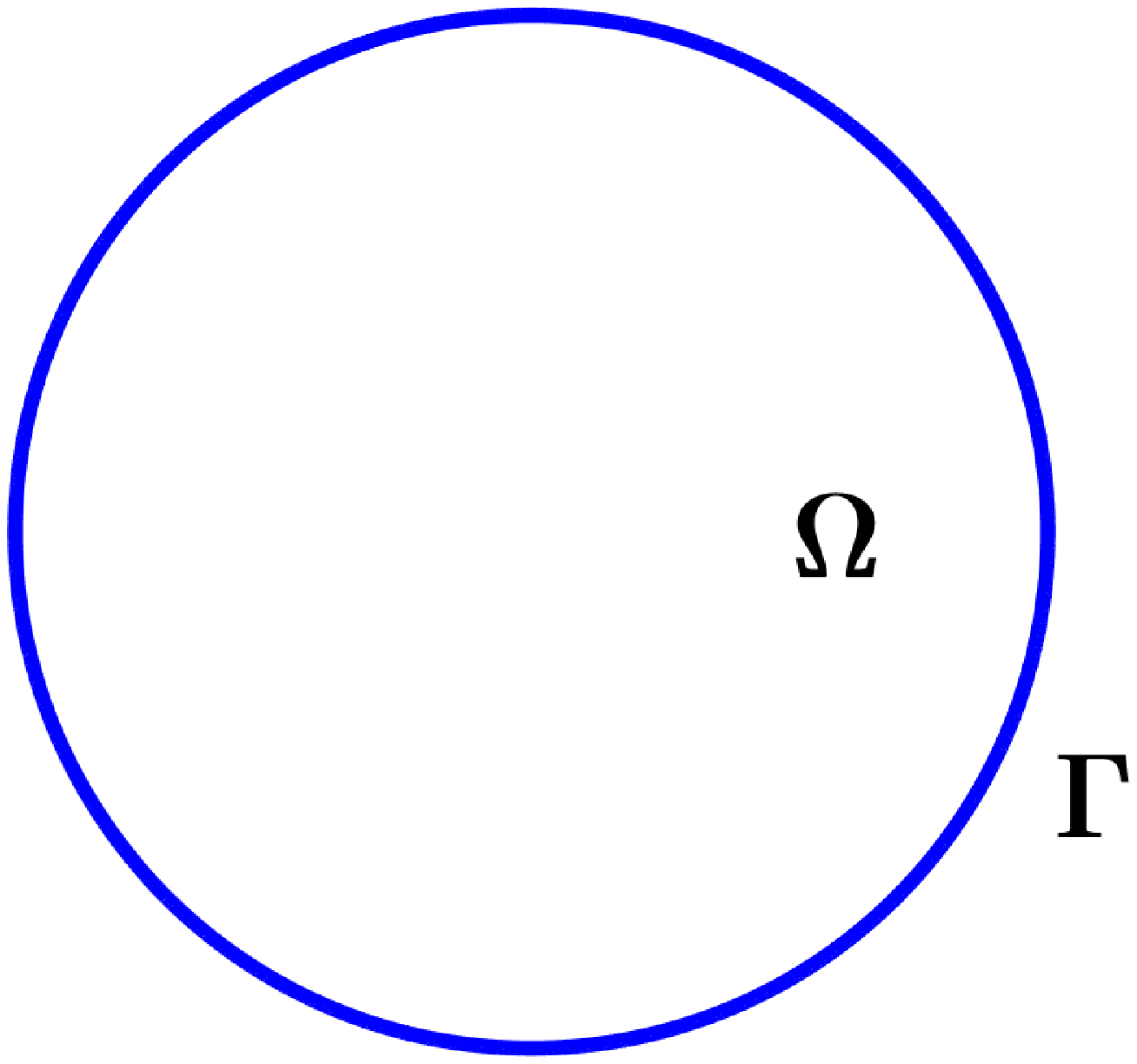} 
\includegraphics[width=2.5in]{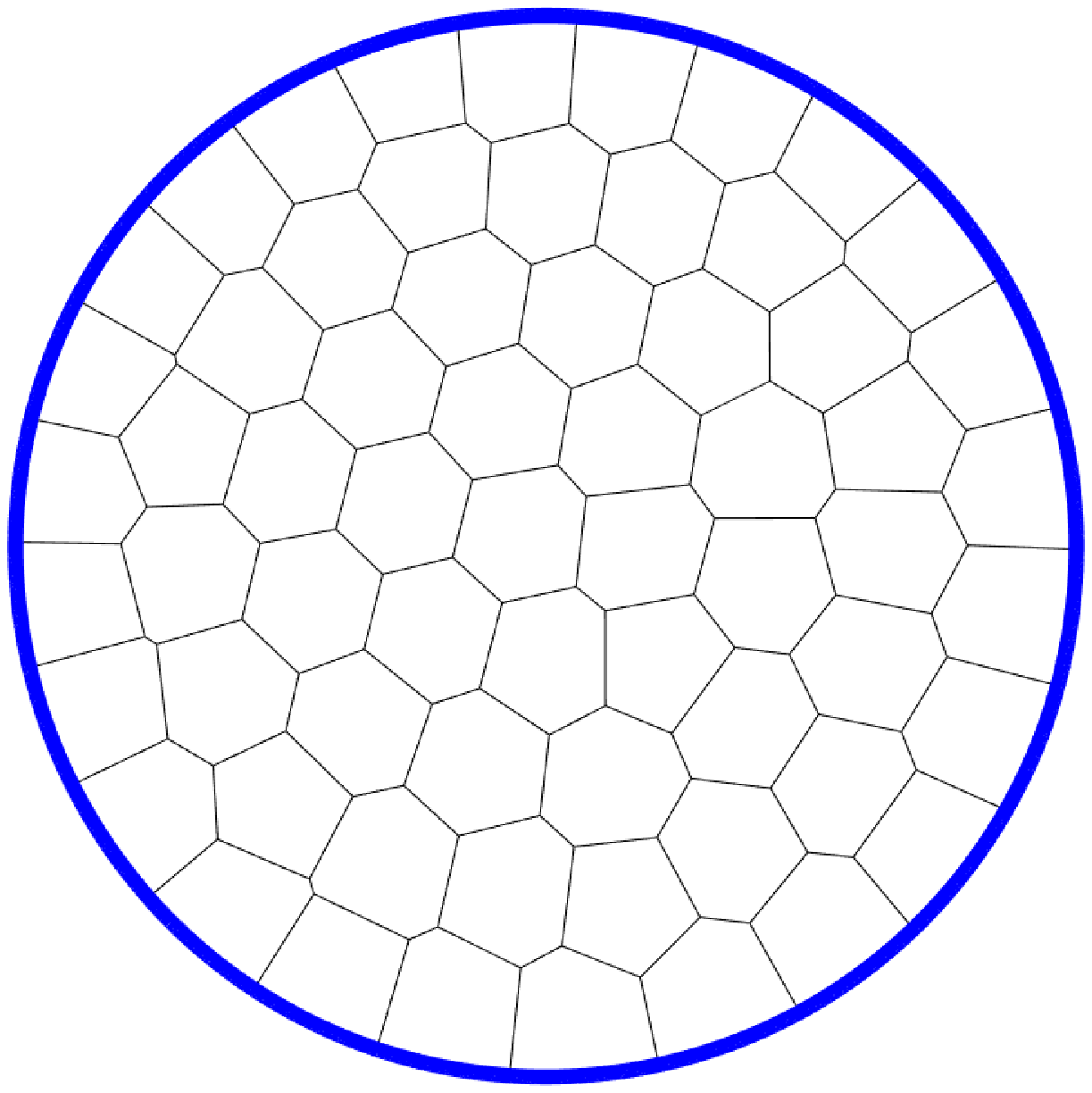}
\caption{\small{
Left-panel: a circular domain~$\Omega$.
Right-panel: an example of (curved) Voronoi mesh over~$\Omega$.}}
\label{figure:test-case-1}
\end{center}
\end{figure}

In Figure~\ref{figure:convergence-test-case-1},
we show the convergence of the two error quantities in~\eqref{computable-quantities} on the given sequence of Voronoi meshes with decreasing mesh size;
see Figure~\ref{figure:test-case-1} (right-panel) for a sample mesh.
We consider virtual elements of ``orders'' $k=2$, $3$, and~$4$.

\begin{figure}[H]
\begin{center}
\includegraphics[width=2.5in]{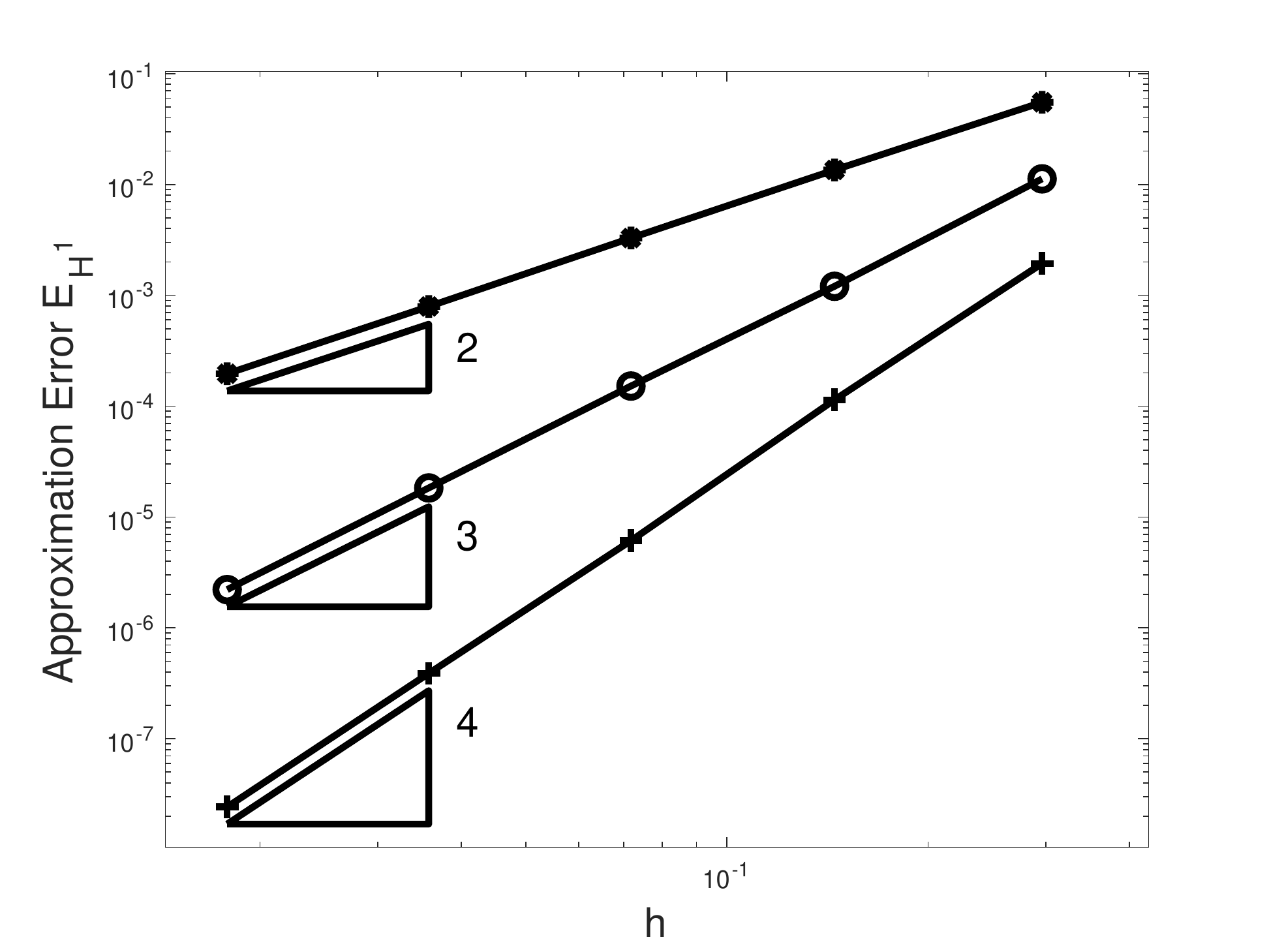} 
\includegraphics[width=2.5in]{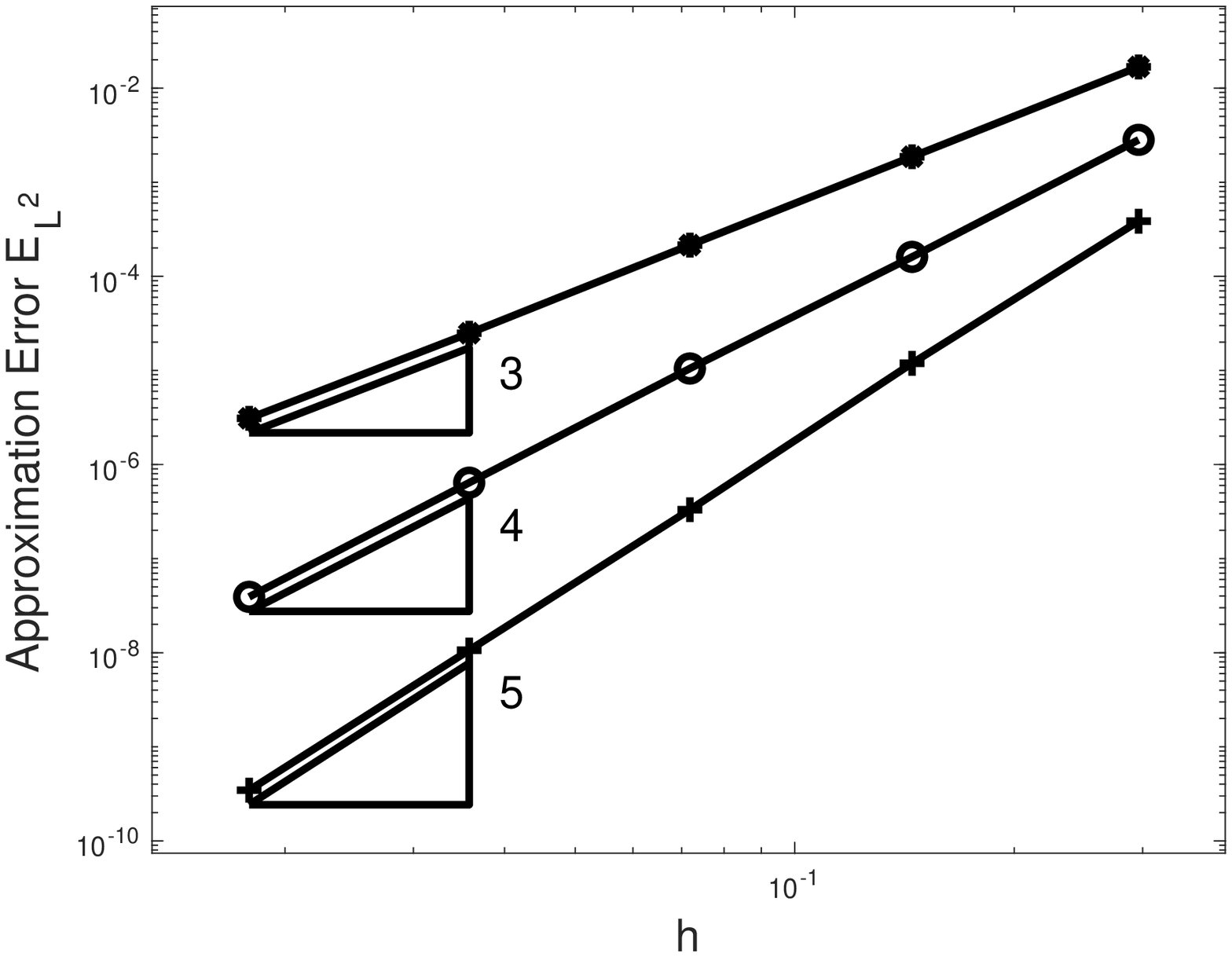}
\caption{\small{Left-panel: convergence of~$E_{H^1}$.
Right-panel: convergence of~$E_{L^2}$.
The exact solution is~$u_1$.
Voronoi meshes with decreasing mesh size are employed.
The ``orders'' of the virtual element spaces are
$k=2$, $3$, and~$4$.}}
\label{figure:convergence-test-case-1}
\end{center}
\end{figure}

As a second test case,
we consider the curved domain~$\Omega$
introduced in~\cite{BeiraodaVeiga-Russo-Vacca:2019}
and defined as
\begin{equation} \label{eqn:sindomain}
  \Omega:=\{(x, y) ~\text{s.t}~ 0 < x < 1,~\text{and}~ g_1(x) < y < g_2(x)\},    
\end{equation}
where
\[
g_1(x):= \frac{1}{20}\sin(\pi x)\quad \text{and}\quad g_2(x):= 1 + \frac{1}{20}\sin(3\pi x).
\]
We represent the domain in Figure~\ref{figure:test-case-2domain}.
On such an~$\Omega$, we consider
the exact (analytic) solution
\[
u_2(x,y) =-(y-g_1(x))(y-g_2(x))(1-x)x(3+\sin(5x)\sin(7y)).
\]
The function~$u_2$ has homogeneous Dirichlet boundary conditions over~$\partial \Omega$.

\begin{figure}[H]
\begin{center}
\includegraphics[width=2.5in]{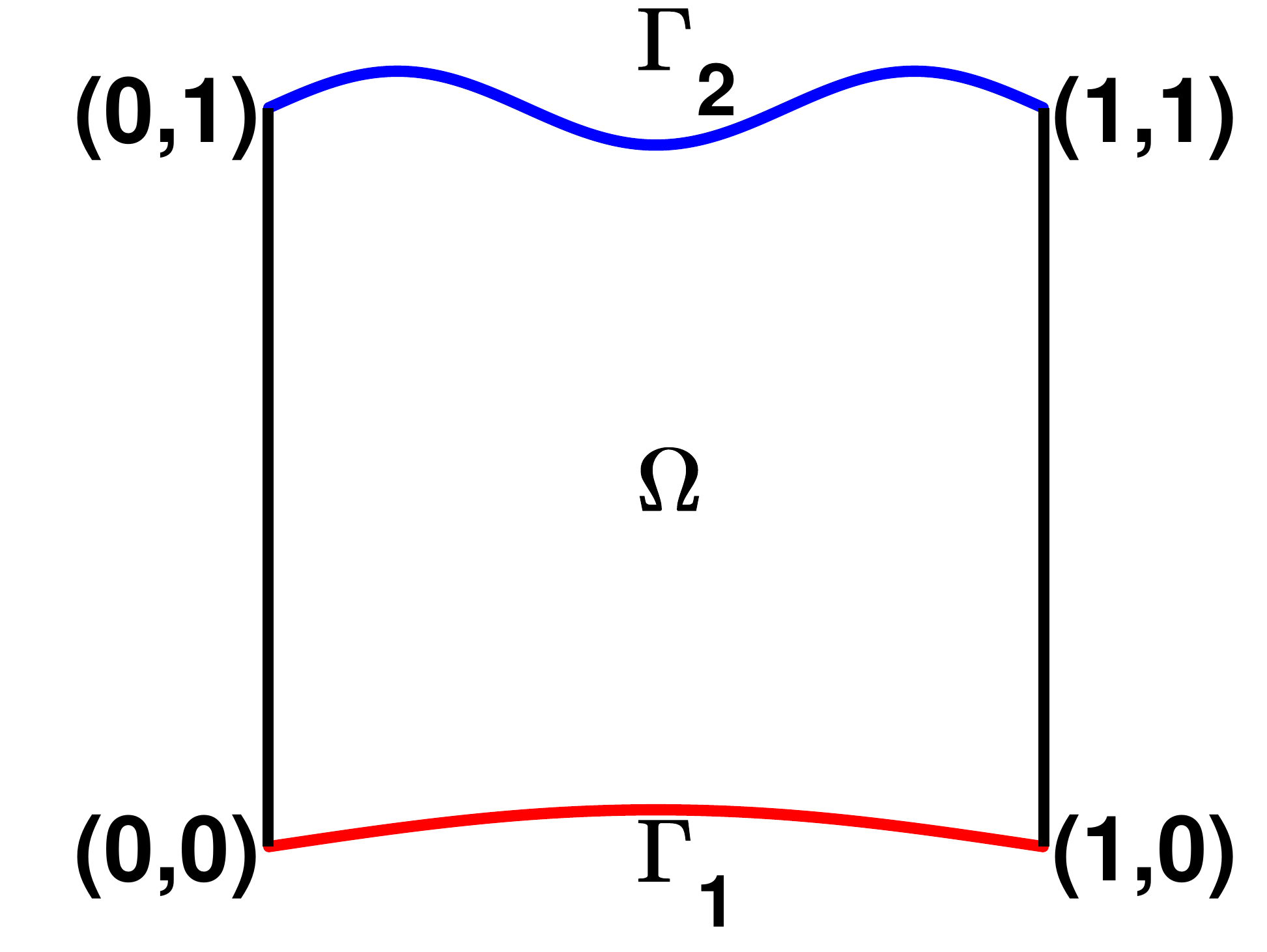}
\caption{Domain~$\Omega$ described in \eqref{eqn:sindomain}.}
\label{figure:test-case-2domain}
\end{center}
\end{figure}
\begin{figure}[H]
\begin{center}
\includegraphics[width=2.5in]{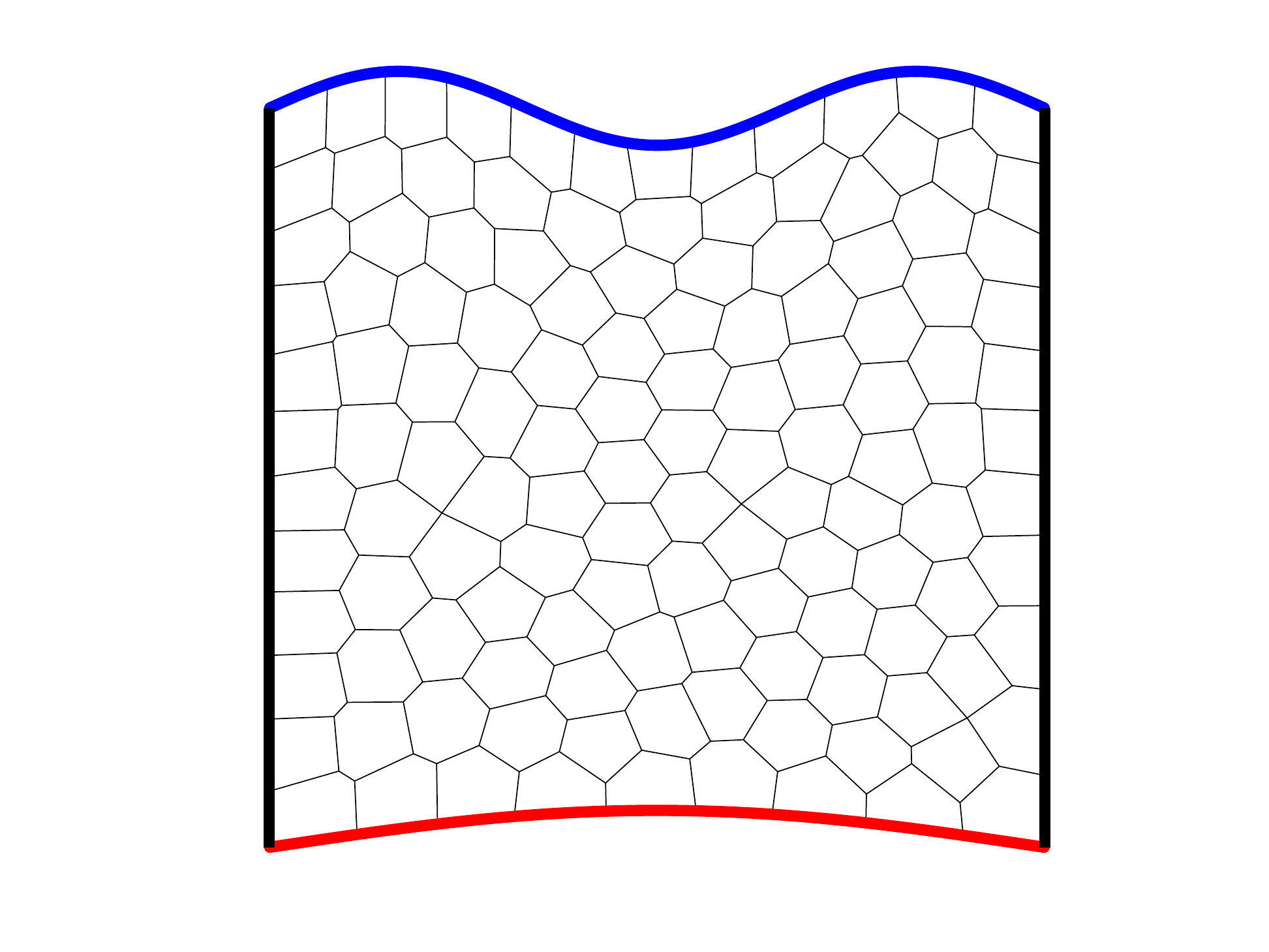}
\includegraphics[width=2.5in]{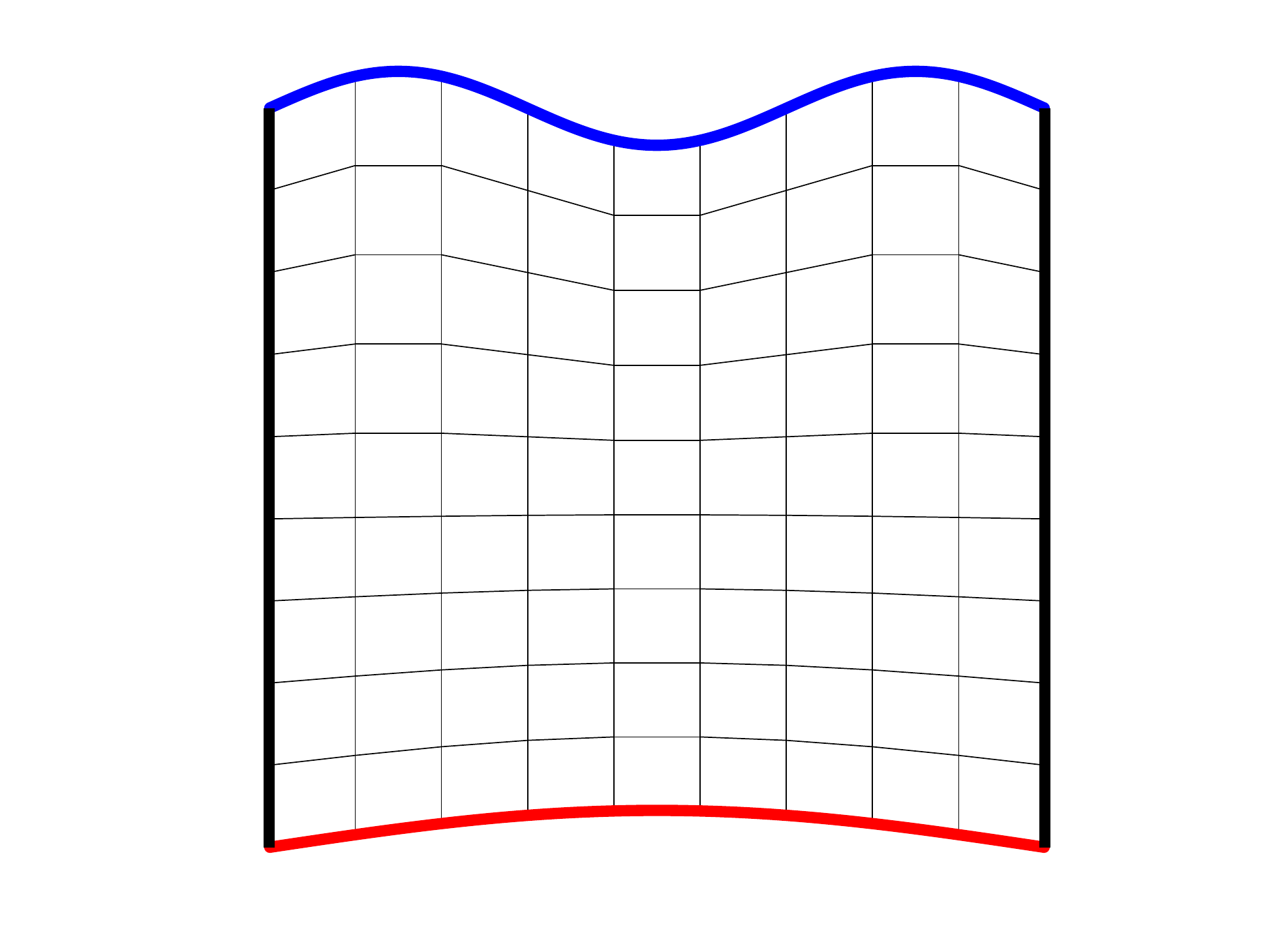}
\caption{Left-panel: an example of (curved) Voronoi mesh over~$\Omega$.
Right-panel: an example of (curved) quadrilateral mesh over~$\Omega$.}
\label{figure:test-case-2}
\end{center}
\end{figure}

The finite element partition on the curved domain~$\Omega$
is constructed starting from a mesh for the square $[0,1]^2$ and mapping the nodes accordingly to the following rule:
\[
(x_{\Omega},y_\Omega)=
\begin{cases}
(x_s,y_s+g_1(x_s)(1-2y_s)),   &\text{if}~ y_s \leq \frac12 ,\\
(x_s, 1-y_s+g_2(x_s)(2y_s-1), & \text{if}~ y_s \geq \frac12
\end{cases}
\]
Above, $(x_s, y_s)$
denotes the mesh generic node on the square domain~$(0,1)^2$, and~$(x_\Omega,y_\Omega)$
denotes the associated node in the curved domain~$\Omega$.
The edges on the curved boundary consist of
an arc of~$\Gamma_1$ or~$\Gamma_2$, while all the internal edges are straight. 
In Figure~\ref{figure:test-case-2},
we display two examples of meshes,
namely, a (curved) Voronoi and a (curved) square mesh.

In Figure~\ref{figure:convergence-test-case-2},
we show the convergence of the two error quantities in~\eqref{computable-quantities}
on the given sequences of meshes under uniform mesh refinements
for ``orders'' $k=2$, $3$, and~$4$.

\begin{figure}[H]
\begin{center}
\includegraphics[width=2.5in]{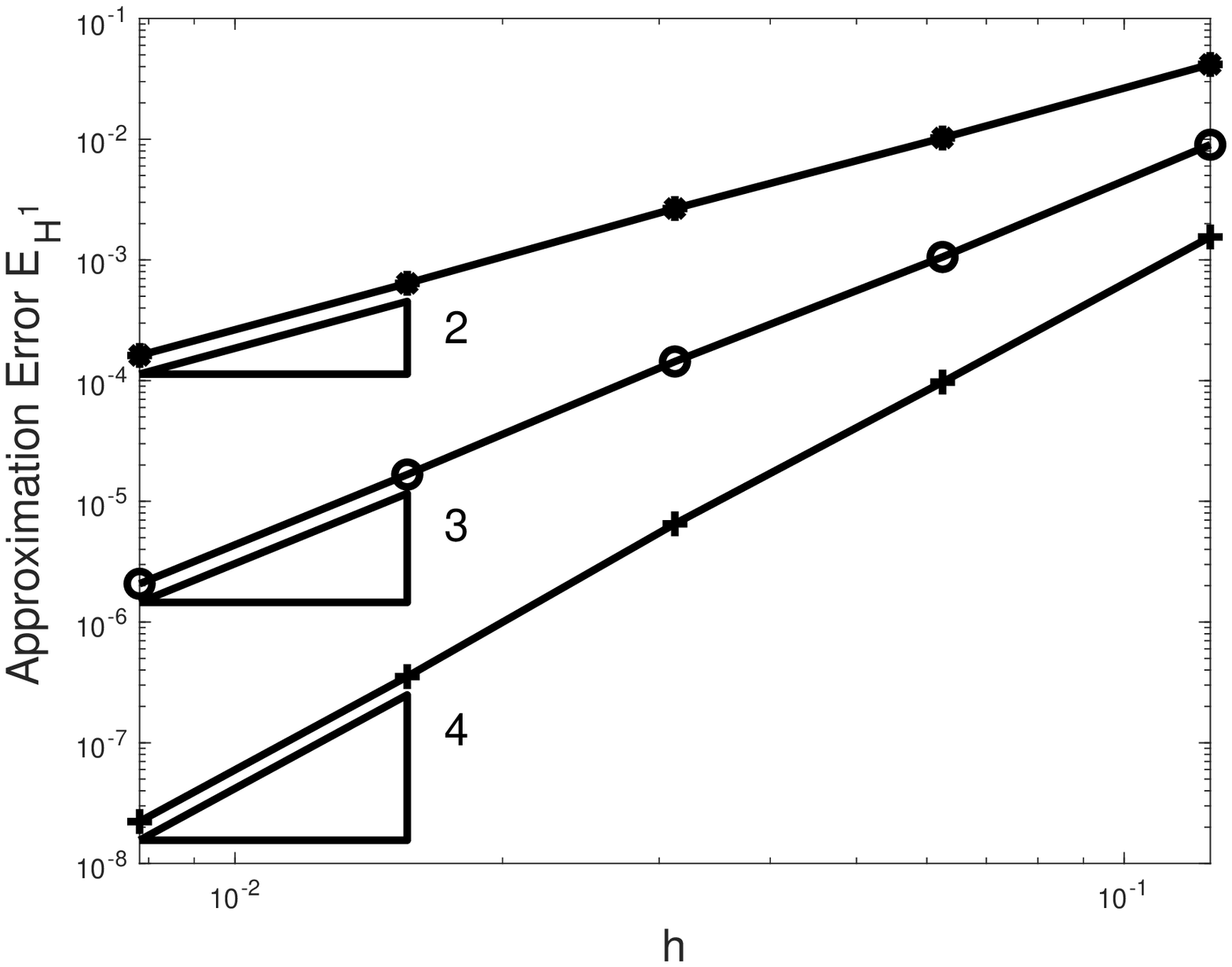} 
\includegraphics[width=2.5in]{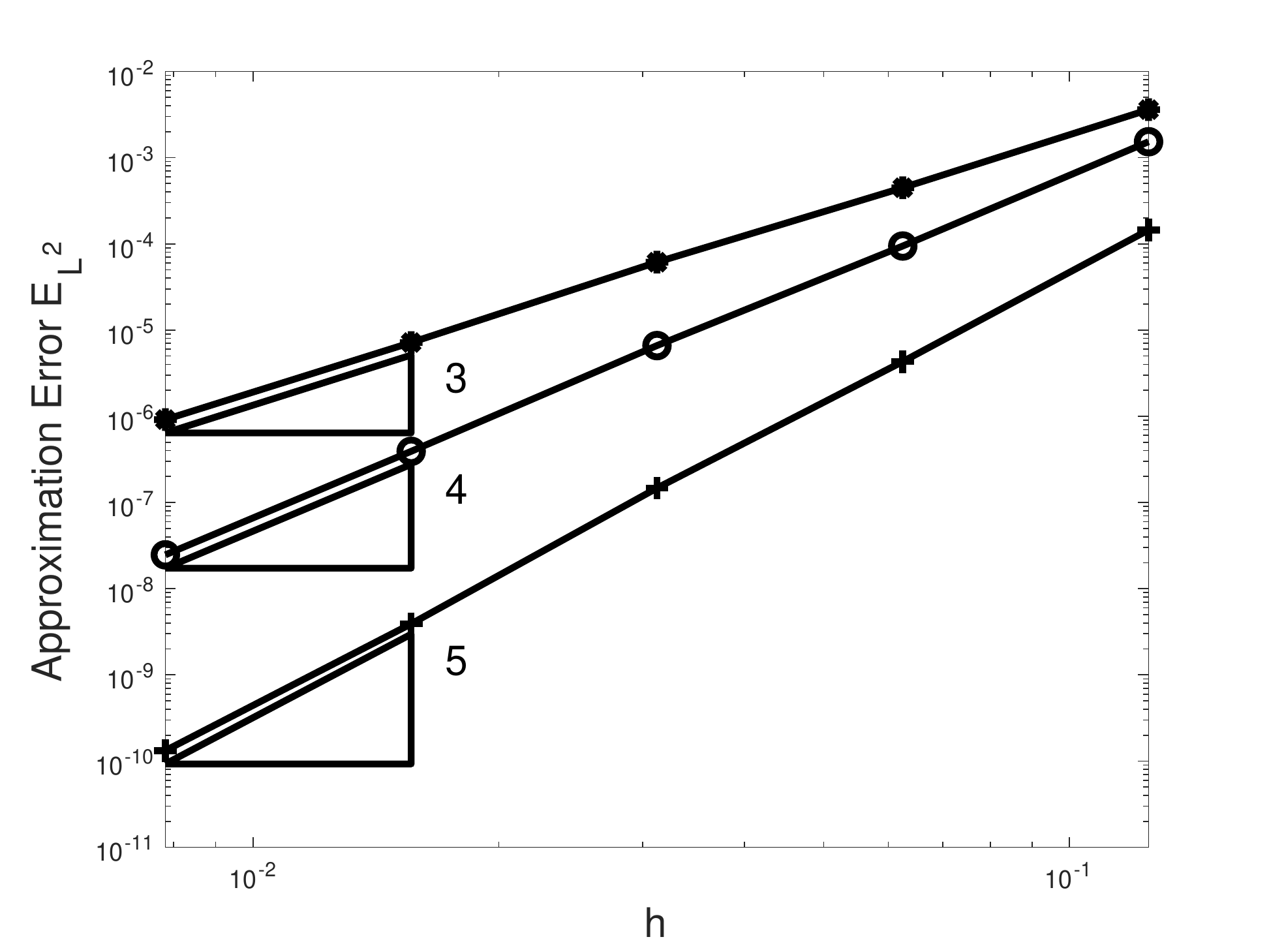}
\includegraphics[width=2.5in]{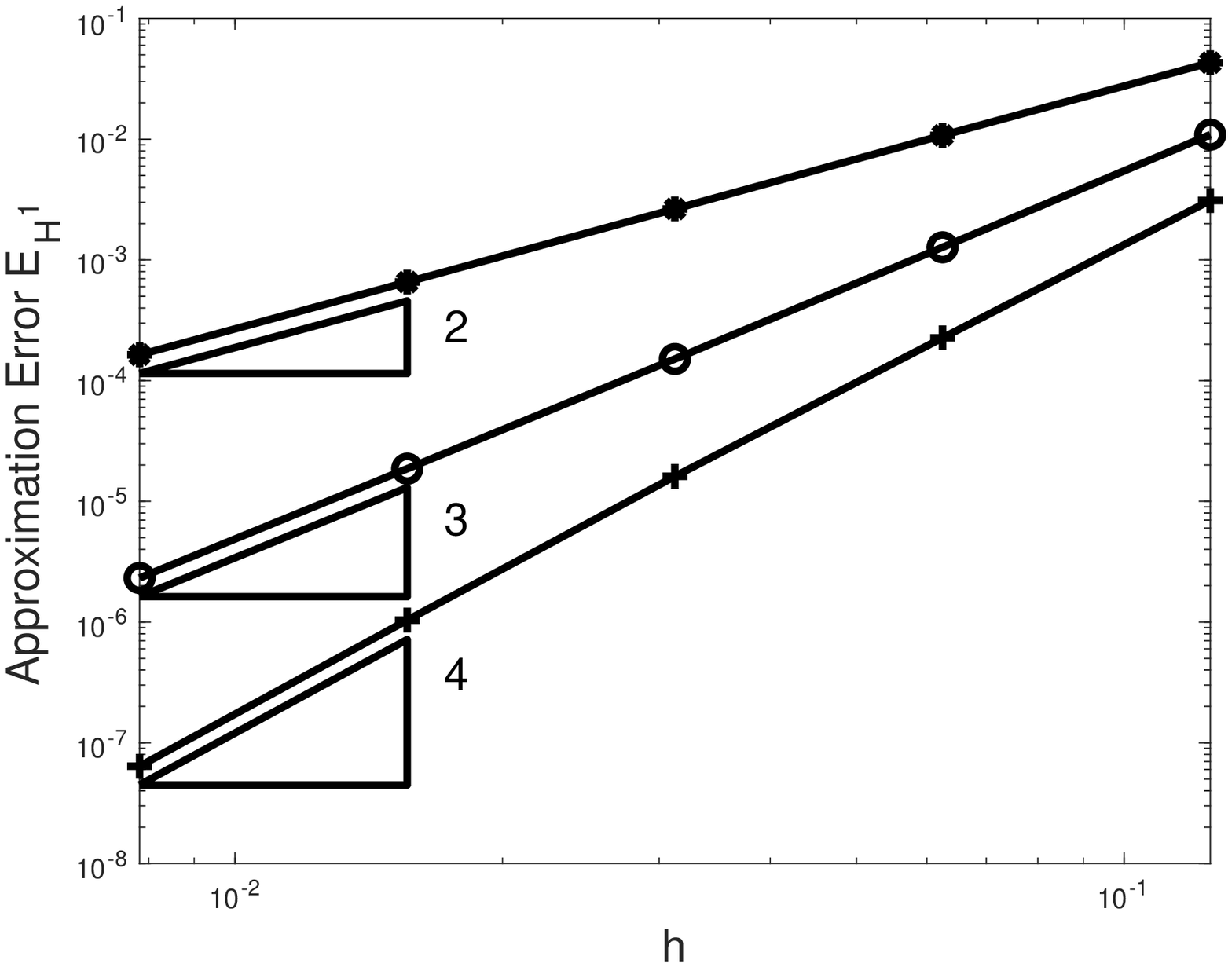} 
\includegraphics[width=2.5in]{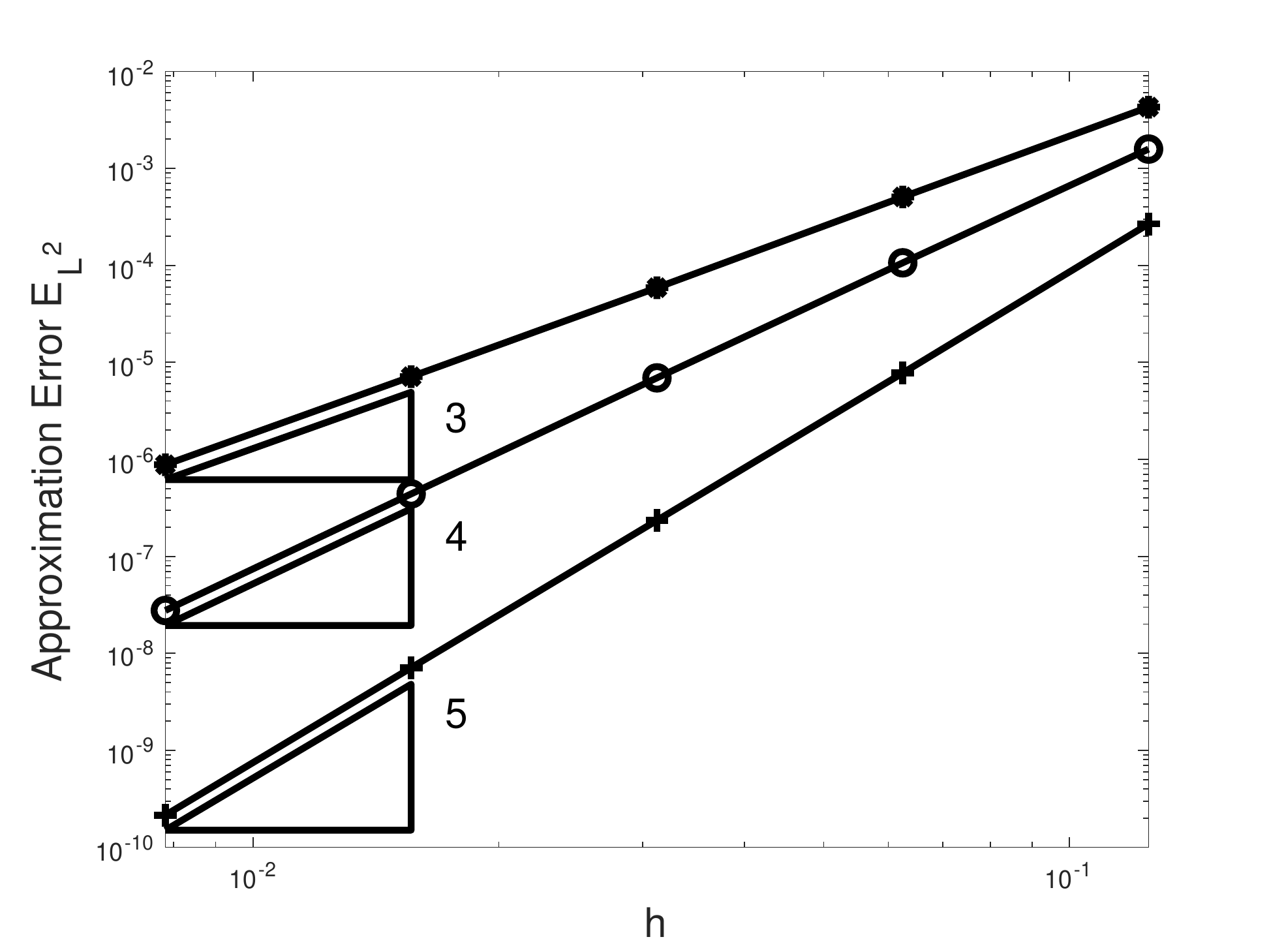}
\caption{\small{Left-panels: convergence of~$E_{H^1}$.
Right-panels: convergence of~$E_{L^2}$.
The exact solution is~$u_2$.
We employ sequences of (curved)} Voronoi meshes (first row)
and (curved) square meshes (second row)
with decreasing mesh size are employed.
The ``orders'' of the virtual element spaces are
$k=2$, $3$, and~$4$.}
\label{figure:convergence-test-case-2}
\end{center}
\end{figure}

The theoretical predictions of Section \ref{section:convergence} are confirmed:
convergence of order~$\mathcal O(h^k)$
and~$\mathcal O(h^{k+1})$ is observed for the energy and $L^2$-type errors in~\eqref{computable-quantities}.

Next, we approximate the curved domain by using a polygonal mesh sequence
of elements with straight edges.
Notably, we approximate
the curved boundary by straight segments and
force homogeneous Dirichlet boundary conditions;
see Figure~\ref{figure:straightmesh-test-case-2}.

In Figure \ref{figure:weak-convergence-test-case-2} we plot the results for the sequence of Voronoi meshes on the approximated domain, obtained with the standard nonconforming VEM on polygons.

\begin{figure}[H]
\begin{center}
\includegraphics[width=2.5in]{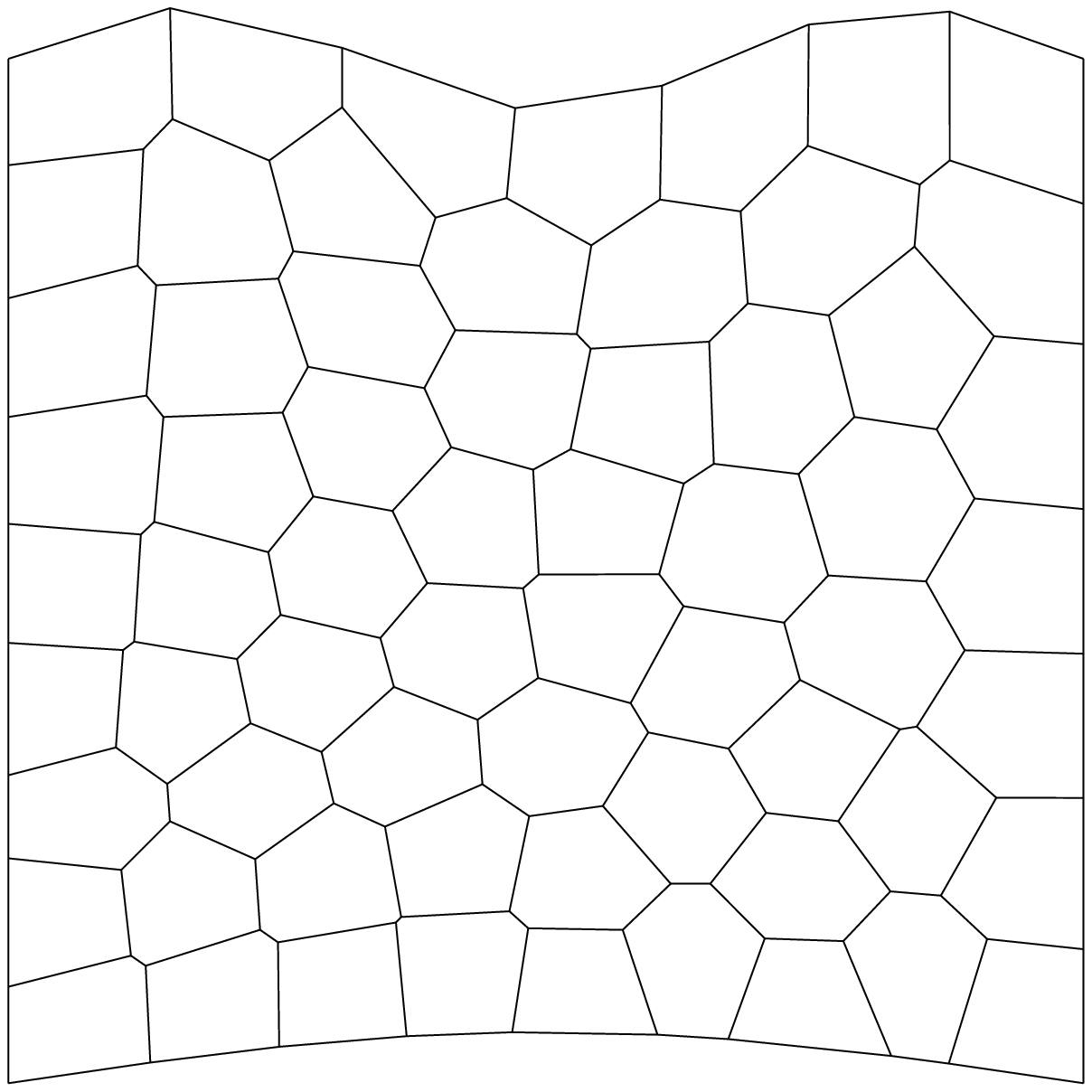} 
\caption{\small{An example of (straight) Voronoi mesh over~$\Omega$.}}
\label{figure:straightmesh-test-case-2}
\end{center}
\end{figure}
\begin{figure}[H]
\begin{center}
\includegraphics[width=2.5in]{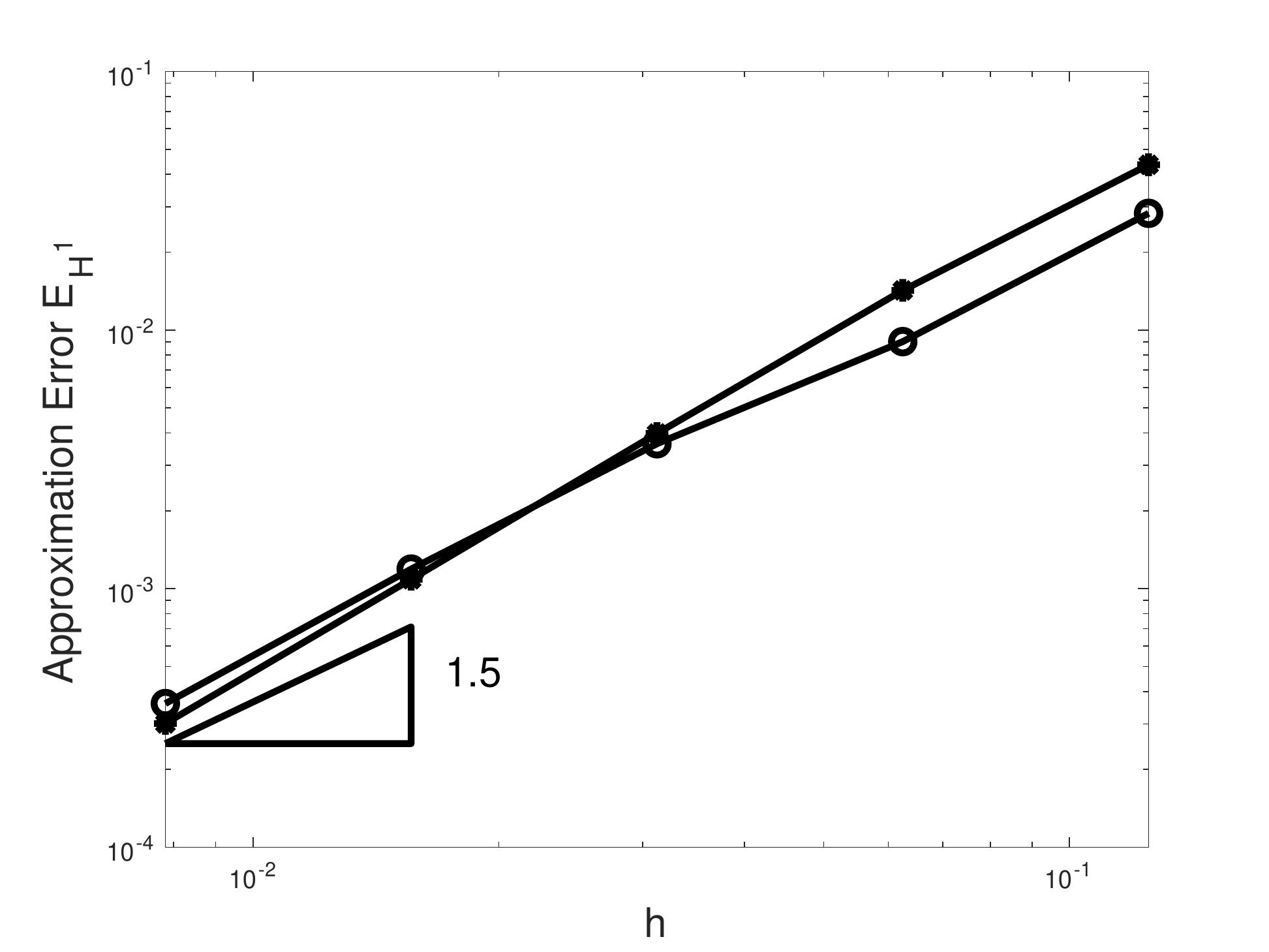} 
\includegraphics[width=2.5in]{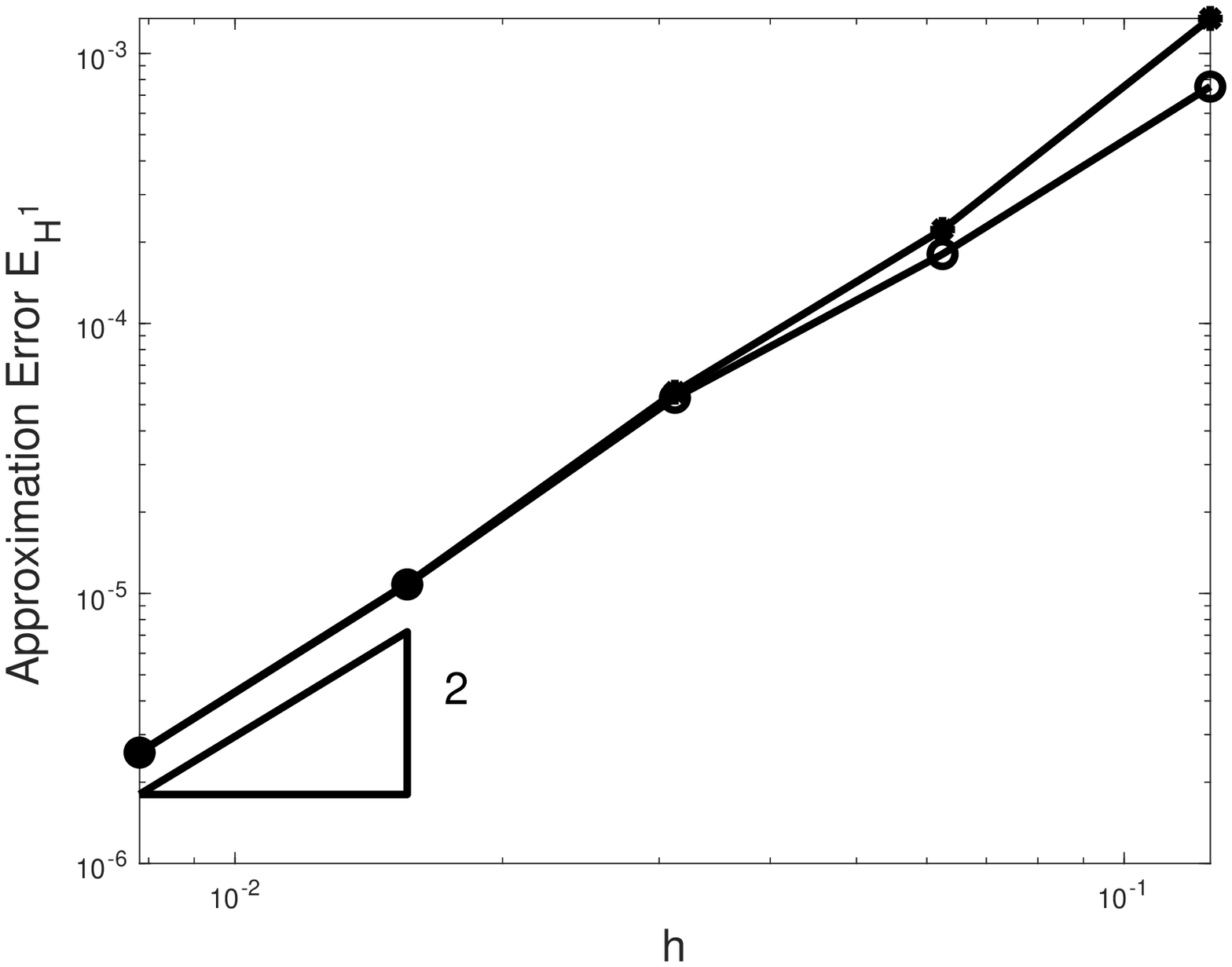}
\caption{\small{Left-panel: convergence of~$E_{H^1}$.
Right-panel: convergence of~$E_{L^2}$.
The exact solution is~$u_2$.
Voronoi meshes with decreasing mesh size are employed.
The ``orders'' of the virtual element spaces are
$k=2$, $3$.}}
\label{figure:weak-convergence-test-case-2}
\end{center}
\end{figure}

In this case, we observe that the geometrical error dominates
and the rate of convergence is approximately
1.5 and 2 for the energy and $L^2$ type
approximate errors in~\eqref{computable-quantities}.

\subsection{Curved interfaces}
\label{subsection:curved-interface}
As a third test case, see~\cite{BeiraodaVeiga-Russo-Vacca:2019},
we consider the same circular domain~$\Omega$ with boundary~$\Gamma_2$ as in Section~\ref{subsection:curved-boundary},
and we split into two subdomains~$\Omega_1$ and~$\Omega_2$
by an internal interface~$\Gamma_1$
so as~$\Omega_2$ has half the radius of~$\Omega$;
see Figure~\ref{figure:test-case-3} (left-panel).
Further, we are given a diffusion coefficient~$\kappa$
and a loading term~$f$ piecewise defined as
\[
\begin{cases}
\kappa = 1 & \text{in } \Omega_1,\\
\kappa = 5 & \text{in } \Omega_2,
\end{cases}
\qquad \qquad
\begin{cases}
f = 5  & \text{in } \Omega_1,\\
f = 1  & \text{in } \Omega_2,
\end{cases}
\]
We are interested in approximating the solution $u_2$ to the elliptic problem
\[
\begin{cases}
-{\mathrm{div}} (\kappa \nabla u) =f  & \text{in } \Omega,\\
u =0                                  & \text{on } \Gamma_2,
\end{cases}
\]
which is given by (here $r:= \sqrt{x^2+y^2}$)
\[
u_3 (x,y) =
u_3 (r) :=
\begin{cases}
-\frac{5}{4}r^2 + \frac{7}{20} + \frac{\ln(2)}{10}
& \text{if }  r \le 1/2,\\
-\frac{1}{20}r^2 - \frac{1}{10}\ln(r) + \frac{1}{20}
& \text{if } 1/2<r<1.
\end{cases}
\]
The function~$u_3$ is analytic
in~$\Omega_1$ and~$\Omega_2$
but has finite Sobolev regularity across the interface~$\Gamma_1$,
see Figure \ref{figure:test-case-3}).

\begin{figure}[H]
\begin{center}
\includegraphics[width=2.5in]{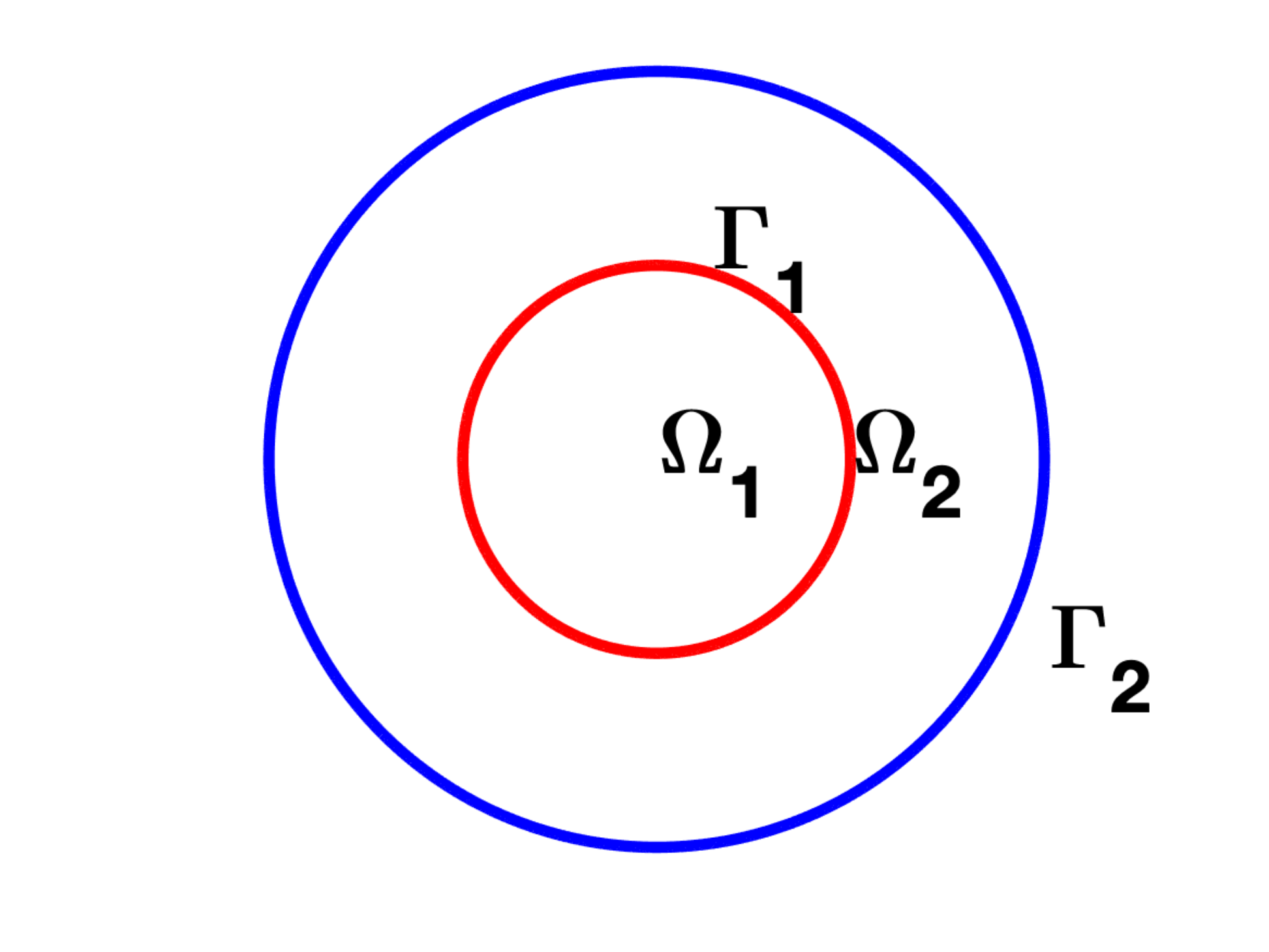} 
\includegraphics[width=2.5in]{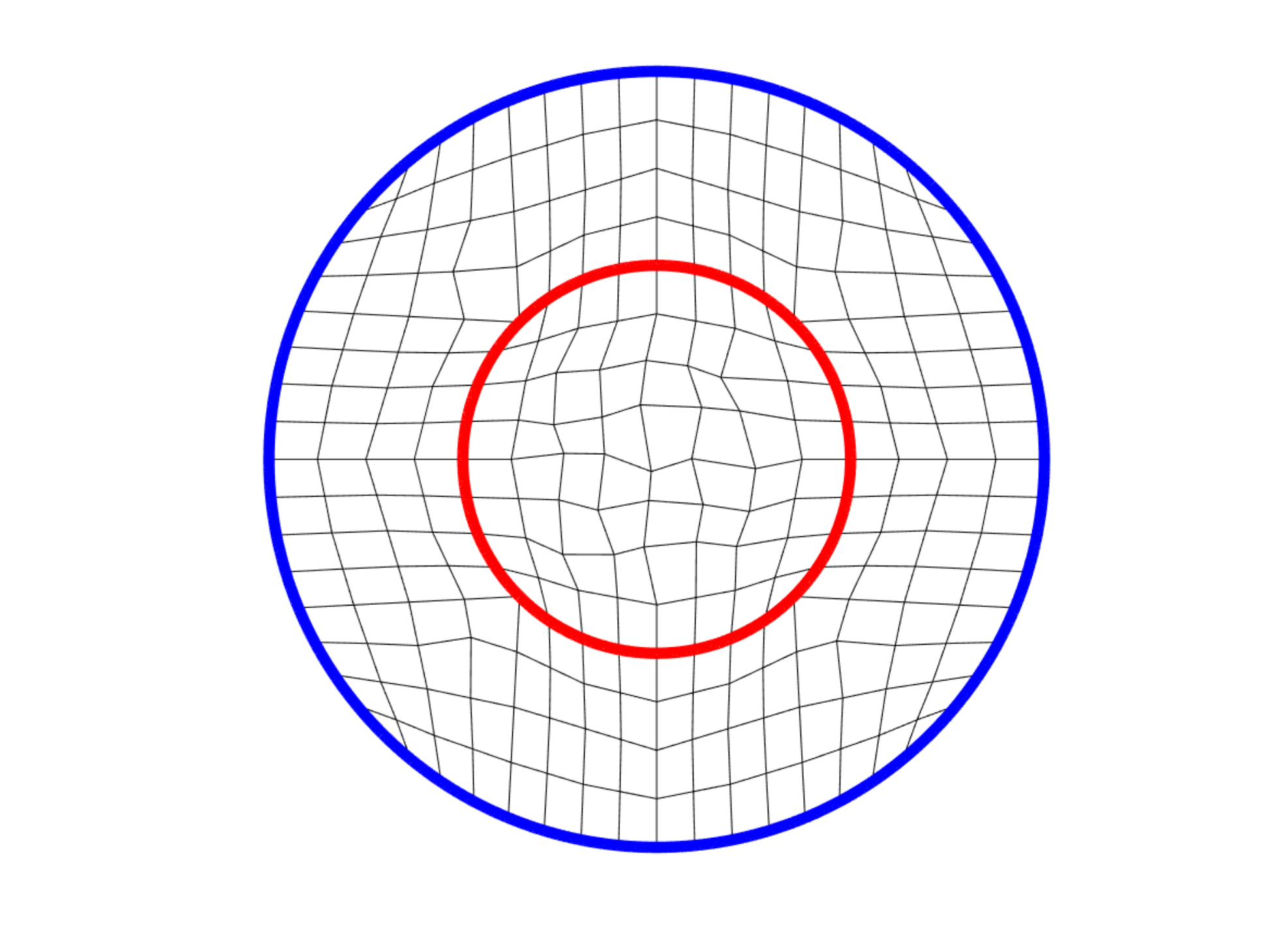}
\caption{\small{
Left-panel: a circular domain~$\Omega$ with a curved internal interface~$\Gamma_1$.
Right-panel: an example of (curved) mesh conforming with respect to the internal interface.}}
\label{figure:test-case-3}
\end{center}
\end{figure}

In Figure~\ref{figure:convergence-test-case-3},
we show the convergence of the two quantities in~\eqref{computable-quantities} on the given sequence of meshes with decreasing mesh size
that are conforming with respect to the curved internal interface~$\Gamma_1$;
see Figure~\ref{figure:test-case-3} (right-panel).
We consider virtual elements of ``order'' $k=2$, $3$, and~$4$.

\begin{figure}[H]
\begin{center}
\includegraphics[width=2.5in]{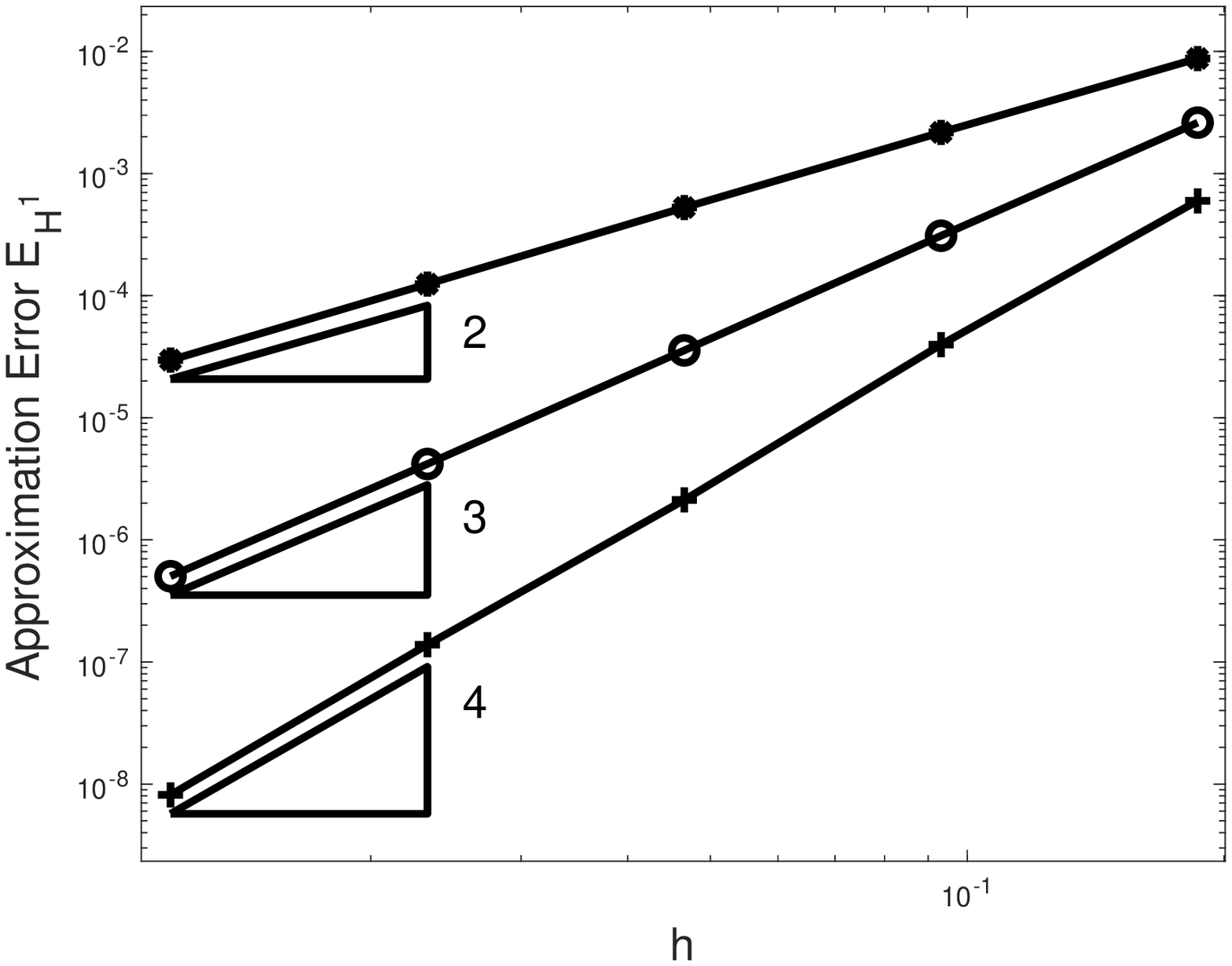}
\includegraphics[width=2.5in]{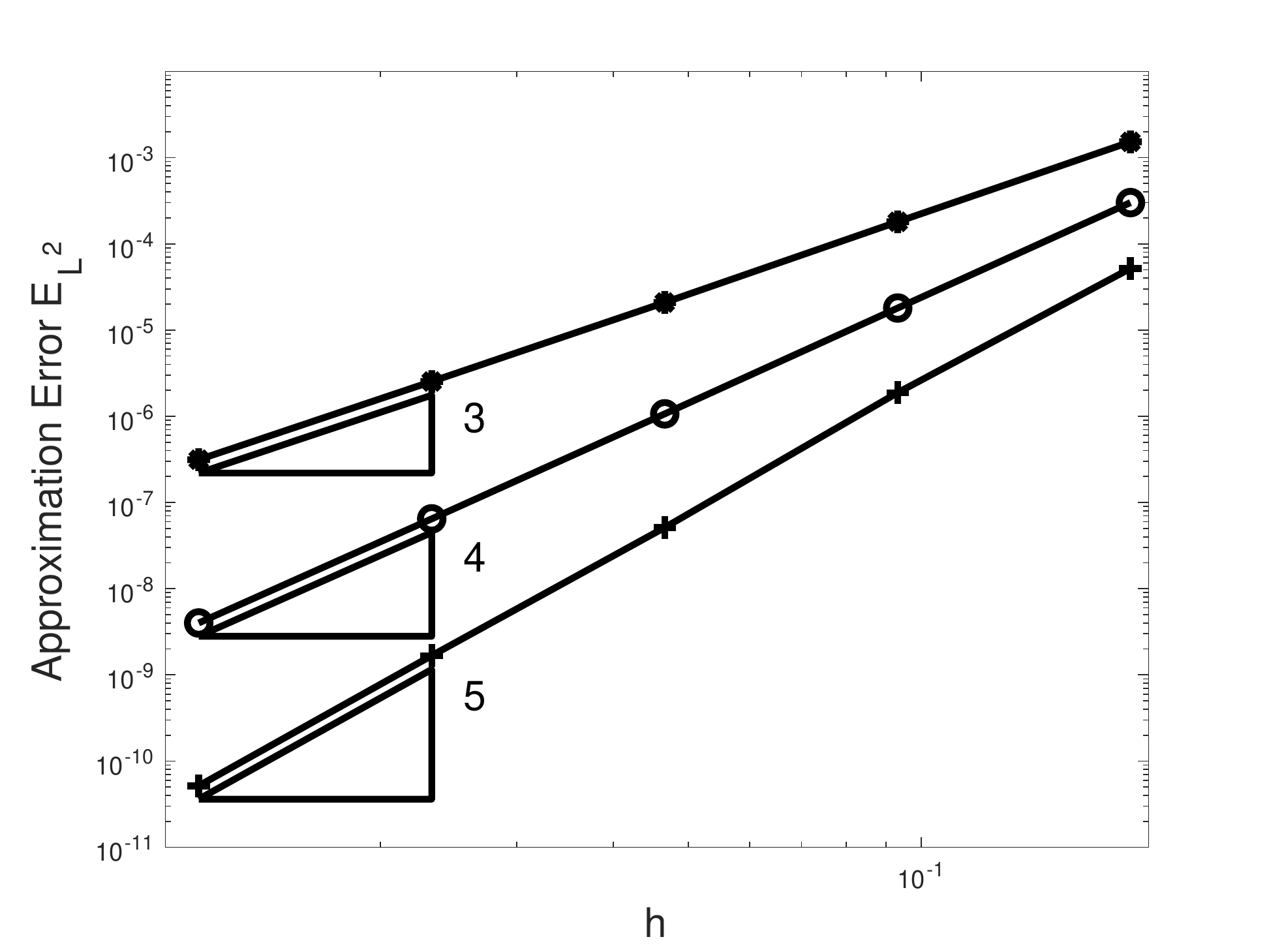}
\caption{\small{Left-panel: convergence of~$E_{H^1}$.
Right-panel: convergence of~$E_{L^2}$.
The exact solution is~$u_3$.
Meshes with decreasing mesh size are employed
that are conforming with respect to the internal interface~$\Gamma_1$.
The ``orders'' of the virtual element spaces are
$k=2$, $3$, and~$4$.}}
\label{figure:convergence-test-case-3}
\end{center}
\end{figure}

The theoretical predictions of Section~\ref{section:convergence} are confirmed
also for domains with internal curved interface:
convergence of order~$\mathcal O(h^k)$
and~$\mathcal O(h^{k+1})$ is observed for the energy and $L^2$-type errors in~\eqref{computable-quantities}.


\paragraph*{Acknowledgment}
Y.L. is supported by the NSFC grant 12171244 and China Scholarship Council 202206860034.
L. Beir\~ao da Veiga was partially supported by the Italian MIUR through the PRIN Grant No. 905 201744KLJL.

\bibliographystyle{plain}
{\footnotesize\bibliography{nc_vem.bib}}

\begin{thebibliography}{10}

\bibitem{Ahmad-Alsaedi-Brezzi-Marini-Russo:2013}
B.~Ahmad, A.~Alsaedi, F.~Brezzi, L.~D. Marini, and A.~Russo.
\newblock Equivalent projectors for virtual element methods.
\newblock {\em Comp. Math. Appl.}, 66(3):376--391, 2013.

\bibitem{Anand-Ovall-Reynolds-Weisser:2020}
A.~Anand, J.~S. Ovall, S.~E. Reynolds, and S.~Wei{\ss}er.
\newblock Trefftz finite elements on curvilinear polygons.
\newblock {\em SIAM J. Sci. Comput.}, 42(2):A1289--A1316, 2020.

\bibitem{AyusodeDios-Lipnikov-Manzini:2016}
B.~Ayuso~de Dios, K.~Lipnikov, and G.~Manzini.
\newblock The nonconforming virtual element method.
\newblock {\em ESAIM Math. Model. Numer. Anal.}, 50(3):879--904, 2016.

\bibitem{BeiraodaVeiga-Mascotto:2022}
L.~Beir\~ao~da Veiga and L.~Mascotto.
\newblock Interpolation and stability properties of low order face and edge
  virtual element spaces.
\newblock {\em IMA J. Numer. Anal.}, 2022.
\newblock \url{https://doi.org/10.1093/imanum/drac008}.

\bibitem{BeiradaVeiga-Brezzi-Cangiani-Manzini-Marini-Russo:2013}
L.~Beir{\~a}o~da Veiga, F.~Brezzi, A.~Cangiani, G.~Manzini, L.~D. Marini, and
  A.~Russo.
\newblock Basic principles of virtual element methods.
\newblock {\em Math. Models Methods Appl. Sci.}, 23(01):199--214, 2013.

\bibitem{BeiraodaVeiga-Brezzi-Marini:2013}
L.~Beir{\~a}o~da Veiga, F.~Brezzi, and L.~D. Marini.
\newblock Virtual elements for linear elasticity problems.
\newblock {\em SIAM J. Numer. Anal.}, 51(2):794--812, 2013.

\bibitem{BeiradaVeiga-Brezzi-Marini-Russo:2014}
L.~Beir{\~a}o~da Veiga, F.~Brezzi, L.~D. Marini, and A.~Russo.
\newblock The hitchhiker's guide to the virtual element method.
\newblock {\em Math. Models Methods Appl. Sci.}, 24(08):1541--1573, 2014.

\bibitem{BeiraodaVeiga-Brezzi-Marini-Russo:2020}
L.~Beir{\~a}o~da Veiga, F.~Brezzi, L.~D. Marini, and A.~Russo.
\newblock Polynomial preserving virtual elements with curved edges.
\newblock {\em Mathematical Models and Methods in Applied Sciences},
  30(08):1555--1590, 2020.

\bibitem{BeiraodaVeiga-Lovadina-Russo:2017}
L.~Beir{\~a}o~da Veiga, C.~Lovadina, and A.~Russo.
\newblock Stability analysis for the virtual element method.
\newblock {\em Math. Models Methods Appl. Sci.}, 27(13):2557--2594, 2017.

\bibitem{BeiraodaVeiga-Russo-Vacca:2019}
L.~Beir{\~a}o~da Veiga, A.~Russo, and G.~Vacca.
\newblock The virtual element method with curved edges.
\newblock {\em ESAIM Math. Model. Numer. Anal.}, 53(2):375--404, 2019.

\bibitem{Bertoluzza-Pennacchio-Prada:2019}
S.~Bertoluzza, M.~Pennacchio, and D.~Prada.
\newblock High order {VEM} on curved domains.
\newblock {\em Atti Accad. Naz. Lincei Rend. Lincei Mat. Appl.},
  30(2):391--412, 2019.

\bibitem{Botti-DiPietro:2018}
L.~Botti and D.A. Di~Pietro.
\newblock Assessment of hybrid high-order methods on curved meshes and
  comparison with discontinuous {G}alerkin methods.
\newblock {\em J. Comput. Phys.}, 370:58--84, 2018.

\bibitem{Bramble-Dupont-Thomee:1972}
J.~H. Bramble, T.~Dupont, and V.~Thom{\'e}e.
\newblock Projection methods for {D}irichlet’s problem in approximating
  polygonal domains with boundary-value corrections.
\newblock {\em Math. Comp.}, 26(120):869--879, 1972.

\bibitem{Brenner:2003}
S.~C. Brenner.
\newblock Poincar{\'e}--{F}riedrichs inequalities for piecewise ${H}^1$
  functions.
\newblock {\em SIAM J. Numer. Anal.}, 41(1):306--324, 2003.

\bibitem{Brenner-Scott:2008}
S.~C. Brenner and L.R. Scott.
\newblock {\em The mathematical theory of finite element methods}, volume~3.
\newblock Springer, 2008.

\bibitem{Burman-Cicuttin-Delay-Ern:2021}
E.~Burman, M.~Cicuttin, G.~Delay, and A.~Ern.
\newblock An unfitted hybrid high-order method with cell agglomeration for
  elliptic interface problems.
\newblock {\em SIAM J. Sci. Comput.}, 43(2):A859--A882, 2021.

\bibitem{Burman-Ern:2019}
E.~Burman and A.~Ern.
\newblock A cut cell hybrid high-order method for elliptic problems with curved
  boundaries.
\newblock In {\em European Conference on Numerical Mathematics and Advanced
  Applications}, pages 173--181. Springer, 2019.

\bibitem{Burman-Hansbo-Larson:2018}
E.~Burman, P.~Hansbo, and M.~Larson.
\newblock A cut finite element method with boundary value correction.
\newblock {\em Math. Comp.}, 87(310):633--657, 2018.

\bibitem{Chen-Huang:2018}
L.~Chen and J.~Huang.
\newblock Some error analysis on virtual element methods.
\newblock {\em Calcolo}, 55(5):1--23, 2018.

\bibitem{Cockburn-DiPietro-Ern:2016}
B.~Cockburn, D.~A. Di~Pietro, and A.~Ern.
\newblock Bridging the hybrid high-order and hybridizable discontinuous
  {G}alerkin methods.
\newblock {\em ESAIM Math. Model. Numer. Anal.}, 50(3):635--650, 2016.

\bibitem{Cottrell-Hughes-Bazilevs:2009}
J.~A. Cottrell, T.~J.~R. Hughes, and Y.~Bazilevs.
\newblock {\em Isogeometric analysis: toward integration of {CAD} and {FEA}}.
\newblock John Wiley \& Sons, 2009.

\bibitem{Dassi-Fumagalli-Losapio-Scialo-Scotti-Vacca}
F.~Dassi, A.~Fumagalli, D.~Losapio, S.~Scial{\`o}, A.~Scotti, and G.~Vacca.
\newblock The mixed virtual element method on curved edges in two dimensions.
\newblock {\em Comput. Methods Appl. Mech. Engrg.}, 386:114098, 2021.

\bibitem{Dassi-Fumagalli-Mazzieri-Scotti-Vacca}
F.~Dassi, A.~Fumagalli, I.~Mazzieri, A.~Scotti, and G.~Vacca.
\newblock A virtual element method for the wave equation on curved edges in two
  dimensions.
\newblock {\em J. Sci. Comput.}, 90(1):1--25, 2022.

\bibitem{Dassi-Fumagalli-Scotti-Vacca}
F.~Dassi, A.~Fumagalli, A.~Scotti, and G.~Vacca.
\newblock Bend {3D} mixed virtual element method for {D}arcy problems.
\newblock {\em Comput. Math. Appl.}, 119:1--12, 2022.

\bibitem{Dong-Ern:2022}
Z.~Dong and A.~Ern.
\newblock Hybrid high-order and weak {G}alerkin methods for the biharmonic
  problem.
\newblock {\em SIAM J. Numer. Anal.}, 60(5):2626--2656, 2022.

\bibitem{Ergatoudis-Irons-Zienkiewicz:1968}
I.~Ergatoudis, B.~M. Irons, and O.~C. Zienkiewicz.
\newblock Curved, isoparametric, ``quadrilateral'' elements for finite element
  analysis.
\newblock {\em Int. J. Solids Struct.}, 4(1):31--42, 1968.

\bibitem{Gurkan-SalaLardies-Kronbichler-FernandezMandez:2016}
C.~G{\"u}rkan, E.~Sala-Lardies, M.~Kronbichler, and
  S.~Fern{\'a}ndez-M{\'e}ndez.
\newblock e{X}tended {H}ybridizable {D}iscontinous {G}alerkin ({X-HDG}) for
  void problems.
\newblock {\em J. Sci. Comput.}, 66(3):1313--1333, 2016.

\bibitem{Lenoir:1986}
M.~Lenoir.
\newblock Optimal isoparametric finite elements and error estimates for domains
  involving curved boundaries.
\newblock {\em SIAM J. Numer. Anal.}, 23(3):562--580, 1986.

\bibitem{Mascotto-Perugia-Pichler:2018}
L.~Mascotto, I.~Perugia, and A.~Pichler.
\newblock Non-conforming harmonic virtual element method: $h$- and
  $p$-versions.
\newblock {\em J. Sci. Comput.}, 77(3):1874--1908, 2018.

\bibitem{Schwab:1998}
C.~Schwab.
\newblock {\em $p$- and $hp$- {F}inite {E}lement {M}ethods: {T}heory and
  {A}pplications in {S}olid and {F}luid {M}echanics}.
\newblock Clarendon Press Oxford, 1998.

\bibitem{Stein:1970}
E.~M. Stein.
\newblock {\em Singular integrals and differentiability properties of
  functions}, volume~2.
\newblock Princeton University Press, 1970.

\bibitem{Strang-Berger:1973}
G.~Strang and A.~E. Berger.
\newblock The change in solution due to change in domain.
\newblock In {\em Partial differential equations ({P}roc. {S}ympos. {P}ure
  {M}ath., {V}ol. {XXIII}, {U}niv. {C}alifornia, {B}erkeley, {C}alif., 1971)},
  pages 199--205. Amer. Math. Soc., Providence, R.I., 1973.

\bibitem{Thomee:1973}
V.~Thom{\'e}e.
\newblock Polygonal domain approximation in {D}irichlet's problem.
\newblock {\em IMA J. Appl. Math.}, 11(1):33--44, 1973.

\bibitem{Yemm:2022}
L.~Yemm.
\newblock A new approach to handle curved meshes in the hybrid high-order
  method.
\newblock \url{https://arxiv.org/abs/2212.05474}, 2023.

\end{thebibliography}
\end{document}